\documentclass[en,oneside,bbfont,article]{amsart} 
\usepackage[lmargin = 28mm, rmargin = 28mm, bmargin = 25mm, tmargin=25mm]{geometry}
\usepackage[utf8]{inputenc}

\usepackage{mathtools}
\DeclarePairedDelimiter\ceil{\lceil}{\rceil}

\providecommand{\examplename}{Example}

\newtheorem{theorem}{Theorem}[section]
\newtheorem{proposition}[theorem]{Proposition}
\newtheorem{corollary}[theorem]{Corollary}
\newtheorem{lemma}[theorem]{Lemma}
\theoremstyle{remark}
\newtheorem{remarkx}[theorem]{Remark}
\newenvironment{remark}
    {\pushQED{\qed}\remarkx}
    {\popQED\endremarkx}
\theoremstyle{definition}
\newtheorem*{example*}{\protect\examplename}
\newtheorem{examplex}[theorem]{\protect\examplename}
\newenvironment{example}
    {\pushQED{\qed}\examplex}
    {\popQED\endexamplex}
\theoremstyle{plain}

\newtheorem*{assumption*}{Assumption}
\newtheorem*{defin}{Definition}

\usepackage{mathrsfs,array}
\usepackage{amssymb}
\usepackage{amstext}
\usepackage{dsfont}
\usepackage{amsthm}
\usepackage{amsmath}
\usepackage{bm}
\usepackage{caption,subcaption,enumerate,enumitem}
\usepackage{scalefnt}
\usepackage{float,graphicx,verbatim}
\usepackage{tikz}
\usetikzlibrary{calc}
\usepackage{pgfplots}
\pgfplotsset{compat=1.16}
\RequirePackage[colorlinks=true,linkcolor=red,
	citecolor=blue,urlcolor=blue]{hyperref}

\usepackage[sort, numbers]{natbib}

\newcommand\N{\mathbb{N}}

\newcommand{\mW}{\mathcal{W}}
\newcommand\R{\mathbb{R}}
\newcommand\Sp{\mathbb{S}}
\newcommand\C{\mathbb{C}}
\newcommand\E{\mathds{E}}
\newcommand\e{\mathord{\mathrm{e}}}

\newcommand\1{\mathds{1}}
\newcommand\Oh{\mathcal{O}}
\newcommand\oh{\mathrm{o}}
\newcommand\ka{{\scalebox{.5}{(}\kappa\scalebox{.5}{)}}}

\newcommand{\Leb}{\text{\normalfont Leb}}

\newcommand\da{\downarrow}
\newcommand\la{\leftarrow}
\newcommand\scrL{\mathscr{L}}
\newcommand{\ve}{\varepsilon}
\renewcommand{\le}{\leqslant}
\renewcommand{\ge}{\geqslant}

\newcommand\eqd{\overset{d}{=}}
\newcommand\cid{\xrightarrow{d}}

\newcommand\SV{\mbox{SV}}
\newcommand\co{\mathsf{c}}
\newcommand\nf[1]{\normalfont{#1}}

\newcommand{\D}{\mathrm{d}}

\newcommand{\wt}[1]{\widetilde{#1}}

\newcommand{\tra}{{\scalebox{0.6}{$\top$}}}

\newcolumntype{C}{>{\centering\arraybackslash}m{3cm}}
\newcolumntype{T}{>{\centering\arraybackslash}m{5.5cm}}
\newcolumntype{L}{>{\centering\arraybackslash}m{8.5cm}}

\newcommand{\genstirlingII}[3]{%
  \genfrac{\{}{\}}{0pt}{#1}{#2}{#3}%
}
\newcommand{\stirII}[2]{\genstirlingII{}{#1}{#2}}

\usepackage[foot]{amsaddr}

\title[Coupling and Wasserstein rates in the small time stable domain of attraction]{Asymptotically optimal Wasserstein couplings for the small-time stable domain of attraction}
\author{Jorge Gonz\'alez C\'azares$^{*}$, David Kramer-Bang$^\dag$ \& Aleksandar Mijatovi\'c$^{\ddag}$}

\address{$^*$IIMAS--UNAM, Mexico.}

\address{$^\dag$Department of Mathematics, Aarhus University, Denmark.}

\address{$^\ddag$Department of Statistics, University of Warwick \& The Alan Turing Institute, UK.}

\email{$^*$jorge.gonzalez@sigma.iimas.unam.mx}
\email{$^\dag$bang@math.au.dk}
\email{$^\ddag$a.mijatovic@warwick.ac.uk}


\begin{document}

\begin{abstract}
We develop two novel couplings between general pure-jump L\'evy processes in $\R^d$ and apply them to obtain upper bounds on the rate of convergence in a suitable Wasserstein distance on path space for a wide class of L\'evy processes attracted to a multidimensional stable process in the small-time regime. We also establish general lower bounds based on universal properties of slowly varying functions and the relationship between the Wasserstein and Toscani-Fourier distances. Our upper and lower bounds typically exhibit matching rates. In particular, the rate of convergence is polynomial for the domain of normal attraction and slower than a slowly varying function for the domain of non-normal attraction.
\end{abstract}

\maketitle

\section{Introduction}

Stable processes arise naturally as universal scaling limits of a vast class of stochastic processes at either small or large times. In particular, in the small-time regime, stable processes arise as weak limits of discretisation errors of widely used models in theoretical and applied probability~\cite{MR3833470,MR4244191,MR3803961}. Most of these models are based on L\'evy processes in the small-time domain of attraction of  stable processes~\cite{MR3784492,MR4161123,MR2792485}. In contrast with the more classical long-time regime of Lamperti, where the literature is abundant (see, e.g.~\cite{MR629531,MR1106283,MR2172843,MR3806899,manou-abi}), the study of the convergence in the small-time regime, which is the focus of this paper, has been underdeveloped. In the long-time regime, the convergence is a consequence of heavy tails with regularly varying tail probabilities or a finite second moment. On the other hand, in the small-time regime, the convergence depends on the activity of the small jumps of the underlying L\'evy process and does not depend on the behaviour of the tail probabilities~\cite{MR3784492}. However, having a heavy-tailed limit may severely deteriorate the convergence speed as uniform integrability typically fails. Quantifying such an error is a fundamental problem, crucial to numerous disparate application areas, such as controlling the bias of discretised models in mathematical finance and elsewhere (see~\cite{MR3784492} and the references therein), quantifying the model misspecification risk~\cite{MR3959085} or  asserting the convergence properties of estimators for the index of variation, such as Hill's estimator, which is known to require a second-order condition for the convergence to have good properties~\cite[p.~193--195]{MR1458613}. 

The main aim of the present paper is to establish lower and upper bounds in Wasserstein distance on the convergence rate of multivariate L\'evy processes attracted to a stable process in the domains of both normal and non-normal attraction (see definition in Section~\ref{sec:main_results} below). Moreover, we will show that our bounds are often sharp. Our upper bounds are applicable to a large class of L\'evy processes that are attracted to a multivariate $\alpha$-stable process (which is Gaussian if $\alpha=2$, and heavy-tailed if $\alpha\in(0,2)$), while our lower bounds are universal within the small-time regime. To establish the upper bounds on the path supremum norm, we construct two couplings between any two arbitrary L\'evy processes, inspired by the stochastic representations in~\cite{MR2307403}, and bound the $L^p$-norm of the maximum distance between the paths of the resulting processes. The lower bounds  for the domain of normal attraction are obtained by comparing the Wasserstein distance with the Toscani--Fourier distance between the marginals and in the case of  non-normal attraction, using a universal property of slowly varying functions. 

We show that in the domain of normal attraction to heavy-tailed laws, under suitable second-order assumptions, the rates of convergence of the upper and lower bounds are polynomial and agree for the $L^q$-norm in the cases $q<\alpha<1$ and $q=1<\alpha$, making our couplings rate-optimal in this sense. In the domain of non-normal attraction (to either Gaussian or a heavy-tailed stable law), the upper and lower bounds are both `slow' and in particular, the convergence is never faster than $\log^{-1-\ve}(1/t)$ as $t\da 0$ for any $\ve>0$. Moreover, for large subclasses of L\'evy processes, the upper and lower bounds on the convergence rate agree in the case $q=1<\alpha$ (see e.g. Corollary~\ref{cor:upper_lower_general}). In the domain of normal attraction to the Gaussian law, our upper and lower bounds are also polynomial and dependent on the Blumenthal--Getoor index of the attracted process. The bounds on the convergence rates in this case often agree and, when they do not, the gap between them is small (see Figure~\ref{fig:normal_rates} below). A short YouTube presentation~\cite{YouTube_talk} describes our results, including the ideas behind the proofs.

\subsection{Summary of our results in the heavy-tailed stable domain of attraction}

To introduce the summary of our results in Table~\ref{tab:general_results}, we introduce some notation:  $f(t) \lesssim g(t)$ as $t\da 0$ holds for two functions $f,g\ge 0$ if there exists $c,t_0>0$ satisfying $f(t) \le c g(t)$ for all $t \in (0,t_0]$. A function $G$ is slowly varying at $a\in[0,\infty]$, denoted $G \in \SV_a$, if $G$ is positive on a neighbourhood of $a$ and $\lim_{x \to a} G(cx)/G(x)=1$ for all $c\in(0,\infty)$ (typically, $a\in\{0,\infty\}$).

Table~\ref{tab:general_results} summarises our results on the convergence rates established here for processes in both domains of attraction of stable processes. Recall that an $\alpha$-stable process has a finite $q$-moment if and only if $q<\alpha$. Due to this technical constraint, our upper bounds on the $L^q$-Wasserstein distance, defined in~\eqref{eq:def_wasserstein} below, always require $q<\alpha$ for the corresponding domains of attraction. 
We remark that, in this case, both lower and upper bounds are typically asymptotically equivalent up to a multiplicative constant, making our methods and couplings rate-optimal. Indeed, in the domain of normal attraction, this occurs for any admissible $q>0$ if $\alpha\in(0,1)$ and for the $L^1$-Wasserstein distance if $\alpha\in(1,2)$. Our bounds for the domain of non-normal attraction are also seen to be rate optimal when $G$ is sufficiently regular (see discussion following Theorem~\ref{thm:upper_lower_general} and Corollary~\ref{cor:upper_lower_general} below) when $q=1$ (and hence $\alpha>1$).

\begin{table}[ht]
\centering
\begin{tabular}{|c|c|} \hline
 Domain of attraction as $t\downarrow0$& $\alpha \in (0,2)\setminus\{1\}$ and $q \in (0,\alpha)\cap(0,1]$ 
 \\ \hline
normal  
& $t^{1-q/\alpha}\lesssim \mW_q(\bm{X}^t,\bm{Z})\lesssim t^{1-q/\alpha}\1_{\{\alpha<1\}}+t^{q(1-1/\alpha)}\1_{\{\alpha>1\}}$
\\ \hline
& $L(t)\lesssim \max\{\mW_q(\bm{X}^t,\bm{Z}),\mW_q(\bm{X}^{2t},\bm{Z})\}\lesssim L(t)^q$, where \\   
non-normal  & $L(t)\coloneqq t|G'(t)|/G(t)$ \textbf{cannot} be bounded above by any\\ 
& 
non-decreasing integrable function  $\ell>0$: if $\int_0^1 \ell(t)t^{-1}\D t<\infty$\\
& (e.g. $\ell(t)=|\log t|^{-1-\ve}$ for any $\ve>0$),\\
& then $\limsup_{t\da 0}L(t)/\ell(t)=\infty$ 
\\ \hline
\end{tabular}
\caption{Summary of the main results: Theorems~\ref{thm:upper_lower_simple} and~\ref{thm:upper_lower_general}, Corollary~\ref{cor:upper_lower_general} and Proposition~\ref{prop:normal_domain_comonotonic}.}
    \label{tab:general_results}
\end{table}

More precisely, in Table~\ref{tab:general_results} we let $\bm{X}=(\bm{X}_t)_{t \in [0,1]}$ be a L\'evy process in $\R^d$ attracted to an $\alpha$-stable process $\bm{Z}=(\bm{Z}_t)_{t \in [0,1]}$ with normalising function $g$. That is, $\bm{X}^t=(\bm{X}_s^t)_{s\in[0,1]}\coloneqq (\bm{X}_{st}/g(t))_{s\in[0,1]}$, $t\in (0,1]$, satisfies $\bm{X}^t_1=\bm{X}_t/g(t)\cid\bm{Z}_1$ as $t\da 0$. The table provides asymptotic bounds on the distance $\mW_q(\bm{X}^t,\bm{Z})$ as $t \da 0$ in both regimes of attraction. We let the assumptions of either Theorem~\ref{thm:upper_lower_simple} (with $p=1$, for the domain of normal attraction) or Theorem~\ref{thm:upper_lower_general} (for the domain of non-normal attraction) hold for $\alpha \in (0,2)\setminus\{1\}$ and pick $q \in(0,\alpha)\cap(0,1]$ satisfying $\E[|\bm{X}_1|^q]<\infty$. 
We stress that the lower bounds in Table~\ref{tab:general_results} in both domains of normal and non-normal attraction require no assumptions beyond the existence of the scaling limit (see  Theorem~(\nameref{thm:small_time_domain_stable}) below for the precise characterisation of the domain of attraction of a stable process $\bm{Z}$). In particular, as explained in the caption of Table~\ref{tab:general_results}, for \textit{any} $\bf{X}$ in the domain of non-normal attraction and arbitrary $\ve>0$, there exists a positive increasing sequence $(t_k)_{k\in\N}$ tending to infinity, such that the lower bound on the $L^q$-Wasserstein distance satisfies $L(t_k)\ge |\log t_k|^{-1-\ve}$ for all $k\in\N$. In Example~\ref{ex:arbitrarily_slow} below, we show that, even if the slowly varying function $G(t)\coloneqq g(t)t^{-1/\alpha}$ in the scaling limit grows (or decreases) arbitrarily slowly, the function $L$ may be asymptotically equivalent to it and  bounded below by $1/\log^{1+\ve}(1/t)$ \textit{for all} small $t>0$. 

Recall that the slowly varying function $G$, appearing in the scaling limit $\bm{X}_t/g(t)\cid\bm{Z}_1$ as $t\da 0$, is uniquely determined up to asymptotic equivalence only. Interestingly, our results imply that the rate of convergence in the Wasserstein distance can be affected by different choices of $G$, see Remark~\ref{rem:upper_lower_general}(IV)  below for more details. Finally, we note that the couplings yielding the upper bounds in Table~\ref{tab:general_results} require some structural assumptions on the L\'evy measure of $\bf{X}$ discussed in Sections~\ref{sec:main_results} and~\ref{sec:stable_limits_upper} below.

\subsection{A heuristic account of our couplings of L\'evy processes}
\label{subsec:heuristics}
One of the main purposes of this article is to introduce two couplings between two arbitrary multivariate L\'evy processes 
$\bm{X}$ and $\bm{Y}$ and analyse their properties. Both coupling constructions are centred around coupling the respective Poisson jump measures $\Xi_{\bm{X}}$ and $\Xi_{\bm{Y}}$. As the Brownian components of $\bm{X}$ and $\bm{Y}$  are coupled synchronously under both couplings, the couplings are named after the techniques involved in coupling of the Poisson jump measures $\Xi_{\bm{X}}$ and $\Xi_{\bm{Y}}$:  the first is the \emph{thinning coupling}, as it is based on Poisson thinning (see Section~\ref{sec:thinning}), and the second is the \emph{comonotonic coupling}, based on the minimal transport coupling of real-valued random variables and LePage's simulation method (see Section~\ref{sec:comonotonic_coupling}). As illustrated in Figure~\ref{fig:couplings}, the thinning coupling aims to maximise the intersection of the Poisson jump measures, whereas the comonotonic coupling aims to establish an optimal one-to-one correspondence between the atoms of the Poisson jump measures. 

\begin{figure}[ht]
    \centering
\begin{subfigure}[ht]{.49\linewidth}
    \centering\includegraphics[width=1\linewidth]{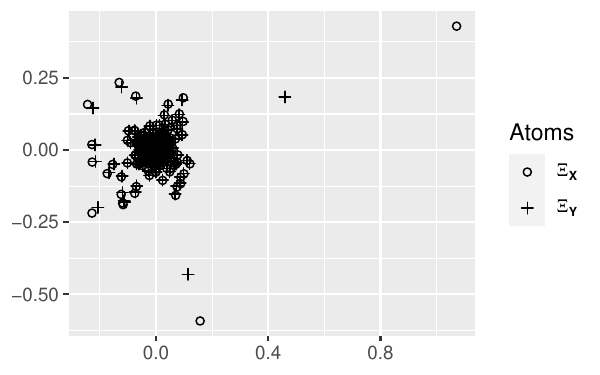}
    \caption{\small Comonotonic coupling}
  \end{subfigure}
  \begin{subfigure}[ht]{.49\linewidth}
    \centering\includegraphics[width=1\linewidth]{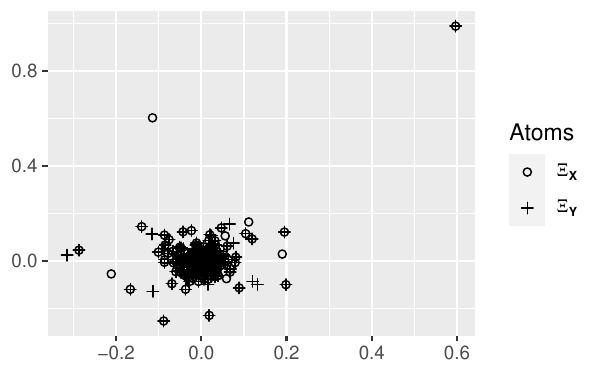}
    \caption{\small Thinning coupling}
  \end{subfigure}
    \caption{\small $\Xi_{\bm{X}}$ and $\Xi_{\bm{Y}}$ are Poisson random measures of jumps of the L\'evy processes $\bm{X}$ and $\bm{Y}$, respectively. Panel (A) depicts atoms of $\Xi_{\bm{X}}$ and $\Xi_{\bm{Y}}$ under the comonotonic coupling, where the angular component of each jump of $\bm{X}$ and $\bm{Y}$ coincides, while there magnitudes are coupled in the comonotonic fashion. In panel (B), the atoms of $\Xi_{\bm{X}}$ and $\Xi_{\bm{Y}}$ from the thinning coupling are sampled: either each jump of $\bm{X}$ coincides with a jump of $\bm{Y}$ or the two jumps are sampled independently.}
    \label{fig:couplings}
\end{figure}

The assumptions and constructions of both couplings are rather different. In the thinning coupling, we consider a common dominating L\'evy measure (say, the sum of both L\'evy measures) such that the L\'evy measures of both processes are absolutely continuous with respect to it with a bounded density. Then, we consider a Poisson measure with mean measure given by the dominating L\'evy measure and then thin the Poisson measure appropriately to produce coupled Poisson measures with mean measures given by the L\'evy measures of the processes. This maximises the common jumps of both processes. 

For the comonotonic coupling, we assume that both processes have a radial decomposition with their own angular measures. We then construct a radial decomposition for both with respect to a common angular measure. We use this measure and LePage's method to construct a one-to-one correspondence between jumps in which both processes jump in the same direction but with different magnitudes. Indeed, both Poisson measures are a transformation of a standard Poisson measure with independent decorations that select the direction of the jump. With our assumption, we construct such a transformation explicitly with the following properties. The decorations of both processes agree. Conditionally given a direction, the jumps of both processes, when ordered by decreasing magnitude, are in a one-to-one correspondence that mimics the relationship of real random variables under the minimal transport (or comonotonic) coupling. More precisely, the magnitudes are expressed as the right-continuous inverse of the radial tail L\'evy measure evaluated on the epochs of a standard Poisson process. 

\subsection{Comparison with the literature}

Few results identifying the small-time convergence rate exist in the multivariate setting even when the limit is Gaussian. Indeed, most results in this regime are restricted to dimension $1$ and often require finite jump activity~\cite{MR1805321,MR1482707,MR2867949,MR2851060}. In those situations, the limit law of the rescaled error can be identified for some functionals, leading to accurate estimates of the resulting bias and the celebrated continuity corrections~\cite{MR1482707,MR2851060}. 

For heavy-tailed stable limits (i.e. non-Gaussian) and infinite activity L\'evy processes attracted to the Gaussian law, again in one dimension, the literature is more scarce and only a fraction of the analogous results exist (see~\cite{MR4161123} for the convergence of certain path statistics to heavy-tailed limits). There are several complications in developing such results for small-time. First, the Berry--Esseen type bounds (see e.g.~\cite{MR1106283,MR2548505}), commonly used to establish convergence rates to the Gaussian law, fail to give convergent upper bounds since the jump intensity is vanishing in the small-time regime. Second, the rescaled variables often either fail to be uniformly integrable or their uniform integrability is difficult to prove (see, e.g.~\cite{MR4161123}). 

In~\cite{Ester2024}, the authors consider estimating the density of a discretely observed L\'evy process satisfying Orey's condition. Under the assumption that sufficiently large jumps are identifiable and removable in the sample, the estimation attains a minimax rate that is optimal up to a logarithmic factor if the Blumenthal--Getoor index is known. This regime is different from our situation, as the authors assume that we may remove all sufficiently large jumps. In fact, under this kind of assumption, the residual small-jump process may not be attracted to a stable process but to a Brownian motion~\cite{MR1834755}. 

In~\cite{MR3833470,Optimal2022}, the authors introduce couplings between L\'evy processes to bound the Wasserstein distance between them. The coupling in~\cite{MR3833470} is generic and pays special attention to the small jumps. However, the bounds fail to converge to zero when applied to a stable process and a L\'evy process in its small-time stable domain of attraction. In contrast to the coupling used in~\cite{MR3833470}, where the authors couple the big-jump components based on the magnitude of the jumps (i.e. based on a common threshold), we couple these components matching their jump intensities. Moreover, in~\cite{MR3833470} the authors couple the small jumps through an artificial Brownian motion, while we instead couple the compensated Poisson measures directly. On the other hand, the coupling in~\cite{Optimal2022}, based on McCann's coupling and Rogers’ results on random walks, is the optimal Markovian coupling. However, the coupling requires such tight control of the infinitesimal dynamics of the processes that the coupling could only be constructed for L\'evy processes with finitely many jumps on compact intervals, excluding all heavy-tailed stable processes and most processes in the small-time domain of attraction of Brownian motion.

Although the slow convergence phenomenon in the presence of a slowly varying function that does not converge to a positive constant has been observed in some specific settings such as in the case of Hill's estimator (see~\cite[p.~193--195]{MR1458613}), to the best of our knowledge it was first documented rigorously in~\cite{MR3806899} in an elementary general setting. The authors in~\cite{MR3806899} lower bound the Prokhorov distance between the marginals of the limit and that of a random walk in its domain of non-normal attraction with a function $b(n)$, satisfying  $\limsup_{n\to\infty}b(n)\log^{1+\ve}(n)=\infty$ for any $\ve>0$. However, as is often the case with lower bounds in the form of upper limits, the sparsity of the sequence of times along which the divergence holds remains unclear. The present paper extends the applicability of such a lower bound and strengthens the conclusions. In particular we show that the function analogous to $b(n)$ is typically slowly varying and provide some explicit asymptotically equivalent lower bounds.

\subsection{Organisation of the remainder of the article}
In Section~\ref{sec:main_results} we introduce our assumptions and the main results of the paper, namely, the upper and lower bounds on the convergence rate for processes in the domains of normal and non-normal attraction. More precisely, Section~\ref{subsec:stable_normal_attraction} deals with the heavy-tailed stable domain of normal attraction. Section~\ref{subsec:stable_non-normal_attraction} is dedicated to the rate of convergence in the stable domain of non-normal attraction. 
Section~\ref{subsec:selecting_the_coupling} discusses which  of the two couplings, simulated in Figure~\ref{fig:couplings} above, 
is used in each proof of the upper bounds in our results for the normal and non-normal stable domains of attraction, and why.
Section~\ref{subsec:Gaussian_domain_of_attraction} presents the results on the Gaussian domain of attraction. 
In Section~\ref{subsec:augment_stable}
we introduce a broad class of L\'evy processes (termed \textit{augmented stable processes}), which includes most of  Rosi\'nski's tempered stable class, a multivariate extension of Kuznetsov's $\beta$-class of L\'evy processes, and many meromorphic L\'evy processes, and we apply our results to augmented stable processes. Section~\ref{subsec:classical} explains the obstacles to directly applying the existing literature  to the problem of obtaining the bounds presented in Section~\ref{sec:main_results}.

The two couplings for general L\'evy processes on $\R^d$ used to establish the upper bounds on the Wasserstein distance in Section~\ref{sec:main_results} are introduced in Section~\ref{sec:couplings_general}. General upper bounds (in $L^p$) for each component in the L\'evy--It\^o decomposition of coupled L\'evy processes are also established in Section~\ref{sec:couplings_general}. The upper bounds for processes in both the normal and non-normal domains of attraction of a stable process (Gaussian and heavy-tailed) are established in Section~\ref{sec:stable_limits_upper}, while the lower bounds are established in Section~\ref{sec:lower_bounds}. The proofs of the results stated in Section~\ref{sec:main_results} are given in Section~\ref{sec:proofs}. Section~\ref{sec:conclusion} concludes the paper. 

\section{Assumptions and main results}
\label{sec:main_results}

The $L^q$-Wasserstein distance $\mW_q(\bm{\xi},\bm{\zeta})$, for any $q\in(0,\infty)$, between the laws of $\R^d$-valued random vectors $\bm{\xi}$ and $\bm{\zeta}$ equals $\inf_{\bm{\xi}'\eqd\bm{\xi}, \bm{\zeta}'\eqd\bm{\zeta}}\E[|\bm{\xi}'-\bm{\zeta}'|^q]^{1/(q\vee 1)}$ where the infimum is taken over all couplings $(\bm{\xi}',\bm{\zeta}')$ with $\bm{\xi}'\eqd\bm{\xi}$ and $\bm{\zeta}'\eqd\bm{\zeta}$ (throughout $|\cdot|$ denotes the Euclidean norm in $\R^d$ and $x\vee y:=\max\{x,y\}$ for $y,x\in\R$).\footnote{To show that $\mW_q$ is a metric for $q\in(0,1)$, note that for any $u>0$ and $y\ge0$ we have $q(u+y)^{q-1}\le qu^{q-1}$. Hence, by integrating in $u\in(0,x)$, we obtain $(x+y)^q\le x^q+y^q$ for all $x,y\ge0$.} For $\R^d$-valued stochastic processes $\bm{\mathcal{X}}=(\bm{\mathcal{X}}_t)_{t\in[0,1]}$ and $\bm{\mathcal{Y}}=(\bm{\mathcal{Y}}_t)_{t\in[0,1]}$, the $L^q$-Wasserstein distance, based on the uniform norm between the  paths, is given by:
\begin{equation}
\label{eq:def_wasserstein}
\mW_q(\bm{\mathcal{X}},\bm{\mathcal{Y}})
\coloneqq\inf_{\bm{\mathcal{X}}'\eqd\bm{\mathcal{X}},\bm{\mathcal{Y}}'\eqd\bm{\mathcal{Y}}}
    \E\bigg[\sup_{t\in[0,1]}
        |\bm{\mathcal{X}}'_t-\bm{\mathcal{Y}}'_t|^q\bigg]^{1/(q\vee 1)},
\qquad q>0,
\end{equation}
where  the infimum is taken over all couplings $(\bm{\mathcal{X}}',\bm{\mathcal{Y}}')$ with $\bm{\mathcal{X}}'\eqd\bm{\mathcal{X}}$ and $\bm{\mathcal{Y}}'\eqd\bm{\mathcal{Y}}$, where $\bm{\mathcal{X}}'\eqd\bm{\mathcal{X}}$ means that $\bm{\mathcal{X}}'$ and $\bm{\mathcal{X}}$ are equal in law as processes. The case $q\in(0,1)$ is important in our setting,  as the stable limit may not have a finite first moment. 

Let $\bm{X}=(\bm{X}_t)_{t \in [0,1]}$ and $\bm{Z}=(\bm{Z}_t)_{t\in [0,1]}$ be L\'evy processes in $\R^d$ (see~\cite[Ch.~1, Def.~1.6]{MR3185174} for definition), where $\bm{Z}$ is $\alpha$-stable (see Definition~\ref{def:stable} below for details). Note that, for any $\alpha\in(0,2]$, the generating triplet (see~\cite[Def.~8.2]{MR3185174} for a definition) of the $\alpha$-stable L\'evy process $\bm{Z}$ is given by $(\bm{\gamma_Z},\bm{\Sigma_Z}\bm{\Sigma_Z}^\tra,\nu_{\bm{Z}})$ (for the cutoff function $\bm{w}\mapsto\1_{B_{\bm{0}}(1)}(\bm{w})$ where $B_{\bm{0}}(r)\coloneqq \{x \in \R^d: |\bm{x}|<r\}$), where the L\'evy measure equals
\begin{equation}
\label{eq:Levy-measure-stable}
\nu_{\bm{Z}}(A)
\coloneqq c_\alpha\int_0^\infty\int_{\Sp^{d-1}}
    \1_{A}(x\bm{v})\sigma(\D\bm{v})x^{-\alpha-1}\D x,
\quad A\in\mathcal{B}(\R^d_{\bm{0}}),
\end{equation}
for a probability measure $\sigma$ on the Borel $\sigma$-algebra $\mathcal{B}(\Sp^{d-1})$ (hereafter, $\R^d_{\bm 0}\coloneqq\R^d\setminus\{\bm 0\}$ and $\Sp^{d-1}$ denotes the unit sphere in $\R^d$)\footnote{Note that $\bm{Z}$ need not be isotropic: $\sigma$ need not be uniform on $\Sp^{d-1}$.} and an ``intensity'' parameter $c_\alpha\in[0,\infty)$, satisfying $c_\alpha=0$ when $\alpha=2$ and $\bm{\Sigma_Z}=\bm{0}$ otherwise.

We say $\bm{X}$ is in the \textit{small-time domain of attraction} of $\bm{Z}$ if $(\bm{X}_{st}/g(t))_{s \in [0,1]}\cid (\bm{Z}_s)_{s\in [0,1]}$ as $t \da 0$ in the Skorokhod space for some normalising positive function $g:(0,1]\to(0,\infty)$. Then, it is well known that $\bm{Z}$ is $\alpha$-stable for some $\alpha\in(0,2]$ and the normalising function admits the representation $g(t)=t^{1/\alpha}G(t)$ where $G\in\SV_0$ (see~\cite[Eq.~(8)]{MR3784492}) is asymptotically unique (see Theorem~(\nameref{thm:small_time_domain_stable}) below for the description of all L\'evy processes $\bm{X}$ attracted to $\bm{Z}$). We say $\bm{X}$ is in the \textit{domain of normal attraction} if the slowly varying function $G$ has a positive finite limit at $0$ (see~\cite[p.~181]{MR0233400}). Otherwise, we say $\bm{X}$ is in the \textit{domain of non-normal attraction}. Throughout the paper we denote $\bm{X}^t=(\bm{X}_s^t)_{s\in[0,1]}\coloneqq (\bm{X}_{st}/g(t))_{s\in[0,1]}$ for $t\in (0,1]$. 

\subsection{Heavy-tailed stable domain of normal attraction}
\label{subsec:stable_normal_attraction}
For a L\'evy process $\bm{X}$ to be in the domain of attraction of an $\alpha$-stable process $\bm{Z}$, the necessary condition in~\eqref{eq:jump-stable-limit} of Theorem~(\nameref{thm:small_time_domain_stable}) suggests the L\'evy measure of $\bm{X}$ around the origin should be ``asymptotically absolutely continuous'' with respect to the $\alpha$-stable L\'evy measure of  $\bm{Z}$ (see also Remark~\ref{rem:finite-moments-under-T}(b) below). An effective use of the thinning coupling requires the generating triplet $(\bm{\gamma_X},\bm{\Sigma_X}\bm{\Sigma_X}^\tra,\nu_{\bm{X}})$ of $\bm{X}$ to satisfy the following.

\begin{assumption*}[T]
\label{asm:T}
Assume that $\bm{X}$ has no Gaussian component (i.e., $\bm{\Sigma_X}=\bm{0}$) and its L\'evy measure has a decomposition 
$\nu_{\bm{X}}=\nu_{\bm{X}}^{\co}+\nu_{\bm{X}}^{\mathrm{d}}$
satisfying the following: $\nu_{\bm{X}}^{\mathrm{d}}$ is arbitrary with finite mass $\nu_{\bm{X}}^{\mathrm{d}}(\R^d_{\bm{0}})<\infty$ and
\[
\nu_{\bm{X}}^{\co}(\D \bm{w})=c^{-1}f_{\bm{S}}(\bm{w})\nu_{\bm{Z}}(\D \bm{w})\quad\&\quad
|f_{\bm{S}}(\bm{w})-c|\le K_T(1\wedge |\bm{w}|^p),\quad\text{for all}\quad\bm{w} \in \R^d_{\bm{0}},
\]
a measurable function 
$f_{\bm{S}}:\R^d_{\bm{0}}\to[0,1]$  and constants
$K_T\in[0,\infty)$, $p\in(0,\infty)$ and $c\in(0,1]$.   
\end{assumption*}

\begin{remark}
\label{rem:finite-moments-under-T}
\nf{(a)} Assumption~(\nameref{asm:T}) quantifies the regularity of the corresponding density at the origin $\bm{0}$ via the parameter $p>0$ (the larger $p$ is, the more asymptotic regularity there is). 
Condition~\eqref{eq:jump-stable-limit} in Theorem~(\nameref{thm:small_time_domain_stable}) suggests that the L\'evy measure of the process $\bm{X}$, which is in the domain of attraction of a stable process $\bm{Z}$, possesses a decomposition of the type $\nu_{\bm{X}}=\nu_{\bm{X}}^{\co}+\nu_{\bm{X}}^{\mathrm{d}}$. Moreover, Assumption~(\nameref{asm:T}), widely satisfied in practice (e.g. the class of augmented stable processes introduced in Section~\ref{subsec:augment_stable} below, which includes the exponentially tempered stable processes~\cite{MR2327834} with $p=1$), can be seen as quantifying the speed of convergence in the necessary condition~\eqref{eq:jump-stable-limit} for $\bm{X}$ to be in the stable domain of attraction. In this sense, Assumption~(\nameref{asm:T}) is a second-order condition.

\noindent\nf{(b)} Assumption~(\nameref{asm:T}) implies $\int_{\R^d\setminus B_{\bm{0}}(1)}|\bm{w}|^q\nu_{\bm{X}}(\D\bm{w})\le (1+K_Tc^{-1})\int_{\R^d\setminus B_{\bm{0}}(1)}|\bm{w}|^q\nu_{\bm{Z}}(\D\bm{w})<\infty$ for all $q\in(0,\alpha)$. Hence, by~\cite[Thm~25.3]{MR3185174}, the component $\bm{S}$ of $\bm{X}$ with L\'evy measure $\nu^\co_{\bm{X}}$ has as many moments as the limit $\bm{Z}$. Note that this is not a restriction on $\bm{X}$ since $\nu_{\bm{X}}^{\mathrm{d}}$ may contain all the mass of $\nu_{\bm{X}}$ outside of some neighborhood of $\bm{0}$.
\end{remark}

Throughout the paper, for positive functions $f_1$ and $f_2$, we use the notation $f_1(x)=\Oh(f_2(x))$ (resp. $f_1(x)=\oh(f_2(x))$; $f_1(x)\sim f_2(x)$) as $x \da 0$ if $\limsup_{x \da 0}f_1(x)/f_2(x)<\infty$ (resp. $\lim_{x \da 0}f_1(x)/f_2(x)=0$; $\lim_{x \da 0}f_1(x)/f_2(x)=1$).

\begin{theorem}
\label{thm:upper_lower_simple}
Let $\alpha \in (0,2)\setminus\{1\}$, and suppose $\bm{Z}$ is $\alpha$-stable and $\bm{X}$ is in the domain of normal attraction of $\bm{Z}$.\\
{\nf{(a)}} Let Assumption~(\nameref{asm:T}) hold for $p\ge 1\vee 2(\alpha-1)$. Then, for any $q\in(0,1]\cap(0,\alpha)$ with $\E[|\bm{X}_1|^q]<\infty$, as $t \da 0$, 
\[
\mW_q\big(\bm{X}^t,\bm{Z}\big)
=\begin{dcases}
    \Oh\big(t^{1-q/\alpha}\big), 
   & \alpha<1,\\
   \Oh\big(t^{q(1-1/\alpha)}\big), 
   & \alpha>1.
\end{dcases}
\]
{\nf{(b)}} If $\bm{X}$ does not have the law of $\bm{Z}$, then for any $q\in(0,1]\cap(0,\alpha)$ there exists some $C_q>0$ satisfying 
\begin{align*}
\mW_q(\bm{X}^t,\bm{Z})
    \ge\mW_q(\bm{X}_1^t,\bm{Z}_1) 
    \ge C_q t^{1-q/\alpha},
\quad \text{for all sufficiently small }t>0. 
\end{align*}
\end{theorem}

\begin{remark}
\nf{(a)} The upper bounds in Theorem~\ref{thm:upper_lower_simple}(a) are based on the thinning coupling in Section~\ref{sec:thinning} below. The upper bounds are asymptotically proportional to the lower bounds of Theorem~\ref{thm:upper_lower_simple}(b) when either $\alpha<1$ or $q=1$ when $\alpha>1$,  making the thinning coupling rate-optimal with respect to these Wasserstein distances. Coincidentally, the upper bounds decay the fastest for small values of $q$ when $\alpha<1$ and for $q=1$ when $\alpha>1$.
Moreover, the multiplicative constants in $\Oh$ can be made explicit and depend on the dimension $d$ only through the characteristics of $\bm{X}$ and $\bm{Z}$. The lower bounds are based on the lower bound on the Toscani--Fourier distance, see Section~\ref{subsec:lower_bounds_toscani_fourier} below for details.\\
\nf{(b)} Observe that $\E[|\bm{Z}_1|^q]<\infty$ for $q<\alpha$ and that most models in practice satisfy Assumption~(\nameref{asm:T}) with $p=1$. Theorem~\ref{thm:upper_lower_simple} thus focuses on the case $p\ge 1\vee 2(\alpha-1)$ in order to simplify the exposition, while retaining the key message of the paper. Our technical result Theorem~\ref{thm:d_thin_dom_attract} in Section~\ref{sec:stable_limits_upper} (resp. Lemma~\ref{lem:lower_bound_alpha_stab} in Section~\ref{subsec:lower_bounds_toscani_fourier}), used to prove part~(a) (resp. part~(b)) of Theorem~\ref{thm:upper_lower_simple}, covers all parameters $p>0$ and $\mW_q$-distances with $q\in(0,1]\cap (0,\alpha)$. The statement of the corresponding general version of Theorem~\ref{thm:upper_lower_simple} is omitted for brevity.\\
\nf{(c)} It is unsurprising to see the convergence of the $\mW_q$-distance deteriorate as $\alpha\da 1$ (for fixed $q=1$) since there is a regime change for the limiting stable processes, switching from infinite variation processes with finite mean to finite variation processes with infinite mean.\\
\nf{(d)} If we remove the normalisation in $\bm{X}^t$ by multiplying it with the normalising function $g(t)= t^{1/\alpha}$ and applying the self-similarity of $\bm{Z}$, then Theorem~\ref{thm:upper_lower_simple} implies that the Wasserstein distance  $\mW_q(\bm{X}_t,\bm{Z}_t)$ between the marginals of $\bm{X}$ and $\bm{Z}$ at time $t$ is of order $\Oh(t)$ (resp. $\Oh(t^q)$) when $\alpha<1$ (resp. $\alpha>1$). 
Moreover, $\mW_q(\bm{X}_t,\bm{Z}_t)$ is also lower bounded by a multiple of $t$. In this interpretation of Theorem~\ref{thm:upper_lower_simple}, the rates depend on $\alpha$ only through whether $\alpha<1$ or $\alpha>1$ and not on its actual value. 
\end{remark}

\subsection{Stable domain of non-normal attraction}
\label{subsec:stable_non-normal_attraction}
Consider the case where the slowly varying function $G$ in the scaling limit $\bm{X}_{t}/(t^{1/\alpha}G(t))\cid \bm{Z}_1$, as $t \da 0$, is not asymptotically equivalent to a positive constant. In this section, we show that the lower bound on the $\mW_q$-distance cannot be upper bounded by a positive non-decreasing function $\phi$ satisfying $\int_0^1\phi(t)t^{-1}\D t<\infty$. The lower bound requires no assumptions (beyond $\bm{X}$ belonging to a domain of attraction), while the upper bounds require two sets of assumptions. The first gives us multiplicative non-asymptotic control over the distance from $G(xt)/G(t)$ to $1$ for small $t>0$ and any $x>0$ (recall that $G(xt)/G(t)\to 1$ as $t\da 0$ by the slow variation of $G$).
\begin{assumption*}[S]
\label{asm:S}
There exist $G_1,\,G_2:[0,\infty) \to [0,\infty)$, such that $G_2$ is bounded and satisfies $G_2(t)\to 0$ as $t \da 0$, $G_1$ is a slowly varying function both at zero and at infinity and   
\begin{equation*}
\left|G(t/x)/G(t)-1 \right|
\le G_1(x)G_2(t), \quad \text{ for all }x>0\text{ and all sufficiently small }t>0.
\end{equation*} 
\end{assumption*}

Define the directional measure $\rho_{\bm{Z}}(\D x,\bm{v})\coloneqq c_\alpha x^{-\alpha-1}\D x$ on $\mathcal{B}((0,\infty))$ and note that the right inverse of its tail
$x \mapsto \rho_{\bm{Z}}([x,\infty),\bm{v})$ is given by 
$\rho_{\bm{Z}}^{\la}(x,\bm{v})=(c_\alpha/\alpha)^{1/\alpha}x^{-1/\alpha}$ for all $x>0$ and $\bm{v}\in \Sp^{d-1}$.
The comonotonic coupling of $\bm{Z}$ and $\bm{X}$, used to establish the upper bound on the Wasserstein distance in the weak limit, requires the following assumption on the generating triplet $(\bm{\gamma_X},\bm{\Sigma_X}\bm{\Sigma_X}^\tra,\nu_{\bm{X}})$ of $\bm{X}$.
 
\begin{assumption*}[C]
\label{asm:C} 
Assume $\bm{X}$ has no Gaussian component $\bm{\Sigma_X}=\bm{0}$ and $\nu_{\bm{X}}=\nu_{\bm{X}}^{\co}+\nu_{\bm{X}}^{\mathrm{d}}$, where the measure $\nu_{\bm{X}}^{\mathrm{d}}$ is arbitrary with finite mass  $\nu_{\bm{X}}^{\mathrm{d}}(\R^d_{\bm{0}})<\infty$ and  the L\'evy measure $\nu_{\bm{X}}^{\co}$ 
can be expressed as 
\begin{equation}
\label{eq:radial_tail_decomp}
    \nu_{\bm{X}}^{\co}(B)
=\int_{\Sp^{d-1}}\int_0^\infty \1_{B}(x\bm{v})\rho_{\bm{X}}^{\co}(\D x,\bm{v})\sigma(\D\bm{v}) \quad\&\quad
\rho_{\bm{X}}^{\co}([x,\infty),\bm{v})
=\frac{c_\alpha}{\alpha}
    (1+h(x,\bm{v}))H(x)^\alpha x^{-\alpha}
\end{equation}
for all $B\in\mathcal{B}(\R^d_{\bm{0}})$,
$x>0$, $\bm{v}\in \Sp^{d-1}$ and some
monotonic function $H:(0,\infty)\to (0,\infty)$ (either non-increasing or non-decreasing), slowly varying at $0$, and a measurable function $h:(0,\infty)\times \Sp^{d-1}\to [-1,\infty)$. Define
\begin{equation}
\label{eq:defn_G}
G(x)\coloneqq\int_{\Sp^{d-1}}H(\rho_{\bm{X}}^{\co\la}(1/x,\bm{u}))\sigma(\D \bm{u}),\qquad x>0,
\end{equation}
where  
$\rho_{\bm{X}}^{\co\la}(\cdot,\bm{v})$ is the right-continuous inverse of $x\mapsto \rho_{\bm{X}}^\co([x,\infty),\bm{v})$, and assume that $H$, $h$ and $G$ satisfy
\begin{equation}\label{eq:old_assump_(H)}
|h(x,\bm{v})|
    \le K_h(1\wedge x^{p}) 
\quad \& \quad 
\left|H(\rho_{\bm{X}}^{\co\la}(1/x,\bm{v}))/G(x)-1\right|
    \le K_Q(1\wedge x^{\delta}) 
\quad \text{for all }x>0, \bm{v} \in \Sp^{d-1},
\end{equation}
and some constants $p,\delta>0$ and $K_h,K_Q\ge 0$. 
\end{assumption*}

\begin{remark}
\label{rem:asm-C-vs-minimal}
Condition~\eqref{eq:jump-stable-limit} in Theorem~(\nameref{thm:small_time_domain_stable}) below states that the L\'evy measure of the process $\bm{X}$ in the domain of attraction of a stable process $\bm{Z}$ behaves as the L\'evy measure of $\bm{Z}$ in every half-space of the form $\scrL_{\bm v}(c)$ for every $\bm{v}\in\Sp^{d-1}$ and small $c>0$. Assumptions~(\nameref{asm:S}) and (\nameref{asm:C}) may thus be interpreted as a refinement of this condition, requiring the L\'evy measure of $\bm{X}$ to satisfy an analogue of~\eqref{eq:jump-stable-limit} but on every ray directed by $\bm{v}\in\Sp^{d-1}$ and quantifying how fast such a limit holds. In other words, Assumption~(\nameref{asm:C}) essentially requires the regular variation of the $\nu^\co_{\bm X}$ at the origin to be sufficiently uniform in every direction $\bm{v}\in\Sp^{d-1}$ via the second-order conditions on $h$ and $H/G$ in~\eqref{eq:old_assump_(H)} (which are similar to the second-order condition of Assumption~(\nameref{asm:T}) in spirit).

Under Assumptions~(\nameref{asm:S}) and (\nameref{asm:C}) and for $G$ as in \eqref{eq:defn_G}, let $g(t):=t^{1/\alpha}G(t)$ with $t\in(0,1]$. For such a $g$ it follows that $(\bm{X}_{st}/g(t))_{s\in [0,1]} \cid (\bm{Z}_s)_{s \in [0,1]}$, i.e. $\bm{X}$ is in the small-time domain of non-normal attraction of $\bm{Z}$ (by Theorem~(\nameref{thm:small_time_domain_stable}) below). Thus, under Assumption~(\nameref{asm:C}), the function $G$ is explicit, monotonic and slowly varying at infinity. In fact, for any $\bm{v}\in\Sp^{d-1}$, we have\footnote{Throughout, we use the notation $f(x) \sim g(x)$ as $x \to a$, if $\lim_{x \to a} f(x)/g(x)=1$.} $G(x)\sim H(\rho_{\bm{X}}^{\co\la}(1/x,\bm{v}))$ as $x\da 0$ by virtue of~\eqref{eq:old_assump_(H)}, which is slowly varying by~\cite[Prop.~1.5.7(ii)]{MR1015093} (since $x\mapsto\rho_{\bm{X}}^{\co\la}([x,\infty),\bm{v})$ is regularly varying at infinity and $H$ is slowly varying).
\end{remark}

Intuitively, these assumptions require non-parametric structural properties of the L\'evy measure of $\bm{X}$ that allows us to compare it to the  L\'evy measure of the stable limit $\bm{Z}$. Indeed, the necessary condition in~\eqref{eq:jump-stable-limit} of Theorem~(\nameref{thm:small_time_domain_stable}) suggests the L\'evy measure of $\bm{X}$ around the origin should ``asymptotically admit a radial decomposition that is close to that of the stable process''. Assumption~(\nameref{asm:C}) states precisely this and specifies the proximity of the corresponding radial decomposition to that of the stable process via the parameters $p,\delta>0$ (as before, the larger $p$ and $\delta$ are, the closer the radial decompositions are). Moreover, both conditions are commonly satisfied (for example, within the class of augmented stable processes, see Section~\ref{subsec:augment_stable} below) and, with $p=1$, for the class of exponentially tempered $\alpha$-stable processes (see~\cite{MR2327834} and Example~\ref{ex:exp_temp_stable} below). 

\begin{theorem}
\label{thm:upper_lower_general}
Let $\bm{X}$ belong to the domain of non-normal attraction of an $\alpha$-stable process $\bm{Z}$.\\
{\normalfont(a)} Let $\alpha\in(0,2)\setminus\{1\}$ and Assumptions~(\nameref{asm:C}) and (\nameref{asm:S}) hold for some $p\in(0,\infty)\setminus\{ \alpha-1\}$, $\delta>0$ and $G_2\in\SV_0$. Then, 
\[
\mW_q(\bm{X}^t,\bm{Z})
=\Oh ( G_2(t)^q)
\qquad\text{as }t \da 0,
\]
for any $q\in(0,1]\cap(0,\alpha)\setminus\{\alpha/(p+1),\alpha/(\alpha\delta+1)\}$ with $\E[|\bm{X}_1|^q]<\infty$.\\
{\normalfont(b)} Let $\alpha \in (0,2]$ and define $a(t) \coloneqq G(2t)/G(t)$ for $t>0$. Then for any $q\in(0,1]\cap(0,\alpha)$,  
\begin{align*}
\max\big\{\mW_q(\bm{X}^t_1,\bm{Z}_1)
,\mW_q(\bm{X}^{2t}_1,\bm{Z}_1)\big\}
\ge |1-a(t)^q|\E\big[|\bm{Z}_1|^q\big]/3,
\quad \text{for all sufficiently small }t >0.
\end{align*} 
Moreover, $|1-a(t)^q|$ cannot be upper bounded by a non-decreasing function $\phi$ with $\int_0^1\phi(t)t^{-1}\D t<\infty$. 
\end{theorem}

Since $G_2$ bounds $|G(2t)/G(t)-1|\le G_1(1/2)G_2(t)$ for all small $t>0$ and $G$ is not asymptotically equivalent to a positive finite constant, Lemma~\ref{lem:1-a_subpoly} below (which extends~\cite[Prop.,~p.~683]{MR3806899}) implies $G_2$
cannot be upper bounded by any non-decreasing function $\phi$ satisfying $\int_0^1\phi(t)t^{-1}\D t<\infty$. The assumption on the slow variation of 
$G_2$ in Theorem~\ref{thm:upper_lower_general}(a) is not essential and may be replaced by assuming that
$G_2$ dominates any (positive) power at zero. However, by Lemma~\ref{lem:1-a_SV} below, in most cases of interest such a function $G_2$ will be slowly varying. 

Given a slowly varying function $G$, the construction of functions $G_1$ and $G_2$ satisfying Assumption~(\nameref{asm:S}) is not immediately clear. However, in most cases and for a sufficiently regular $G$, by virtue of Lemma~\ref{lem:1-a_SV}, Assumption~(\nameref{asm:S}) will be satisfied by choosing $G_2(t)\sim t|G'(t)|/G(t)$ as $t\downarrow0$ and a slowly varying $G_1$ (at $0$ and $\infty$) with $G_1(x)\ge |\log x|\cdot\sup_{t>0, y\in[x\wedge 1,x\vee1]}yG'(yt)/G'(t)$ (for $a,b\in\R$, we denote $a\wedge b\coloneqq \min\{a,b\}$). In such cases, the lower bound in Theorem~\ref{thm:upper_lower_general} is (by Lemma~\ref{lem:1-a_SV})  proportional to $G_2$, i.e. $G_2(t)\sim |1-a(t)|/\log2$ as $t\downarrow0$, making the comonotonic coupling rate-optimal with respect to the $\mW_1$-distance when $\alpha>1$. The following corollary makes this precise and shows that this is the case for a large class of processes in the domain of non-normal attraction.

\begin{corollary}
\label{cor:upper_lower_general}
Let $\bm{X}$ be in the domain of non-normal attraction of an $\alpha$-stable process $\bm{Z}$.\\ 
{\normalfont(a)} Let $\alpha\in(0,2)\setminus\{1\}$ and Assumption~(\nameref{asm:C}) hold for some $\delta>0$ and $p\ne\alpha-1$. Suppose $G$ is $C^1$ with derivative $\wt{G}(t)/t$, where $|\wt{G}|\in\SV_0$. Further suppose there exists some $\phi\in\SV_0\cap\SV_\infty$ satisfying $\phi(x)\ge\sup_{t>0,y\in[x\wedge 1,x\vee1]}\wt{G}(yt)/\wt{G}(t)$ for $x>0$. Define $L(t)\coloneqq |\wt{G}(t)|/G(t)$, then, for any $q\in(0,1]\cap(0,\alpha)\setminus\{\alpha/(p+1),\alpha/(\alpha\delta+1)\}$ with $\E[|\bm{X}_1|^q]<\infty$, we have $\mW_q(\bm{X}^t,\bm{Z})=\Oh(L(t)^q)$ as $t \da 0$ and
\begin{align*}
    \max\{\mW_q(\bm{X}^t_1,\bm{Z}_1)
    ,\mW_q(\bm{X}^{2t}_1,\bm{Z}_1)\}
    \ge \frac{q\log 2}{3}\E\big[|\bm{Z}_1|^q\big]\cdot L(t),
    \quad \text{for all sufficiently small }t >0.
\end{align*}
{\normalfont(b)} Define iteratively the functions $\ell_1(t)=\log(e+t)$ and $\ell_{n+1}(t)=\log(e+\ell_n(t))$ for $t\ge 0$ and $n\in\N$. Suppose $G(t)$ is eventually equal to $\ell_n(1/t)^{q_n}\cdots\ell_m(1/t)^{q_m}$ where $1\le n\le m$ in $\N$ and either $q_n,\ldots,q_m\ge 0$ with $q_n,q_m>0$ or $q_n,\ldots,q_m\le 0$ with $q_n,q_m<0$. Then $G$ satisfies the assumptions of Part (a).
\end{corollary}

\begin{remark}
\label{rem:upper_lower_general}
\nf{(I)} The upper bound in Theorem~\ref{thm:upper_lower_general}(a) is based on the comonotonic coupling in Section~\ref{sec:comonotonic_coupling} below. Since the bound is independent of both $p$ and $\delta$, the restrictions $p\ne \alpha-1$ and $q\notin\{\alpha/(p+1),\alpha/(\alpha\delta+1)\}$ are nonessential. Indeed, if $p$ and $\delta$ satisfy Assumption~(\nameref{asm:C}), then any $p'\in(0,p)$ and $\delta'\in(0,\delta)$ also satisfy it. 
Moreover, the multiplicative constants in $\Oh$ can be made explicit and depend on the dimension $d$ only through the characteristics of $\bm{X}$ and $\bm{Z}$.\\
\nf{(II)} The lower bounds are based on elementary estimates and a universal property of slowly varying functions, see Section~\ref{subsec:prokhorov_lower_bound} below for details. When $\alpha=2$, despite the absence of an upper bound in Theorem~\ref{thm:upper_lower_general}(a) for this case, the lower bound of Theorem~\ref{thm:upper_lower_general}(b) ensures the nonexistence of a coupling for which the $L^q$-norm decays polynomially.\\
\nf{(III)} Corollary~\ref{cor:upper_lower_general} is a consequence of Theorem~\ref{thm:upper_lower_general} and Lemmas~\ref{lem:1-a_SV} \&~\ref{lem:iterated-log} below. Furthermore, we stress that the resulting upper and lower bounds may converge slowly and at a rate that is, in some sense, ``bounded away from polynomials'' even for a very slow function $G$ such as $t\mapsto\ell_n(1/t)$ or $t\mapsto1/\ell_n(1/t)$, $n\in\N$, see Example~\ref{ex:arbitrarily_slow} below. Furthermore, given any $\ell\in\SV_0$ with $\ell(t)\to 0$ as $t\da 0$ and $\int_0^1\ell(t)t^{-1}\D t=\infty$, the functions $G_\pm(t)\coloneqq\exp(\pm\int_t^1\ell(s)s^{-1}\D s)$ are slowly varying, $G_+(t)\to\infty$, $G_-(t)\to 0$ and the corresponding $G_2(t)$ functions are proportional to $\ell(t)$ as $t\downarrow0$. Thus, we may construct processes $\bm{X}$ such that $\max\{\mW_q(\bm{X}^t,\bm{Z}),\mW_q(\bm{X}^{2t},\bm{Z})\}$ is asymptotically bounded above and below by multiples of $\ell(t)$ as $t\da 0$, see Lemma~\ref{lem:1-a_subpoly} and Example~\ref{ex:arbitrarily_slow} below.\\
\nf{(IV)} For a given process $\bm{X}$, we may choose two asymptotically equivalent slowly varying functions $G$ and $\hat{G}$ that have different convergence properties. Indeed, if $G'$ is not asymptotically equivalent to $\hat{G'}$, then the resulting bounds will change (recall that $G$ is only unique up to asymptotic equivalence and that $\bm{X}^t_s=\bm{X}_{st}/(t^{1/\alpha}G(t))$). For instance, fix $r_1,r_2\in(0,1)$, denote $\ell(t)=|\log t|^{r_1}$ and $r_3=(r_1-1)(r_2+1)+1$ and let 
\begin{align*}
&G'(t)
\coloneqq \ell'(t)[1+(1-\ell'(t)^{r_2})\cos\ell(t)],
&&G(t)\coloneqq
\int_t^1G'(s)\D s
=\ell(t)+\sin\ell(t)+\Oh(|\log t|^{r_3}),\\
&\hat{G}'(t)
\coloneqq \ell'(t)[1+(1-\ell'(t)^{r_2})\sin\ell(t)],
&&\hat{G}(t)\coloneqq
\int_t^1G'(s)\D s=\ell(t)-\cos\ell(t)+\Oh(|\log t|^{r_3}).
\end{align*}
Then $tG'(t)$ and $t\hat{G}'(t)$ are slowly varying and $G(t)/\hat{G}(t)\to1$ as $t\to0$ but $\limsup_{t\to0}G'(t)/\hat{G}'(t)=\infty$ and $\liminf_{t\to0}G'(t)/\hat{G}'(t)=0$. Optimising the convergence rate within this class appears to be a very difficult task; however, the limitations imposed by Lemma~\ref{lem:1-a_subpoly} would apply to any choice of $G$. A similar phenomenon was also observed recently in the standard central limit theorem for L\'evy processes in~\cite{BANG2021109187}, where the Kolmogorov distance is shown to satisfy (resp. fail) an integral condition for a non-standard (resp. standard) normalisation.
\\
\nf{(V)} 
Theorem~\ref{thm:upper_lower_general} makes full use of Assumption~(\nameref{asm:S}), however, a more detailed analysis that does not require $G_2$ to be slowly varying can be found in our technical result Theorem~\ref{thm:upper_bound_comono_tech} in Section~\ref{sec:stable_limits_upper} below.
\\
\nf{(VI)} We note that a lower bound via the Toscani--Fourier distance is plausible but appears suboptimal since the rate has a polynomial factor. Moreover, we believe the slow lower bound in part (b) to hold for $\mW_q(\bm{X}^t_1,\bm{Z}_1)$ alone (i.e. without taking the maximum value between times $t$ and $2t$). However, this remains a conjecture.
\end{remark}

\subsection{Selecting the coupling}
\label{subsec:selecting_the_coupling}

The main idea behind the proof of Theorems~\ref{thm:upper_lower_simple} \&~\ref{thm:upper_lower_general} is a good coupling between $\bm{X}$ and $\bm{Z}$. The two couplings we apply in this article, are the thinning coupling and the comonotonic coupling introduced in Sections~\ref{sec:thinning} \&~\ref{sec:comonotonic_coupling} below. In Theorem~\ref{thm:upper_lower_general} we solely apply the comonotonic coupling, since this yields clear and concise results. Note that one could apply the thinning coupling to get a similar result in the domain of non-normal attraction. However, since this would require a lengthy argument, and would distort the main story and result, this has been left out of the paper. In comparison, it is easier to use the comonotonic coupling to give bounds for processes in the domain of normal attraction.

\begin{proposition}
\label{prop:normal_domain_comonotonic}
Let $\bm{Z}$ be $\alpha$-stable with $\alpha \in (0,2)$ and $\bm{X}$ be in its domain of normal attraction. Let Assumption~(\nameref{asm:C}) (with constant $H\equiv G$ and $Q\equiv 1$) hold for some $p\in(0,\infty)\setminus\{\alpha-1\}$. Then, for any $q\in(0,1]\cap(0,\alpha)$ with $\E[|\bm{X}_1|^q]<\infty$, we have, as $t \da 0$, 
\[
\mW_q\big(\bm{X}^t,\bm{Z}\big)
=\begin{dcases}
\Oh\big(t^{\min\{pq/\alpha,1-q/\alpha\}}\big),&\alpha\in(0,1),\\
\Oh\big(t^{q\min\{p/\alpha,1-1/\alpha\}}\big),&\alpha\in(1,2).
\end{dcases}
\]
\end{proposition}

\begin{remark}
Proposition~\ref{prop:normal_domain_comonotonic} follows from Theorem~\ref{thm:upper_bound_comono_tech} (see Remark~\ref{rem:thm_upper_lower_comno_tech}) below. The assumptions in Theorem~\ref{thm:upper_lower_simple}(a) and Proposition~\ref{prop:normal_domain_comonotonic} differ significantly, making it necessary to split the upper bounds in two statements. Indeed, as seen in Example~\ref{ex:fulfill_assT_NOT_ass(C+H)} below, Assumption~(\nameref{asm:C}) is somewhat stricter than Assumption~(\nameref{asm:T}), since we can show that there exist processes for which Assumption~(\nameref{asm:T}) is true, where Assumption~(\nameref{asm:C}) is no longer valid. In the case where Assumptions~(\nameref{asm:T}) and~(\nameref{asm:C}) are valid simultaneously with the same parameter $p\ge 1\vee 2(\alpha-1)$, Theorem~\ref{thm:upper_lower_simple} yields an upper bound that is never worse than that of Proposition~\ref{prop:normal_domain_comonotonic}. However, if we only have $p>\alpha-1>0$ (say, $p=1$ and $\alpha>3/2$), then Proposition~\ref{prop:normal_domain_comonotonic} leads to the optimal bound while Theorem~\ref{thm:upper_lower_simple} (or rather, Theorem~\ref{thm:d_thin_dom_attract}) does not.
\end{remark}

\begin{example}\label{ex:fulfill_assT_NOT_ass(C+H)}
We now show that there exist processes satisfying Assumption~(\nameref{asm:T}) but not Assumption~(\nameref{asm:C}). Let $\alpha \in (1,2)$ and $\alpha ' \in (1,\alpha)$. Next, let $X$ be a one-dimensional $\alpha$-stable process and $Y$ be an $\alpha'$-stable process that is spectrally negative, with L\'evy measures $\nu_X(\D x)=c_1|x|^{-1-\alpha}\D x$ and $\nu_Y(\D x)=c_2\1_{(-\infty,0)}(x)|x|^{-1-\alpha'}\D x$ for some constants $c_1,c_2>0$. We note that $X+Y$ has L\'evy measure $\nu_{X+Y}(\D x)=[c_2\1_{(-\infty,0)}(x)|x|^{-1-\alpha'}+c_1|x|^{-1-\alpha}]\D x$, demonstrating that Assumption~(\nameref{asm:T}) is satisfied. We can, however, note that Assumption~(\nameref{asm:C}) cannot be fulfilled, since there does not exist the necessary radial decomposition of $\nu_{X+Y}$.
\end{example}

\subsection{Gaussian domain of attraction}
\label{subsec:Gaussian_domain_of_attraction}
The domain of attraction to Brownian motion is substantially different, as the previously described couplings are inapplicable. Obtaining a coupling between Brownian motion and other L\'evy processes that reduced the $L^p$-distance in uniform norm has been the work of a large body of literature (which we review in Section~\ref{subsec:classical} below). In this paper, we use a simple independent coupling, which, heuristically, compares the pure-jump component of $\bm{X}$ with the null process $\bm{0}$. Let $\langle\cdot,\cdot\rangle$ denote the Euclidean inner product on $\R^{d \times d}$, $\bm{0}$ denote the zero-vector in $\R^d$ as well as the zero-matrix in $\R^{d \times d}$ and let $\R^d_{\bm{0}}\coloneqq\R^d\setminus\{\bm{0}\}$. Let $\varphi_{\bm{X}}(\bm{u})\coloneqq \E[e^{i\langle \bm{u},\bm{X}\rangle}]$ for $\bm{u} \in \R^d$ denote the characteristic function of $\bm{X}$. Furthermore, let $\psi_{\bm{S}}$ be the L\'evy-Khintchine exponent of $\bm{S}$, given by $\psi_{\bm{S}}(\bm{u})\coloneqq t^{-1}\log \varphi_{\bm{S}_t}(\bm{u})$ for $\bm{u} \in \R^d$ and $t>0$. 

\begin{theorem}
\label{thm:conv_BM_limit}
Let $\bm{\Sigma}$ be a symmetric non-negative definite matrix on $\R^{d \times d}$ and define $\bm{X}^t=((\bm{\Sigma}\bm{B}_{st}+\bm{S}_{st})/\sqrt{t})_{s\in[0,1]}$ for $t\in(0,1]$ where $(\bm{B}_t)_{t \in [0,1]}$ is a standard Brownian motion on $\R^d$ independent of the pure-jump L\'evy process $\bm{S}$ with Blumenthal--Getoor index~$\beta$ (defined in~\eqref{eq:BG}).\\
{\nf(a)} Suppose $\beta\in[0,2)$ and fix any $\beta_* \in (\beta,2]$ when $\bm{S}$ is of infinite variation and $\beta_*=1$ otherwise. Then for any $q>0$ with $\E[|\bm{X}_1|^q]<\infty$, we have
\begin{align*}
\mW_q\big(\bm{X}^t,\bm{\Sigma}\bm{B}\big)
=\Oh\big(t^{(q\wedge 1)
    (\min\{1/q,1/\beta_*\}-1/2)}\big), 
    \qquad \text{ as }t \da 0.
\end{align*}
{\normalfont(b)} Pick any $\bm{u}_*\in\R^d_{\bm{0}}$ and define $C_*\coloneqq |\bm{u}_*|^{-1}|\psi_{\bm{S}}(\bm{u}_*)|>0$. Then for all $q \ge 1$, we have 
\begin{align*}
\mW_q(\bm{X}^t_1,\bm{\Sigma}\bm{B}_1)
    \ge\frac{C_*}{\sqrt{2}} \sqrt{t}+ \Oh(t^{3/2}),
\quad\text{as }t \da 0.
\end{align*}
{\normalfont(c)}  Let $\lambda$ be the largest eigenvalue of $\bm{\Sigma}^2$. Suppose there exist $\delta\in[1,2)$ and vectors $(\bm{u}_r)_{r\in(0,\infty)}$ with $|\bm{u}_r|=r$ satisfying $c\coloneqq\inf_{r>1}r^{-\delta}|\psi_{\bm{S}}(\bm{u}_r)|>0$. Then for any $C_*\in(0,ce^{-\lambda/2})$ we have
$\mW_q(\bm{X}^t_1,\bm{\Sigma}\bm{B}_1) \ge(C_*/\sqrt{2})t^{1-\delta/2}$ for all sufficiently small $t>0$.
\end{theorem}

Parts~(a) and (c) of Theorem~\ref{thm:conv_BM_limit}, with $q=1$, imply that for processes whose pure jump part is in the domain of attraction of a $\beta$-stable process, 
the upper and lower bounds are essentially proportional to $t^{1/\max\{1,\beta\}-1/2}$ and $t^{1-\max\{1,\beta\}/2}$, respectively. These agree in the finite variation case with rate $\sqrt{t}$ and also as $\beta\to 2$ with an arbitrarily deteriorating convergence rate. As shown in Figure~\ref{fig:normal_rates} these bounds are not far from each other for fixed $\beta$, and the powers of $t$ from the rates are also not far. In the `limiting case' where $\bm{S}$ is itself attracted to a Brownian motion and $\beta=2$, the rescaled process $(\bm{S}_{st}/\sqrt{t})_{s\in[0,1]}$ is distributionally close to $\ell(t)\bm{B}$ (see e.g.~\cite[Thm~2]{MR3784492}) for a slowly varying function $\ell$ satisfying $\lim_{t\da 0}\ell(t)=0$. It is thus natural to expect that the convergence, in this case, is slow as in Theorem~\ref{thm:upper_lower_general} above, see~Example~\ref{ex:arbitrarily_slow_BM_conv}. 

\begin{example}\label{ex:Gaussian_perturbation}
Let $\bm{\Sigma}$ be a positive definite matrix on $\R^{d \times d}$ and set $\bm{X}^t\coloneqq((\bm{\Sigma}\bm{B}_{st}+\bm{S}_{st})/\sqrt{t})_{s\in[0,1]}$ for $t\in(0,1]$ where $(\bm{B}_t)_{t \in [0,1]}$ is a standard Brownian motion on $\R^d$ independent of the pure-jump tempered $\alpha$-stable L\'evy process $\bm{S}$. Assume $\alpha\in[1,2)$, that $\bm{S}$ has zero-mean, and fix any $\beta_* \in (\alpha,2]$. 
Then, by Theorem~\ref{thm:conv_BM_limit}(a), we have the upper bound  $\mW_1\big(\bm{X}^t,\bm{\Sigma}\bm{B}\big)
=\Oh\big(t^{
    1/\beta_*-1/2}\big)$ as $t \da 0$. To find the lower bound, we let $\lambda$ be the largest eigenvalue of $\bm{\Sigma}^2$, and define $c:=\inf_{r>1}r^{-\alpha}|\psi_{\bm{S}}(r\bm{u})|>0$, for some $\bm{u}\in \R^d$ with $|\bm{u}|=1$. Then, for any $C_*\in(0,ce^{-\lambda/2})$, Theorem~\ref{thm:conv_BM_limit}(c) implies that
$\mW_1(\bm{X}^t_1,\bm{\Sigma}\bm{B}_1) \ge C_*t^{1-\alpha/2}$ for all sufficiently small $t>0$.

Note that, as $\alpha$ approaches $1$, the gap between the lower and upper bound decreases. Indeed, for $\alpha=1$, we have $\beta_*=1+\ve$ for some small $\ve>0$, so the upper bound is of the rate $t^{1/(1+\ve)-1/2}$, while the lower bound has the rate $\sqrt{t}$, making the quotient of the two bounds proportional to $t^{\ve/(1+\ve)}$.
\end{example}

\begin{figure}[ht]
\centering
\includegraphics[width=.49\linewidth]{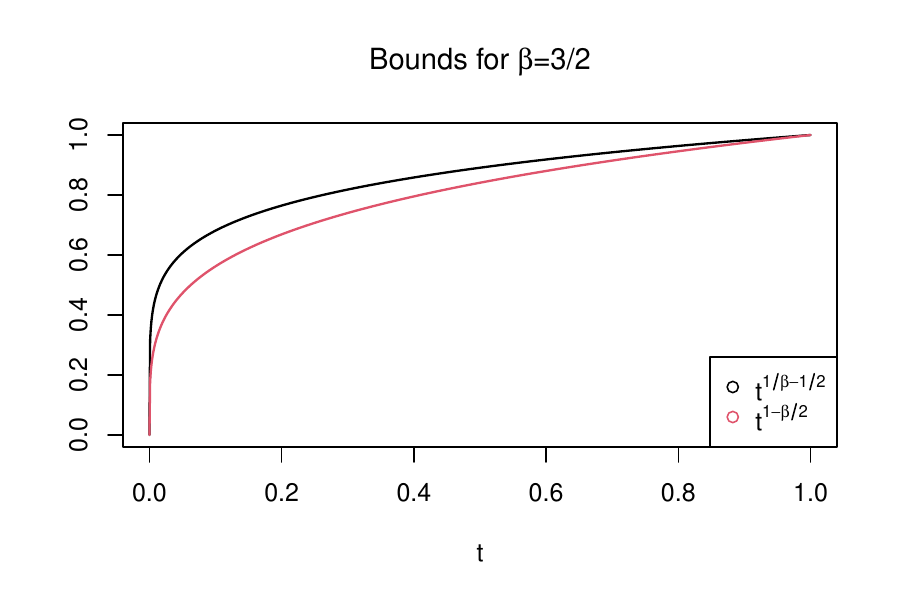}
\includegraphics[width=.49\linewidth]{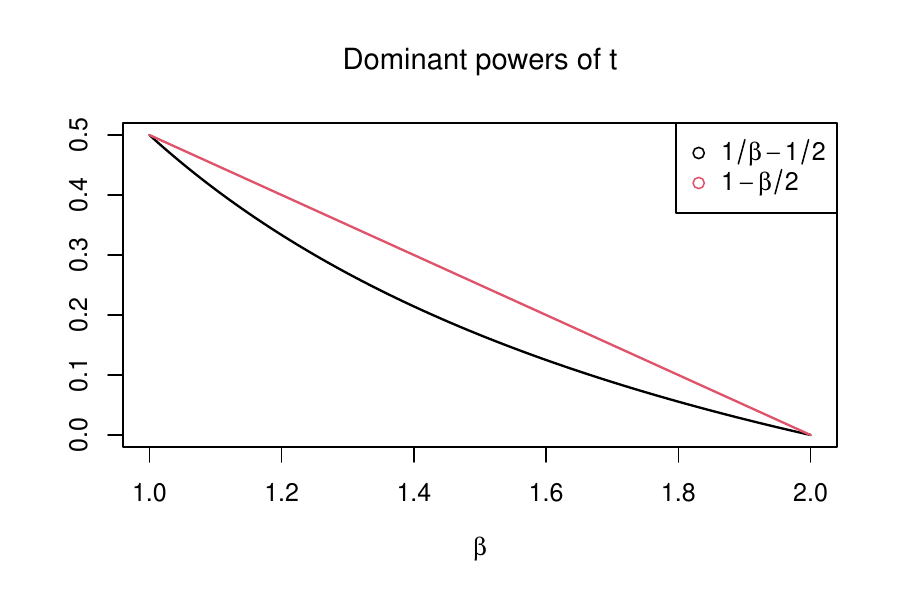}
\caption{\small The left picture shows the polynomials with the dominant powers from the upper and lower bounds for $\mW_1(\bm{X}_1^t,\bm{\Sigma}\bm{B}_1)$ from Theorem~\ref{thm:conv_BM_limit} with $\beta=3/2$. The right picture shows the dominant powers of $t$ in the upper and lower bounds as a function of $\beta\in [1,2]$.}
\label{fig:normal_rates}
\end{figure}

\subsection{Augmented \texorpdfstring{$\alpha$}{alpha}-stable processes}\label{subsec:augment_stable}

In this section we introduce the general class of \emph{augmented $\alpha$-stable processes} and apply  our main results from Sections~\ref{subsec:stable_normal_attraction} and~\ref{subsec:stable_non-normal_attraction}  to the augmented $\alpha$-stable processes in the stable domains of normal and non-normal attraction. We also show that the class of tempered $\alpha$-stable processes contains many other L\'evy processes, beyond the classical exponentially tempered stable processes.

\begin{defin}
\label{def:augmented-stable}    
A L\'evy process $\bm{X}=(\bm{X}_t)_{t \in [0,1]}$ is an augmented $\alpha$-stable process if it has no Gaussian component, and, for some $\alpha\in(0,2]$, its L\'evy measure $\nu_{\bm{X}}$ has the form
\begin{equation} \label{eq:temp_stab_alpha}
\nu_{\bm{X}}(A)
= \int_0^\infty\int_{\Sp^{d-1}}  \1_{A}(x\bm{v}) \mathcal{Q}(x,\bm{v})  \sigma(\D \bm{v})x^{-\alpha-1}\D x, \quad \text{ for } A \in \mathcal{B}(\R^d_{\bm{0}}),
\end{equation}
where $\sigma$ is a probability measure on $\mathcal{B}(\Sp^{d-1})$ and $\mathcal{Q}:(0,\infty)\times \Sp^{d-1} \mapsto [0,\infty)$ is a measurable satisfying $\mathcal{Q}(\cdot,\bm{v})\in\SV_0$ for $\bm{v}\in\Sp^{d-1}$. 
\end{defin}

Note that $\int_0^\infty (x^2\wedge 1)\mathcal{Q}(x,\bm{v})x^{-\alpha-1}\D x$ must be finite for $\sigma$-a.e. $\bm{v}\in\Sp^{d-1}$ for $\nu_{\bm{X}}$ to be a proper L\'evy measure. While the class of augmented stable processes could be extended beyond the requirement that $\mathcal{Q}(\cdot,\bm{v})\in\SV_0$ for all $\bm{v}\in\Sp^{d-1}$, our definition is motivated by the characterisation in Theorem~(\nameref{thm:small_time_domain_stable}) of the small-time domain of attraction of stable processes. Indeed, for any value of $\alpha\in(0,2)$, the key condition in~\eqref{eq:jump-stable-limit} of  Theorem~(\nameref{thm:small_time_domain_stable}) requires the limit $t\nu_{\bm{X}}(\scrL_{\bm{v}}(g(t)))\to \nu_{\bm{Z}}(\scrL_{\bm{v}}(1))$ (as $t\da 0$) to hold for all $\bm{v}\in\Sp^{d-1}$, where $G:t\mapsto t^{-1/\alpha}g(t)\in\SV_0$. In the case of augmented $\alpha$-stable processes, this is equivalent to 
\[
t\int_0^\infty\int_{\Sp^{d-1}}  \1_{\{\langle\bm{v},\bm{w}\rangle\ge g(t)/x\}}\mathcal{Q}(x,\bm{w}) \sigma(\D \bm{w})x^{-\alpha-1}\D x
\to \nu_{\bm{Z}}(\scrL_{\bm{v}}(1)), \quad \text{ as }t \da 0.
\]
Change of variables $y=x/g(t)$ 
yields 
$G(t)^{-\alpha} \int_{\Sp^{d-1}}\int_0^\infty  \1_{\{\langle\bm{v},\bm{w}\rangle\ge y\}}\mathcal{Q}(g(t) y,\bm{w}) y^{-\alpha-1}\D y \sigma(\D \bm{w}),$
suggesting a ``homogeneous'' slow variation of the function $x\mapsto\mathcal{Q}(x,\bm{v})$ in any direction (measured in terms of integrals), clearly linked to the slow variation of $G$. In Definition~\ref{def:augmented-stable} we strengthen this sense of slow variation by requiring $\mathcal{Q}$ to be slowly varying in every direction (but without imposing the slow variation to be ``the same'' in every direction).

\begin{corollary}\label{cor:augment_stable}
Assume that $\bm{X}$ is an augmented $\alpha$-stable process and $\bm{Z}$ is an $\alpha$-stable process with $\alpha\in(0,2)\setminus\{1\}$ and fix any $q\in(0,1]\cap(0,\alpha)$ with $\E[|\bm{X}_1|^q]<\infty$.
\begin{description}[leftmargin=0em, nosep]
\item[(a)]
Assume $|\mathcal{Q}(x,\bm{v})-c_\alpha| \le K(1\wedge x)$ for  some $K>0$ and all $x>0$, $\bm{v}\in \Sp^{d-1}$. Then, for $\bm{X}^t\coloneqq(\bm{X}_{st}/t^{1/\alpha})_{s\in[0,1]}$,  
\begin{equation*}
c_1 t^{1-q/\alpha}
\le\mW_q\big(\bm{X}^t_1,\bm{Z}_1\big)
\le\mW_q\big(\bm{X}^t,\bm{Z}\big)\le
c_2\big(t^{1-q/\alpha}\1_{\{\alpha<1\}} + t^{q(1-1/\alpha)}\1_{\{\alpha>1\}}\big),
\end{equation*}
for some $c_1,c_2>0$ and all sufficiently small $t>0$.
\item[(b)] Assume $\mathcal{Q}(x,\bm{v})\!=\! c_\alpha H(x)^\alpha$ for $x>0$, $\bm{v}\in \Sp^{d-1}$ and some monotone $H\in\SV_0\cap C^1$. Set $\varrho(x)\!\coloneqq \!\int_x^\infty H(y)^\alpha y^{-\alpha-1}\D y$ and let $\varrho^{\la}$ be its right-continuous inverse. Suppose $G(x)\coloneqq H(\varrho^{\la}(1/x))$ satisfies Assumption~(\nameref{asm:S}) for some function $G_2\in\SV_0$. Then, for $g(t)\coloneqq t^{1/\alpha}G(t)$, $a(t)\coloneqq G(2t)/G(t)$ and $\bm{X}^t\coloneqq(\bm{X}_{st}/g(t))_{s\in[0,1]}$, we have
\[
\mW_q(\bm{X}^t,\bm{Z})=\Oh ( G_2(t)^q),
\quad\text{as }t \da 0,\quad\text{and}\quad
\max\big\{\mW_q(\bm{X}^t_1,\bm{Z}_1)
,\mW_q(\bm{X}^{2t}_1,\bm{Z}_1)\big\}
\ge |1-a(t)^q|\E\big[|\bm{Z}_1|^q\big]/3,
\]
for sufficiently small $t >0$. The function $t\mapsto|1-a(t)^q|$ is not upper bounded by any non-decreasing function $\phi$ with $\int_0^1\phi(t)t^{-1}\D t<\infty$, $G$ is $C^1$ with $t\mapsto t|G'(t)|\in\SV_0$ and there exists some $\phi\in\SV_0\cap\SV_\infty$ satisfying the inequality $\phi(x)\ge\sup_{t>0,y\in[x\wedge 1,x\vee1]}yG'(yt)/G'(t)$ for all $x>0$. Moreover, defining  $L(t)\coloneqq t|G'(t)|/G(t)$, there exist constants $c_1,c_2>0$ such that for all sufficiently small $t>0$, we have
\[
c_2L(t)
\le\max\{\mW_q(\bm{X}^t_1,\bm{Z}_1)
,\mW_q(\bm{X}^{2t}_1,\bm{Z}_1)\}
\quad\text{and}\quad
\mW_q(\bm{X}^t,\bm{Z})
\le c_1L(t)^q.
\]

\end{description}
\end{corollary}

\begin{proof}
Part~(a). Assumption~(\nameref{asm:T}) is satisfied, since the radial decomposition of $\nu_{\bm{X}}^\co(\D \bm{w})$ is $$
c_\alpha^{-1}\mathcal{Q}(x,\bm{v})\nu_{\bm{Z}}(\D \bm{v},\D x)\text{ and $|\mathcal{Q}(x,\bm{v})-c_\alpha|\le K(1\wedge x^p)$ for all $x \in (0,\infty), \bm{v}\in \Sp^{d-1}$.}$$ Thus, the claim for $\alpha<1$ follows from Theorem~\ref{thm:upper_lower_simple}. When $\alpha>1$, we can also establish Assumption~(\nameref{asm:C}) easily, so Proposition~\ref{prop:normal_domain_comonotonic} gives the claim in this case.

Part~(b). The claims follow along similar lines from Theorem~\ref{thm:upper_lower_general} and Corollary~\ref{cor:upper_lower_general} via a simple verification of assumptions. In the notation of Theorem~\ref{thm:upper_lower_general}, $\rho^{\co\la}_{\bm{X}}(x,\bm{v})=c_\alpha\varrho^\la(x)$ does not depend on $\bm{v}\in \Sp^{d-1}$ and satisfies $x\mapsto x^{1/\alpha}\varrho^\la(x)\in\SV_\infty$. Indeed, $\varrho^\la(x) \sim (\alpha x)^{-1/\alpha} H(\varrho^\la(x))$ as $x\to\infty$ by Karamata's theorem~\cite[Thm~1.5.11]{MR1015093}, since $\varrho(x)\sim H(x)^\alpha x^{-\alpha}/\alpha$ as $x\da 0$  and $x\sim\varrho(\varrho^\la(x))$ as $x\to\infty$. Since $H\in\SV_0$, we have
\begin{align*}
G(x)
=H(\varrho^\la(1/x))
\sim H(c_\alpha \varrho^{\la}(1/x))
=\int_{\Sp^{d-1}}H(\rho^{\co\la}_{\bm{X}}(1/x,\bm{v})) \sigma(\D \bm{v}),
\qquad \text{as } x \da 0.
\end{align*}
Theorem~\ref{thm:upper_lower_general} then gives the main claims of Part (b). The remaining claims follow similarly from Corollary~\ref{cor:upper_lower_general}.
\end{proof}

\begin{remark}
\nf{(I)}
If for some $\lambda>1$ we have $(1-H(\lambda x)/H(x))\log H(x)\to 0$ as $x\da 0$, then~\cite[Thm~2.3.3~\&~Cor.~2.3.4]{MR1015093} imply the following simpler form of the slowly varying function $G$ in Corollary~\ref{cor:augment_stable}(b):  $G(x)\sim H(x^{-1/\alpha})$ as $x\da 0$.

\noindent \nf{(II)} We may extend Corollary~\ref{cor:augment_stable}(a) to cover any positive and measurable $\mathcal{Q}$ for which $|\mathcal{Q}(x,\bm{v})-c_\alpha|\le K(1\wedge x^p)$ for some $K,p>0$ and all $x>0$, $\bm{v}\in\Sp^{d-1}$ via Theorem~\ref{thm:d_thin_dom_attract} below. However, as with Theorem~\ref{thm:upper_lower_simple} above, here we only present the case $p=1$ for simplicity. Similarly, the strict conditions on $\mathcal{Q}$ in Corollary~\ref{cor:augment_stable}(b) could be relaxed, at the expense of a more complicated set of assumptions. Further note that, if $q=1$ in Corollary~\ref{cor:augment_stable}(a) or (b), then the lower and upper bounds are asymptotically proportional, so both are sharp up to a multiplicative constant.
\end{remark}

The relevance of the class of augmented $\alpha$-stable processes in our context is due to the readily  applicable  Corollary~\ref{cor:augment_stable}, together with the breadth of this class, demonstrated by the following example. 

\begin{example}[Classes of L\'evy processes that are augmented stable]
\label{ex:subclass_augmented_stable}
We describe the relationship between augmented stable processes and other classes of L\'evy processes used in the literature.
\begin{itemize}[leftmargin=1em, nosep]
\item \textbf{Rosi\'nski's tempered stable processes}~\cite[p.~680]{MR2327834} (which include the well-known exponentially tempered stable processes) admit the decomposition in~\eqref{eq:temp_stab_alpha} but require $\mathcal{Q}(\cdot,\bm{v})$ to be completely monotone (hence, strictly decreasing and convex) and $\mathcal{Q}(\infty,\bm{v})=0$ for all $\bm{v}\in\Sp^{d-1}$. For instance, a $\beta$-stable process is also a Rosi\'nski tempered $\alpha$-stable process if $\beta>\alpha$, making the index $\alpha$ non-unique in the definition~\cite[p.~680]{MR2327834}. Further, the index can actually change depending on the direction $\bm{v}$. On the other hand, if we require every direction to have the same index, by virtue of~\cite[Thm~1.5.4]{MR1015093}, most of these processes are also augmented stable processes. For instance, exponentially tempered stable processes are augmented stable with $\mathcal{Q}(x,\bm{v})=\exp(-x\lambda(\bm{v}))$ for some measurable $\lambda:\Sp^{d-1}\to[0,\infty)$. In fact, in defining augmented stable processes, we take Rosi\'nski's definition of a tempered stable process and replace the complete monotonicity and vanishing at infinity of $\mathcal{Q}(\cdot,\bm{v})$ with its slow variation at $0$. We note that the remaining classes of L\'evy processes in  Example~\ref{ex:subclass_augmented_stable} are typically not Rosi\'nski's tempered stable processes.
\item \textbf{Truncated stable processes} are augmented stable with $\mathcal{Q}(x,\bm{v})=\1_{\{x\le c(\bm{v})\}}$ for some measurable $c:\Sp^{d-1}\to (0,\infty)$ (note, for instance, that these are \emph{not} tempered stable under Rosi\'nski's definition).
\item The most straightforward extension of \textbf{Kuznetsov's $\bm{\beta}$-class of L\'evy processes}~\cite{MR2724421} to $\R^d$ with a common index in every direction is also augmented stable with $\mathcal{Q}(\cdot,\bm{v})=c(\bm{v})\e^{-\lambda(\bm{v})\beta(\bm{v})x}x^{-\alpha-1}(1-\e^{-\beta(\bm{v})x})^{-\alpha-1}$, for some positive measurable functions $c,\beta,\lambda:\Sp^{d-1}\to (0,\infty)$. 
\item Many \textbf{meromorphic L\'evy processes} on the real line~\cite{MR2977987} can be shown to be augmented stable. Indeed, by~\cite[Cor.~3]{MR2977987}, these processes admit a L\'evy density $\pi$ satisfying
\[
\pi(x)
=\1_{\{x>0\}}\sum_{n\in\N}a_n\rho_n\e^{-\rho_n x}
+\1_{\{x<0\}}\sum_{n\in\N}\hat a_n\hat\rho_n\e^{\hat\rho_n x},
\quad x\in\R\setminus\{0\},
\]
where the coefficients $\{a_n, \rho_n, \hat a_n, \hat\rho_n\}_{n\in\N}$ satisfy certain conditions (see details in~\cite[Cor.~3]{MR2977987}). If $\rho_n=\hat\rho_n=n$, by Karamata's Tauberian theorem~\cite[Cor.~1.7.3]{MR1015093}, the process will be augmented $\alpha$-stable if and only if $\sum_{n=1}^{\lfloor x\rfloor} na_n$ and $\sum_{n=1}^{\lfloor x\rfloor} n\hat a_n$ are both regularly varying at infinity with index $\alpha+1$. More generally, if 
\[
\lim_{x\searrow0}\frac{x\sum_{n\in\N}a_n\rho_n\e^{-\rho_n x}}{\sum_{n\in\N}a_n\e^{-\rho_n x}}
=\lim_{x\searrow0}\frac{x\sum_{n\in\N}\hat a_n\hat\rho_n\e^{-\hat\rho_n x}}{\sum_{n\in\N}\hat a_n\e^{-\hat\rho_n x}}
=\alpha,
\]
\end{itemize}
then the process is augmented $\alpha$-stable by Karamata's theorem~\cite[Thm~1.6.1]{MR1015093}.
\end{example}

\begin{example}[Application of Corollary~\ref{cor:augment_stable} to the (classical) exponentially tempered stable processes]\label{ex:exp_temp_stable}
Pick a bounded function $\lambda:\Sp^{d-1}\to[0,\infty)$ and define $\mathcal{Q}(x,\bm{v})\coloneqq e^{-\lambda(\bm{v})x}$ for all $x>0$ and $\bm{v}\in \Sp^{d-1}$. Then $\mathcal{Q}(x,\bm{v})$ satisfies the assumption in Corollary~\ref{cor:augment_stable}(a):
\begin{align*}
|e^{-\lambda(\bm{v})x}-1|
\le (1\wedge|x||\lambda(\bm{v})|) 
\le (1 \wedge |x|)\sup_{\bm{w}\in\Sp^{d-1}}|\lambda(\bm{w})|, 
\quad \text{for all } (x,\bm{v})\in (0,\infty)\times \Sp^{d-1}.
\end{align*} 
The augmented $\alpha$-stable process
$\bm{X}$, with L\'evy measure in~\eqref{eq:temp_stab_alpha} given by $\mathcal{Q}$, is exponentially tempered stable.
As usual, denote by $\bm{Z}$ the corresponding $\alpha$-stable process.

If $\alpha>1$, then Corollary~\ref{cor:augment_stable}(a) implies that $\mW_1(\bm{X}_t/t^{1/\alpha},\bm{Z}_1)=\Oh(t^{1-1/\alpha})$, with lower bound $\mW_1(\bm{X}_t/t^{1/\alpha},\bm{Z}_1) \ge c t^{1-1/\alpha}$ for some $c>0$ and all sufficiently small $t>0$. Thus, the upper and lower bounds have the same rate in this case, yielding a rate-optimal bound. 
\end{example}

Next, we give an example where the function $G$ is non-constant, and see how the rates deteriorate in these cases. 

\begin{example}[Augmented stable processes with arbitrarily slow convergence rate]
\label{ex:arbitrarily_slow}
Let $\bm{X}$ and $\bm{Z}$ be as in Corollary~\ref{cor:augment_stable}(b). Let $\ell_n$ be recursively defined as in Lemma~\ref{lem:iterated-log} below: $\ell_1(t)\coloneqq \log(e+t)$ and $\ell_{n+1}(t)=\log(e+\ell_n(t))$ for $t>0$. If either $H(x)=\ell_n(1/x)$ or $H(x)=\ell_n(1/x)^{-1}$ for some $n\in\N$ (i.e. $G(x)\sim\ell_n(x^{-1/\alpha})$ or $G(x)\sim\ell_n(x^{-1/\alpha})^{-1}$ as $x\to\infty$), then Lemma~\ref{lem:iterated-log} shows that for small $t>0$ we have $$G_2(t)\coloneqq \prod_{k=1}^n(e+\ell_k(1/t))^{-1}\ge |1-G(2t)/G(t)|/\log 2.$$ Moreover, by Lemma~\ref{lem:1-a_SV},  Assumption~(\nameref{asm:S}) holds with this $G_2$. Thus, by Corollary~\ref{cor:augment_stable}(b), there exist constants $0<C_1<C_2$ such that, for all small enough $t>0$, we have
\[
\frac{C_1}{\prod_{k=1}^n(e+\ell_k(1/t))}
\le\max\{\mW_1(\bm{X}_1^t,\bm{Z}),
    \mW_1(\bm{X}_1^{2t},\bm{Z})\}
\le \frac{C_2}{\prod_{k=1}^n(e+\ell_k(1/t))}.
\]
Thus, despite the function $\ell_n$ being ``nearly constant'' for large $n\in\N$, the convergence rates of the upper and lower bounds match and are slower than $\log(1/t)^{-1-\ve}$ for any $\ve>0$. 

Now consider any $\ell\in\SV_\infty$ with $\ell(t)\da 0$ as $t\to\infty$ and $\int_1^\infty \ell(t)t^{-1}\D t=\infty$. Then $G_\pm(t)\coloneqq \exp(\pm\int_1^t \ell(s)s^{-1}\D s)$ are slowly varying, $G_+(t)\to\infty$, $G_-(t)\to 0$ and $|1-G_\pm(t/2)/G_\pm(t)|\sim\ell(t)\log 2$ as $t\to\infty$ by Lemma~\ref{lem:1-a_SV}.
 Thus, by Corollary~\ref{cor:augment_stable}(b), for any $q\in(0,\alpha)\cap(0,1]$, we have $\mW_q(\bm{X}^t,\bm Z)\vee \mW_q(\bm{X}^{2t},\bm Z)\ge C^*\ell(1/t)$ for some $C^*>0$ and all sufficiently small $t>0$. In particular, by taking an appropriate $\ell$, e.g., $\ell(t)=1/\ell_n(t)$ where $\ell_n$ is as in the previous paragraph and $n\in\N$ is large, the convergence in Wasserstein distance may be arbitrarily slow.
\end{example}

\subsection{Classical bounds are hard to apply!}
\label{subsec:classical}
The couplings and methods used to achieve these bounds are crucial, and they differ significantly from the classical methods used to find rates of convergence. Indeed, if we tried to use standard methods (namely, the Berry--Esseen theorem or~\cite{MR3833470}) to construct bounds for small-time domain of attraction, the bounds would not converge as $t\da 0$. 

The Berry--Essen theorem exploits an increase in the activity of the process to obtain bounds on the distance between a random walk and the limit law. In the small-time regime, the activity is instead decreasing, explaining the unsuitability of this tool in this context (see details below). In fact, the bound would converge to infinity as $t \da 0$ (and in particular does not go to $0$) unlike the bounds introduced in this paper. The coupling in~\cite{MR3833470} couples corresponding components of the L\'evy--It\^o decompositions for a common small-jump cutoff level. When the time horizon is fixed and the L\'evy measure is supported on $[-\ve,\ve]$, the bounds of~\cite{MR3833470} are asymptotically sharp as $\ve\da 0$. However, for general L\'evy measures and as time tends to $0$, no time-dependent cutoff level $\ve_t$ can be used to obtain convergent bounds. The lack of convergent bounds in the small-time domain of attraction of stable processes is mainly caused by a difference in the jump intensities of the large-jump components (see details below).

We first explain why the Berry--Esseen theorem does not yield suitable bounds even when the limit is Gaussian (see~\cite{MR4321329,2021simulation}). For the explanation, it is enough to consider the one-dimensional case. Let $(X_t)_{t \in [0,1]}$ be a zero-mean L\'evy process on $\R$ with characteristic triplet $(\gamma,\sigma^2,\nu_X)$ (see~\cite[Def.~8.2]{MR3185174}) and finite fourth moment. The variance of $X$ is given by $\E[|X_t|^2]=(\sigma^2+\mu_2)t$, where $\mu_2 \coloneqq\int_{\R\setminus\{0\}} x^2 \nu_X(\D x)$. Then, $X_1^t=X_t/\sqrt{(\sigma^2+\mu_2)t}$ is attracted to a standard Gaussian random variable $Z$ as $t\da 0$. Denote by $\nu_t$ the L\'evy measure of $X_1^t$, the Berry--Esseen theorem thus implies that there exists some universal constant $C>0$, such that 
\[
\mW_2(X_1^t,Z)
\le C\int_{\R\setminus\{0\}} x^4\nu_t(\D x)
= C\int_{\R\setminus\{0\}} x^4t\nu_X(\D (\sqrt{(\sigma^2+\mu_2)t}x))
= \frac{C}{t(\sigma^2+\mu_2^2)^2}\int_{\R\setminus\{0\}} x^4\nu_X(\D x),
\] 
for all $t\in(0,1]$. As we can see above, this upper bound will tend to $\infty$ as $t \da 0$ and is therefore not an informative bound in the small-time regime. 

For L\'evy processes in the domain of attraction of an $\alpha$-stable law, an application of the bounds in~\cite{MR3833470} does not yield convergent bounds. The proofs of the bounds in~\cite{MR3833470} rely on the coupling of small jumps to a Gaussian law. Again, it is enough to consider the one-dimensional case. Let $X$ be symmetric and in the domain of attraction of the symmetric $\alpha$-stable random variable $Z$ with $\alpha>1$. Suppose their L\'evy measures satisfy $\nu_X(\R\setminus(-x,x))=x^{-\alpha}+x^{-(\alpha+1)/2}$ and $\nu_Z(\R\setminus(-x,x))=x^{-\alpha}$ for $x>0$. In this case we have $X^t_1=X_t/t^{1/\alpha}\cid Z$ as $t\da 0$. 

Let $\eta=(\alpha-1)/2>0$ and apply~\cite[Thm~11]{MR3833470} (at time $1$ and cutoff $\varepsilon_t$) to obtain: 
\begin{align*}
\mW_1(X^t_1,Z) 
&\le C\varepsilon_t 
    +\Big(\big(\tfrac{1}{2-\alpha}\varepsilon_t^{2-\alpha}
        +\tfrac{1}{3/2-\alpha/2}t^{\eta/\alpha}\varepsilon_t^{3\eta}\big)^{1/2}
        -\big(\tfrac{1}{2-\alpha}\varepsilon_t^{2-\alpha}\big)^{1/2}\Big)\\
&\qquad+ 2\big(\varepsilon_t^{-\alpha}+t^{\eta/\alpha}\varepsilon_t^{-\eta-1}\big)
    \int_{\varepsilon_t}^\infty \bigg|\frac{x^{-\alpha}}{\varepsilon_t^{-\alpha}}
        -\frac{x^{-\alpha}+x^{-(\alpha+1)/2}}{\varepsilon_t^{-\alpha}+t^{\eta/\alpha}\varepsilon_t^{-\eta-1}} \bigg| \D x\\
&\qquad+ 2 t^{\eta/\alpha}\varepsilon_t^{-\eta-1}\int_{\varepsilon_t}^\infty x\frac{\alpha x^{-\alpha-1}\D x}{\varepsilon_t^{-\alpha}}, \quad \text{ for all }t>0,
\end{align*}
where we used the formula for the $L^1$-Wasserstein distance in~\cite[p.~8]{Gibbs02onchoosing}. For the first line in the display above to vanish as $t\da 0$, we require $\varepsilon_t=\oh(1)$. The term in the middle line of the display above equals
\begin{align*}
2\int_{\varepsilon_t}^\infty \big|x^{-\alpha}t^{\eta/\alpha}\varepsilon_t^{-\eta}-x^{-\eta-1} \big| \D x 
&\ge 2\bigg|t^{\eta/\alpha}\varepsilon_t^{-\eta}\int_{\varepsilon_t}^c x^{-\alpha}\D x-\int_{\varepsilon_t}^c x^{-\eta-1} \D x \bigg|\\
&=2(\alpha-1)^{-1}\big|t^{\eta/\alpha}\varepsilon_t^{-3\eta}-t^{\eta/\alpha}\varepsilon_t^{-\eta}c^{-2\eta}
    -2(\varepsilon_t^{-\eta}-c^{-\eta}) \big|,
\end{align*}
where $c\in(0,\infty]$ is an arbitrary number and the inequality holds for all $t>0$ for which $\varepsilon_t<c$. For the right-hand side of the display to vanish at $t\da0$ with $c=\infty$ we must have $\varepsilon_t^{-\eta}(t^{\eta/\alpha}\varepsilon_t^{-2\eta}-2)=\oh(1)$. Then, for $c=1$, the display above will converge to $4/(\alpha-1)$. In particular, the bound implied by~\cite[Thm~11]{MR3833470} cannot vanish for any choice of $\varepsilon_t$.


\section{Two couplings of L\'evy processes}
\label{sec:couplings_general}
Let $\bm{X}=(\bm{X}_t)_{t \ge 0}$ be a L\'evy process on $\R^d$ with generating triplet $(\bm{\gamma_X},\bm{\Sigma_X}\bm{\Sigma_X}^\tra,\nu_{\bm{X}})$ (also called characteristic triplet, see~\cite[Def.~8.2]{MR3185174}) with respect to the cutoff function $\bm{w}\mapsto\1_{B_{\bm{0}}(1)}(\bm{w})$, where $\bm{\gamma_X} \in \R^d$, $\bm{\Sigma_X}\in\R^{d\times d}$ (with transpose $\bm{\Sigma_X}^\tra\in\R^{d\times d}$) and 
$\bm{\Sigma_X}\bm{\Sigma_X}^\tra$ a symmetric non-negative definite matrix
and $\nu_{\bm{X}}$ a L\'evy measure on $\R^d$. 
Throughout, we denote by $|\cdot|$ the Euclidean 
norm of appropriate dimension and recall that
$B_{\bm{0}}(r)=\{\bm{x}\in\R^d:|\bm{x}|<r\}$ is the open ball in $\R^d$ of radius $r>0$, centred at the origin $\bm{0}\in\R^d$. Fix any $\kappa\in(0,1]$ and consider the L\'evy--It\^o decomposition of $\bm{X}$ given by 
\begin{equation}\label{eq:levy_ito_decomp}
    \bm{X}_t=\bm{\gamma}_{\bm{X},\kappa} t+\bm{\Sigma_X}\bm{B^X}_t+\bm{D}^{\bm{X},\kappa}_t+\bm{J}^{\bm{X},\kappa}_t,\quad t\ge 0,
\end{equation}
where $\bm{\gamma}_{\bm{X},\kappa}\coloneqq\bm{\gamma}_{\bm{X}}-\int_{\R^d}\bm{w}\1_{B_{\bm{0}}(1)\setminus B_{\bm{0}}(\kappa)}(\bm{w})\nu_{\bm{X}}(\D\bm{w})$, $\bm{B^X}$ is a standard Brownian motion on $\R^d$, $\bm{D}^{\bm{X},\kappa}$ is the small-jump martingale containing all the jumps of $\bm{X}$ of magnitude less that $\kappa$, $\bm{J}^{\bm{X},\kappa}$ is the driftless compound Poisson process containing all the jumps of $\bm{X}$ of magnitude at least $\kappa$ and all three processes $\bm{B^X}$, $\bm{D}^{\bm{X},\kappa}$ and $\bm{J}^{\bm{X},\kappa}$ are independent. Moreover,
the pure-jump component $\bm{D}^{\bm{X},\kappa}+\bm{J}^{\bm{X},\kappa}$ of $\bm{X}$
is a L\'evy process with 
paths of finite variation (i.e. the jumps are summable on any compact time interval) if and only if $\int_{\R^d_{\bm{0}}}|\bm{w}| \1_{B_{\bm{0}}(1)}(\bm{w})\nu_{\bm{X}}(\D \bm{w})<\infty$~\cite[Thm~21.9]{MR3185174}. In particular, $(\bm{\gamma_X},\bm{0},\nu_{\bm{X}})$ is a characteristic triplet of a L\'evy process $\bm{X}$ without a Gaussian component. Thus, if $\bm{X}$ has finite variation, then $\bm{X}$ has \emph{zero natural drift} (i.e. the process equals the sum of its jumps) if and only if $\gamma_{\bm{X}}=\int_{\R^d_{\bm{0}}}\bm{w} \1_{B_{\bm{0}}(1)}(\bm{w})\nu_{\bm{X}}(\D \bm{w})$.

Similarly, we let $\bm{Y}=(\bm{Y}_t)_{t \ge 0}$ be a L\'evy process on $\R^d$ with characteristic triplet $(\bm{\gamma}_{\bm{Y}},\bm{\Sigma_Y}\bm{\Sigma_Y}^\tra,\nu_{\bm{Y}})$ with respect to the cutoff function $\bm{w} \mapsto \1_{B_{\bm{0}}(1)}(\bm{w})$ and whose corresponding L\'evy--It\^o decomposition is given by $\bm{Y}_t=\bm{\gamma}_{\bm{Y},\kappa} t+\bm{\Sigma_Y}\bm{B^Y}_t+\bm{D}^{\bm{Y},\kappa}_t+\bm{J}^{\bm{Y},\kappa}_t$, defined as above. 
The following elementary inequality will be used throughout:  for any $q\in (0,2]$,
\begin{equation}
\label{eq:Wq-XY}
\mW_q(\bm{X},\bm{Y})
\le |\bm{\gamma}_{\bm{X},\kappa}-\bm{\gamma}_{\bm{Y},\kappa}|^{q\wedge 1}  
    + (2\sqrt{d})^{q \wedge 1} |\bm{\Sigma_X}-\bm{\Sigma_Y}|^{q \wedge 1}
    + 
\mW_q\big(\bm{D}^{\bm{X},\kappa},\bm{D}^{\bm{Y},\kappa}\big) + \mW_q\big(\bm{J}^{\bm{X},\kappa},\bm{J}^{\bm{Y},\kappa}\big),
\end{equation}
where $|\cdot|$ in the last term denotes the Frobenius norm on $\R^{d\times d}$ (i.e.  $|\bm{\Sigma}|^2=\sum_{i,j=1}^n\Sigma_{i,j}^2$ for $\bm{\Sigma}\in\R^{d\times d}$). For completeness, we give a proof of~\eqref{eq:Wq-XY} in Appendix~\ref{app:A} below.

Let $\Xi_{\bm{X}}$ and $\Xi_{\bm{Y}}$ be the Poisson random measures on $[0,\infty)\times \R^d_{\bm{0}}$ of the jumps of $\bm{X}$ and $\bm{Y}$, respectively, with corresponding compensated measures $\wt\Xi_{\bm{X}}=\Xi_{\bm{X}}-\Leb\otimes\nu_{\bm{X}}$ and $\wt\Xi_{\bm{Y}}=\Xi_{\bm{Y}}-\Leb\otimes\nu_{\bm{Y}}$, where $\Leb$ denotes the Lebesgue measure on $[0,\infty)$. Since,  for every $t\ge0$, we have
\begin{align}\label{eq:comp_Poisson_measures}
\begin{aligned}
      \bm{D}^{\bm{X},\kappa}_t&=\int_{[0,t]\times \R^d_{\bm{0}}}\1_{B_{\bm{0}}(\kappa)}(\bm{w})\bm{w}\wt \Xi_{\bm{X}}(\D s, \D \bm{w}), \quad\bm{J}^{\bm{X},\kappa}_t=\int_{[0,t]\times \R^d_{\bm{0}}}\1_{\R^d \setminus B_{\bm{0}}(\kappa)}(\bm{w})\bm{w} \Xi_{\bm{X}}(\D s, \D \bm{w}),\\
    \bm{D}^{\bm{Y},\kappa}_t&=\int_{[0,t]\times \R^d_{\bm{0}}}\1_{B_{\bm{0}}(\kappa)}(\bm{w})\bm{w}\wt \Xi_{\bm{Y}}(\D s, \D \bm{w}), \quad\bm{J}^{\bm{Y},\kappa}_t=\int_{[0,t]\times \R^d_{\bm{0}}}\1_{\R^d \setminus B_{\bm{0}}(\kappa)}(\bm{w})\bm{w} \Xi_{\bm{Y}}(\D s, \D \bm{w}),
\end{aligned}
\end{align} 
the problem of coupling the jump components of $\bm{X}$ and $\bm{Y}$ is reduced to coupling the Poisson random measures $\Xi_{\bm{X}}$ and $\Xi_{\bm{Y}}$. Sections~\ref{sec:thinning} and~\ref{sec:comonotonic_coupling} below each describe such a coupling. 

\subsection{Thinning}
\label{sec:thinning}
Choose any L\'evy measure $\mu$ on $\R^d_{\bm{0}}$ that dominates both $\nu_{\bm{X}}$ and $\nu_{\bm{Y}}$ with Radon-Nikodym derivatives bounded by $1$ $\mu$-a.e., i.e. $f_{\bm{X}}=\D \nu_{\bm{X}}/\D\mu\le 1$ and $f_{\bm{Y}}=\D \nu_{\bm{Y}}/\D\mu\le 1$ $\mu$-a.e. For instance, a possible choice of $\mu$ is $\nu_{\bm{X}}+\nu_{\bm{Y}}$. 
Let $\Xi=\sum_{n\in\N} \delta_{(U_n,\bm{V}_n)}$ be a Poisson random measure on $(0,1]\times\R^d$, with mean 
measure $\Leb\otimes\mu$ and the corresponding compensated Poisson random measure $\wt\Xi(\D s, \D \bm{w})=\Xi(\D s, \D \bm{w})-\D s\otimes\mu(\D\bm{w})$. 
Assume the sequence $(\vartheta_n)_{n\in\N}$
of iid uniform random variables on $[0,1]$ is independent of 
$\Xi$. 
The Marking and Mapping Theorems~\cite{MR1207584} imply that the following Poisson random measures 
\begin{equation}\label{eq:thining_coupling_defn}
    \Xi_{\bm{X}}
=\sum_{n\in\N} \1_{\{\vartheta_n\le f_{\bm{X}}(\bm{V}_n)\}}\delta_{(U_n,\bm{V}_n)},
\qquad\text{and}\qquad
\Xi_{\bm{Y}}
=\sum_{n\in\N} \1_{\{\vartheta_n\le f_{\bm{Y}}(\bm{V}_n)\}}\delta_{(U_n,\bm{V}_n)},
\end{equation}
have mean measures $\Leb\otimes\nu_{\bm{X}}$ and $\Leb\otimes\nu_{\bm{Y}}$, respectively. We couple $\bm{X}$ and $\bm{Y}$ by choosing $\bm{B^X}=\bm{B^Y}$ in their L\'evy--It\^o decompositions and couple their jump parts from~\eqref{eq:comp_Poisson_measures} via the coupling of the Poisson random measures  in~\eqref{eq:thining_coupling_defn}.

\begin{proposition}
\label{prop:L2-thin}
The coupling $(\bm{D}^{\bm{X},\kappa},\bm{D}^{\bm{Y},\kappa},\bm{J}^{\bm{X},\kappa},\bm{J}^{\bm{Y},\kappa})$ defined in~\eqref{eq:comp_Poisson_measures} and~\eqref{eq:thining_coupling_defn} satisfies
\begin{equation}
\label{eq:L2-D-thin}
\E\bigg[\sup_{t\in[0,1]}\big|\bm{D}^{\bm{X},\kappa}_t-\bm{D}^{\bm{Y},\kappa}_t\big|^2\bigg]
\le 4\int_{\R^d_{\bm{0}}}\1_{B_{\bm{0}}(\kappa)}(\bm{w})|\bm{w}|^2|f_{\bm{X}}(\bm{w})-f_{\bm{Y}}(\bm{w})|
    \mu(\D\bm{w}).
\end{equation}
Moreover, if $\nu_{\bm{X}}(\D\bm{w})\1_{\R^d \setminus B_{\bm{0}}(\kappa)}(\bm{w})$ and $\nu_{\bm{Y}}(\D\bm{w})\1_{\R^d \setminus B_{\bm{0}}(\kappa)}(\bm{w})$ have a finite second moment, then 
\begin{gather*}
\E\bigg[\sup_{t\in[0,1]}\big|\bm{D}^{\bm{X},\kappa}_t+\bm{J}^{\bm{X},\kappa}_t-(\bm{D}^{\bm{Y},\kappa}_t+\bm{J}^{\bm{Y},\kappa}_t)-\bm{m}_\kappa t\big|^2\bigg]
\le 4\int_{\R^d_{\bm{0}}}|\bm{w}|^2|f_{\bm{X}}(\bm{w})-f_{\bm{Y}}(\bm{w})|
    \mu(\D\bm{w})\quad\text{and}\\
\mW_2(\bm{X},\bm{Y})
\le |\bm{\gamma}_{\bm{X},\kappa}-\bm{\gamma}_{\bm{Y},\kappa}+\bm{m}_\kappa|
    + 2 d^{1/2} |\bm{\Sigma_X}-\bm{\Sigma_Y}|
    + 2\bigg(\int_{\R^d_{\bm{0}}}|\bm{w}|^2|f_{\bm{X}}(\bm{w})-f_{\bm{Y}}(\bm{w})|
    \mu(\D\bm{w})\bigg)^{1/2},
\end{gather*}
where the mean $\bm{m}_\kappa:=\E[\bm{J}^{\bm{X},\kappa}_1-\bm{J}^{\bm{Y},\kappa}_1]
=\int_{\R^d_{\bm{0}}}\bm{w}\1_{\R^d \setminus B_{\bm{0}}(\kappa)}(\bm{w})
(f_{\bm{X}}(\bm{w})-f_{\bm{Y}}(\bm{w}))\mu(\D\bm{w})$
is finite.
\end{proposition}

\begin{proof} Denote  $f^+:=\max\{0,f\}$ for any function mapping into $\R$.
Define the Poisson random measures 
\begin{equation}\label{eq:lambda_pm}
    \Lambda_+\coloneqq\sum_{n\in\N} \1_{\{f_{\bm{Y}}(\bm{V}_n)<\vartheta_n\le f_{\bm{X}}(\bm{V}_n)\}}\delta_{(U_n,\bm{V}_n)}
\qquad\text{and}\qquad
\Lambda_-\coloneqq\sum_{n\in\N} \1_{\{f_{\bm{X}}(\bm{V}_n)<\vartheta_n\le f_{\bm{Y}}(\bm{V}_n)\}}\delta_{(U_n,\bm{V}_n)},
\end{equation}
with mean measures $\Leb\otimes(f_{\bm{X}}-f_{\bm{Y}})^+\mu$ and $\Leb\otimes(f_{\bm{Y}}-f_{\bm{X}})^+\mu$, respectively. Thus  $\Xi_{\bm{X}}-\Xi_{\bm{Y}}=\Lambda_+-\Lambda_-$. Note that $\Lambda_+$ is independent of $\Lambda_-$ since they are both thinnings of the same Poisson random measure and have disjoint supports. Let $\wt\Lambda_+$ and $\wt\Lambda_-$ denote their respective compensated Poisson random measures and define the L\'evy processes $\bm{D^\pm}=(\bm{D}^\pm_t)_{t\ge0}$ by $\bm{D}^\pm_t:=\int_{(0,t]\times\R^d_{\bm{0}}}\bm{w}\1_{B_{\bm{0}}(\kappa)}(\bm{w})\wt\Lambda_\pm,(\D s, \D \bm{w})$, where $\pm\in\{+,-\}$. By construction, $\bm{D}^+$ and $\bm{D}^-$ are independent square-integrable martingales, satisfying
$\E\big[\bm{D}^+_t\big]=\E\big[\bm{D}^-_t\big]=0$ and
$\bm{D}^{\bm{X},\kappa}_t-\bm{D}^{\bm{Y},\kappa}_t=\bm{D}^+_t-\bm{D}^-_t$ for all $t\in\R_+$. In particular, we have $\E\big[\langle \bm{D}^+_t, \bm{D}^-_t\rangle \big]=0$ and, by Campbell's formula~\cite[p.~28]{MR1207584}, 
\begin{align*}  
\E\left[|\bm{D}^+_t|^2\right]&=t\int_{\R^d_{\bm{0}}}\1_{B_{\bm{0}}(\kappa)}(\bm{w})
    |\bm{w}|^2(f_{\bm{X}}(\bm{w})-f_{\bm{Y}}(\bm{w}))^+
    \mu(\D\bm{w}),\\ 
    \E\left[|\bm{D}^-_t|^2\right]&=t \int_{\R^d_{\bm{0}}}\1_{B_{\bm{0}}(\kappa)}(\bm{w})
    |\bm{w}|^2(f_{\bm{Y}}(\bm{w})-f_{\bm{X}}(\bm{w}))^+
    \mu(\D\bm{w}).
\end{align*}
Doob's maximal inequality~\cite[Prop.~7.16]{MR1876169}, applied to the submartingale 
$|\bm{D}^+-\bm{D}^-|$,
and the independence of martingales $\bm{D}^+$ and $\bm{D}^-$
yield
\begin{align*}   
\E\bigg[\sup_{t\in[0,1]}\big|\bm{D}^{\bm{X},\kappa}_t-\bm{D}^{\bm{Y},\kappa}_t\big|^2\bigg] & = \E\bigg[\sup_{t\in[0,1]}\big|\bm{D}^+_t-\bm{D}^-_t\big|^2\bigg]\le 4\E\big[\big|\bm{D}^+_1-\bm{D}^-_1\big|^2\big]
=4\E\big[\big|\bm{D}^+_1|^2\big]+4\E\big[\big|\bm{D}^-_1|^2\big] \\
& =4\int_{\R^d_{\bm{0}}}\1_{B_{\bm{0}}(\kappa)}(\bm{w})
    |\bm{w}|^2|f_{\bm{X}}(\bm{w})-f_{\bm{Y}}(\bm{w})|
    \mu(\D\bm{w}).
\end{align*}

Assume, that $\int_{\R^d_{\bm{0}}}|\bm{w}|^2\1_{\R^d \setminus B_{\bm{0}}(\kappa)}(\bm{w})\nu_{\bm{X}}(\D \bm{w})<\infty$ and $\int_{\R^d_{\bm{0}}}|\bm{w}|^2\1_{\R^d \setminus B_{\bm{0}}(\kappa)}(\bm{w})\nu_{\bm{Y}}(\D \bm{w})<\infty$,
and define the L\'evy processes $\bm{J}^\pm=(\bm{J}^\pm_t)_{t\ge0}$ by
$\bm{J}^\pm_t:=\int_{(0,t]\times\R^d_{\bm{0}}}\bm{w}\1_{\R^d \setminus B_{\bm{0}}(\kappa)}(\bm{w})\Lambda_\pm,(\D s, \D \bm{w})$.
By the integrability assumption and construction, $\bm{J}^+$ and $\bm{J}^-$ are independent square-integrable processes with
$\E\big[\bm{J}^+_t-\bm{J}^-_t\big]=t\bm{m}_\kappa$ and
$\bm{J}^{\bm{X},\kappa}_t-\bm{J}^{\bm{Y},\kappa}_t=\bm{J}^+_t-\bm{J}^-_t$ for all $t\in\R_+$. Thus, Campbell's formula~\cite[p.~28]{MR1207584} 
yields
\begin{align*}
    \E\big[|\bm{D}^+_t+\bm{J}^+_t-\E[\bm{J}^+_t]|^2\big]&=t\int_{\R^d_{\bm{0}}}|\bm{w}|^2(f_{\bm{X}}(\bm{w})-f_{\bm{Y}}(\bm{w}))^+\mu(\D\bm{w}),\\
    \E\big[|\bm{D}^-_t+\bm{J}^-_t-\E[\bm{J}^-_t]|^2\big]&=t\int_{\R^d_{\bm{0}}}|\bm{w}|^2(f_{\bm{Y}}(\bm{w})-f_{\bm{X}}(\bm{w}))^+
    \mu(\D\bm{w}).
\end{align*} Next, Doob's maximal inequality applied to the submartingale $|\bm{D}^+_t+\bm{J}^+_t-(\bm{D}^-_t+\bm{J}^-_t)-t\bm{m}_\kappa|$, and the independence between $\bm{D}^+_t+\bm{J}^+_t-\E[\bm{J}^+_t]$ and $\bm{D}^-_t+\bm{J}^-_t-\E[\bm{J}^-_t]$, yield
\begin{equation*}
\begin{aligned}
\E\bigg[\sup_{t\in[0,1]}\big|\bm{D}^{\bm{X},\kappa}_t+\bm{J}^{\bm{X},\kappa}_t-\big(\bm{D}^{\bm{Y},\kappa}_t+\bm{J}^{\bm{Y},\kappa}_t\big)-\bm{m}_\kappa t\big|^2\bigg]&=\E\bigg[\sup_{t\in[0,1]}\big|\bm{D}^+_t+\bm{J}^+_t-\big(\bm{D}^-_t+\bm{J}^-_t\big)-\bm{m}_\kappa t\big|^2\bigg]\\
&\le 4\E\big[\big|\bm{D}^+_1+\bm{J}^+_1-(\bm{D}^-_1+\bm{J}^-_1)-\bm{m}_\kappa\big|^2\big]\\
&=4\int_{\R^d_{\bm{0}}}|\bm{w}|^2|f_{\bm{X}}(\bm{w})-f_{\bm{Y}}(\bm{w})|
    \mu(\D\bm{w}),
\end{aligned}
\end{equation*}
completing the proof.
\end{proof}

The following bound is required when the big jump components have infinite variance. 

\begin{proposition}
\label{prop:Lq-J-thin}
Consider the coupling $(\bm{D}^{\bm{X},\kappa},\bm{D}^{\bm{Y},\kappa},\bm{J}^{\bm{X},\kappa},\bm{J}^{\bm{Y},\kappa})$ defined in~\eqref{eq:comp_Poisson_measures} and~\eqref{eq:thining_coupling_defn}. Then
\begin{equation}
\label{eq:Lq-J}
\E\bigg[\sup_{t\in[0,1]}\big|\bm{J}^{\bm{X},\kappa}_t-\bm{J}^{\bm{Y},\kappa}_t\big|^q\bigg]
\le \int_{\R^d_{\bm{0}}}\1_{\R^d \setminus B_{\bm{0}}(\kappa)}(\bm{w})
    |\bm{w}|^q|f_{\bm{X}}(\bm{w})-f_{\bm{Y}}(\bm{w})|\mu(\D\bm{w}),
\,\text{for any }q\in(0,1].
\end{equation}
In particular, the following inequality holds for any $q\in(0,1]$:
\begin{equation}
\label{eq:Lq-X}
\begin{split}
\mW_q(\bm{X},\bm{Y})
&\le 
|\bm{\gamma}_{\bm{X},\kappa}-\bm{\gamma}_{\bm{Y},\kappa}|^q
+\bigg(4\int_{\R^d_{\bm{0}}}\1_{B_{\bm{0}}(\kappa)}(\bm{w})
    |\bm{w}|^2|f_{\bm{X}}(\bm{w})-f_{\bm{Y}}(\bm{w})|\mu(\D\bm{w})\bigg)^{q/2}\\
&\qquad + 2^qd^{q/2}|\bm{\Sigma_X}-\bm{\Sigma_Y}|^q
+\int_{\R^d_{\bm{0}}}\1_{\R^d \setminus B_{\bm{0}}(\kappa)}(\bm{w})
    |\bm{w}|^q|f_{\bm{X}}(\bm{w})-f_{\bm{Y}}(\bm{w})|\mu(\D\bm{w}).
    \end{split}
\end{equation}
\end{proposition}
Note that~\eqref{eq:Lq-J} \&~\eqref{eq:Lq-X} hold  without assuming $\nu_{\bm{X}}(\D\bm{w})\1_{\R^d \setminus B_{\bm{0}}(\kappa)}(\bm{w})$ and $\nu_{\bm{Y}}(\D\bm{w})\1_{\R^d \setminus B_{\bm{0}}(\kappa)}(\bm{w})$ have a finite $q$-moment. If this holds, however, the bound is non-trivial because the big jumps in Proposition~\ref{prop:Lq-J-thin} are then 
 controlled by  $\int_{\R^d_{\bm{0}}}\1_{\R^d \setminus B_{\bm{0}}(\kappa)}(\bm{w})|\bm{w}|^q\nu_{\bm{X}}(\D\bm{w})+\int_{\R^d_{\bm{0}}}\1_{\R^d \setminus B_{\bm{0}}(\kappa)}(\bm{w})|\bm{w}|^q\nu_{\bm{Y}}(\D\bm{w})<\infty$. 
 
\begin{proof}
Recall that $\kappa \in (0,1]$ and let $\kappa' \in (\kappa,\infty)$. For $t \ge 0$,
let
$$\text{$\bm{J}^{\bm{X},\kappa}_{t,\kappa'}\coloneqq\int_{[0,t]\times \R^d_{\bm{0}}}\1_{B_{\bm{0}}(\kappa')\setminus B_{\bm{0}}(\kappa)}(\bm{w})\bm{w} \Xi_{\bm{X}}(\D s, \D \bm{w})$ and  $\bm{J}^{\bm{Y},\kappa}_{t,\kappa'}\coloneqq\int_{[0,t]\times \R^d_{\bm{0}}}\1_{B_{\bm{0}}(\kappa')\setminus B_{\bm{0}}(\kappa)}(\bm{w})\bm{w} \Xi_{\bm{Y}}(\D s, \D \bm{w})$.}$$ Let $\Lambda_\pm$ be as in \eqref{eq:lambda_pm}, and define $\bm{J}^\pm_{\cdot,\kappa'}=(\bm{J}^\pm_{t,\kappa'})_{t \ge 0}$ as $\bm{J}^\pm_{t,\kappa'}
\coloneqq\int_{(0,t]\times\R^d_{\bm{0}}}\bm{w}\1_{B_{\bm{0}}(\kappa')\setminus B_{\bm{0}}(\kappa)}(\bm{w})\Lambda_\pm(\D s, \D \bm{w})$ (both being L\'evy processes), and note that $\bm{J}^+_{t,\kappa'}-\bm{J}^-_{t,\kappa'}=\bm{J}^{\bm{X},\kappa}_{t,\kappa'}-\bm{J}^{\bm{Y},\kappa}_{t,\kappa'}$ for $t \in \R_+$ and $\kappa' \in (\kappa,\infty)$. Note that $\nu_{\bm{X}}(\D\bm{w})\1_{B_{\bm{0}}(\kappa')\setminus B_{\bm{0}}(\kappa)}(\bm{w})$ and $\nu_{\bm{Y}}(\D\bm{w})\1_{B_{\bm{0}}(\kappa')\setminus B_{\bm{0}}(\kappa)}(\bm{w})$ have a finite $q$-moment for all $\kappa' \in (\kappa,\infty)$. Moreover, by the triangle inequality and the 
fact $(x+y)^q\le x^q+y^q$ for all $x,y\ge0$, we have
\begin{align}\label{eq:bound_campbells_large_jumps}
\begin{split}
\sup_{t\in[0,1]}\big|\bm{J}^+_{t,\kappa'}-\bm{J}^-_{t,\kappa'}\big|^q 
&\le \int_{(0,1]\times\R^d_{\bm{0}}}\1_{B_{\bm{0}}(\kappa')\setminus B_{\bm{0}}(\kappa)}(\bm{w})|\bm{w}|^q(\Lambda_+(\D s, \D \bm{w})+\Lambda_-(\D s, \D \bm{w})).
\end{split}
\end{align}
Recall that $\Leb\otimes(f_{\bm{X}}-f_{\bm{Y}})^+\mu$ and $\Leb\otimes(f_{\bm{Y}}-f_{\bm{X}})^+\mu$ are the mean measures of $\Lambda_+$ and $\Lambda_-$, respectively. Thus, by taking expectations in~\eqref{eq:bound_campbells_large_jumps} and applying Campbell's formula~\cite[p.~28]{MR1207584}, we get
\begin{align*}
    \E\bigg[\sup_{t\in[0,1]}\big|\bm{J}^{\bm{X},\kappa}_{t,\kappa'}-\bm{J}^{\bm{Y},\kappa}_{t,\kappa'}\big|^q\bigg]&=\E\bigg[\sup_{t\in[0,1]}\big|\bm{J}^+_{t,\kappa'}-\bm{J}^-_{t,\kappa}\big|^q\bigg]\\
&\le \int_{\R^d_{\bm{0}}}\1_{B_{\bm{0}}(\kappa')\setminus B_{\bm{0}}(\kappa)}(\bm{w})
    |\bm{w}|^q|f_{\bm{X}}(\bm{w})-f_{\bm{Y}}(\bm{w})|\mu(\D\bm{w}).
\end{align*}
Due to the monotone convergence theorem, it follows that, as $\kappa' \to \infty$, 
\begin{equation*}
    \int_{B_{\bm{0}}(\kappa')\setminus B_{\bm{0}}(\kappa)}
    |\bm{w}|^q|f_{\bm{X}}(\bm{w})-f_{\bm{Y}}(\bm{w})|\mu(\D\bm{w}) \to \int_{\R^d_{\bm{0}}}\1_{\R^d \setminus B_{\bm{0}}(\kappa)}(\bm{w})
    |\bm{w}|^q|f_{\bm{X}}(\bm{w})-f_{\bm{Y}}(\bm{w})|\mu(\D\bm{w}).
\end{equation*} 
Furthermore, Fatou's lemma together with the above observations imply that
\begin{align*}
    \E\bigg[\liminf_{\kappa'\to\infty}\sup_{t\in[0,1]}\big|\bm{J}^{\bm{X},\kappa}_{t,\kappa'}-\bm{J}^{\bm{Y},\kappa}_{t,\kappa'}\big|^q\bigg]&\le \liminf_{\kappa' \to \infty}\E\bigg[\sup_{t\in[0,1]}\big|\bm{J}^{\bm{X},\kappa}_{t,\kappa'}-\bm{J}^{\bm{Y},\kappa}_{t,\kappa'}\big|^q\bigg] \\
    &\le \int_{\R^d_{\bm{0}}}\1_{\R^d \setminus B_{\bm{0}}(\kappa)}(\bm{w})
    |\bm{w}|^q|f_{\bm{X}}(\bm{w})-f_{\bm{Y}}(\bm{w})|\mu(\D\bm{w}).
\end{align*}
We have $\liminf_{\kappa'\to\infty}\sup_{t\in[0,1]}\big|\bm{J}^{\bm{X},\kappa}_{t,\kappa'}-\bm{J}^{\bm{Y},\kappa}_{t,\kappa'}\big|^q=\sup_{t\in[0,1]}\big|\bm{J}^{\bm{X},\kappa}_{t}-\bm{J}^{\bm{Y},\kappa}_{t}\big|^q$ a.s., 
since the largest jump of $\bm{J}^{\bm{X},\kappa}$ and $\bm{J}^{\bm{Y},\kappa}$ are finite on the time interval $[0,1]$. 
This implies~\eqref{eq:Lq-J}.

Since 
$\mW_q(\bm{D}^{\bm{X}},\bm{D}^{\bm{Y}})\le\mW_2(\bm{D}^{\bm{X}},\bm{D}^{\bm{Y}})^{q}$, the inequality in~\eqref{eq:Lq-X} follows from~\eqref{eq:Wq-XY}, \eqref{eq:L2-D-thin} and~\eqref{eq:Lq-J}. 
\end{proof}

\subsection{Comonotonic coupling}
\label{sec:comonotonic_coupling}

In this section, we introduce the $d$-dimensional comonotonic coupling of jumps for any $d \ge 1$. 
We use two ingredients to construct this coupling of the L\'evy processes $\bm{X}$ and $\bm{Y}$: \textbf{(I)} the comonotonic coupling of real-valued random variables $\xi$ and $\zeta$, given by  $(\xi,\zeta)=(F_\xi^{\la}(U),F_\zeta^{\la}(U))$, where 
 $U$ is uniform on $(0,1)$ and the functions $F_\xi^{\la}$ and $F_\zeta^\la$ are the right inverses of the functions $F_\xi$ and $F_\zeta$;
 \textbf{(II)} LaPage's representation of the Poisson random measures of a L\'evy process (see~\cite[p.~4]{MR1833707}).

The comonotonic coupling of the real-valued variables in \textbf{(I)} is optimal for the $L^p$-Wasserstein distance (see~\cite[Ex.~3.2.14]{MR1619170}),
$\mW_p(\xi,\zeta)^p
=\int_0^1|F_\xi^{\la}(u)-F_\zeta^{\la}(u)|^p\D u
=\E\big[|F_\xi^{\la}(U)-F_\zeta^{\la}(U)|^p\big]$ for $p\ge 1$.
The representation in~\textbf{(II)} decomposes the jumps of a L\'evy process into its magnitude (i.e. norm) and angular component. The main idea behind our coupling of the L\'evy processes $\bm{X}$ and $\bm{Y}$
is to couple their respective Poisson random measures of jumps via a comonotonic coupling of the magnitudes of jumps, while simultaneously aligning their angular components. We now describe this construction.

Recall that the L\'evy processes $\bm{X}$ and $\bm{Y}$ in $\R^d$ have characteristic triplets $(\bm{\gamma_X},\bm{\Sigma_X}\bm{\Sigma_X}^\tra,\nu_{\bm{X}})$ and $(\bm{\gamma_Y},\bm{\Sigma_Y}\bm{\Sigma_Y}^\tra,\nu_{\bm{Y}})$, respectively. Suppose the L\'evy measure $\nu_{\bm{X}}$ (resp. $\nu_{\bm{Y}}$) of $\bm{X}$ (resp. $\bm{Y}$) admits a radial decomposition (see~\cite[p.~282]{MR647969}): there exists a probability measure $\sigma_{\bm{X}}$ (resp. $\sigma_{\bm{Y}}$) on the unit sphere $\Sp^{d-1}$ (with convention $\Sp^{0}\coloneqq\{-1,1\}$) such that:
\[
\nu_{\bm{X}}(B)= \int_{\Sp^{d-1}}\int_0^\infty \1_{B}(x\bm{v})\rho_{\bm{X}}^0(\D x,\bm{v})\sigma_{\bm{X}}(\D \bm{v}),\,\,\bigg(\text{resp. }
\nu_{\bm{Y}}(B)= \int_{\Sp^{d-1}}\int_0^\infty \1_{B}(x\bm{v})\rho_{\bm{Y}}^0(\D x,\bm{v})\sigma_{\bm{Y}}(\D \bm{v})\bigg),
\]
for any $B\in \mathcal{B}(\R^d \setminus \{\bm{0}\})$, where $\{\rho^0_{\bm{X}}(\cdot,\bm{v})\}_{\bm{v}\in \Sp^{d-1}}$ (resp. $\{\rho^0_{\bm{Y}}(\cdot,\bm{v})\}_{\bm{v}\in \Sp^{d-1}}$) is a measurable family of L\'evy measures on $(0,\infty)$. Define the probability measure $\sigma \coloneqq (\sigma_{\bm{X}}+\sigma_{\bm{Y}})/2$ on $\Sp^{d-1}$ and the Radon-Nikodym derivatives $f^\sigma_{\bm{X}}(\bm{v})\coloneqq \sigma_{\bm{X}}(\D\bm{v})/\sigma(\D\bm{v})\le 2$ and $f^\sigma_{\bm{Y}}(\bm{v})\coloneqq \sigma_{\bm{Y}}(\D\bm{v})/\sigma(\D\bm{v})\le 2$ for $\bm{v} \in \Sp^{d-1}$. Consider the following radial decompositions of $\nu_{\bm{X}}$ and $\nu_{\bm{Y}}$: 
\begin{equation}\label{eq:radial_decomp}
\nu_{\bm{X}}(B)= \int_{\Sp^{d-1}}\int_0^\infty \1_{B}(x\bm{v})\rho_{\bm{X}}(\D x,\bm{v})\sigma(\D \bm{v}), \quad
\nu_{\bm{Y}}(B)= \int_{\Sp^{d-1}}\int_0^\infty \1_{B}(x\bm{v})\rho_{\bm{Y}}(\D x,\bm{v})\sigma(\D \bm{v}),
\end{equation} 
for $B\in \mathcal{B}(\R^d \setminus \{\bm{0}\})$, where $\rho_{\bm{X}}(\cdot,\bm{v})\coloneqq f^\sigma_{\bm{X}}(\bm{v})\rho_{\bm{X}}^0(\cdot,\bm{v})$ and $\rho_{\bm{Y}}(\cdot,\bm{v})\coloneqq f^\sigma_{\bm{Y}}(\bm{v})\rho_{\bm{Y}}^0(\cdot,\bm{v})$ for $\bm{v} \in \Sp^{d-1}$. The advantage of the decomposition in~\eqref{eq:radial_decomp}, compared to the one in the display above, is that the angular components of jumps are sampled from the same measure $\sigma$ on $\Sp^{d-1}$, making it possible to couple the jumps of $\bm{X}$ and $\bm{Y}$
by coupling their magnitudes. 

For every $\bm{v}\in \Sp^{d-1}$, let $u\mapsto\rho_{\bm{X}}^{\la}(u,\bm{v})$ (resp. $u\mapsto\rho_{\bm{Y}}^{\la}(u,\bm{v})$) be the right inverse of $x\mapsto\rho_{\bm{X}}([x,\infty),\bm{v})$ (resp. $x\mapsto\rho_{\bm{Y}}([x,\infty),\bm{v})$). Let $(U_n)_{n \in \N}$ be a sequence of iid uniform random variables on $[0,1]$, and let $(\Gamma_n)_{n\in \N}$ be a sequence of partial sums of iid standard exponentially distributed random variables that is independent of $(U_n)_{n \in \N}$. Next, independent of $(U_n,\Gamma_n)_{n \in \N}$, we denote by $(\bm{V}_n)_{n \in \N}$ a sequence of iid random vectors on $\Sp^{d-1}$ with common distribution $\sigma$. Define the Poisson point process $\Xi$ on $[0,1]\times (0,\infty)\times \Sp^{d-1}$ with measure $\Leb\otimes\Leb\otimes\sigma$ and the compensated Poisson random measure $\wt\Xi(\D s,\D x, \D \bm{v})$ as follows:
\begin{equation}\label{eq:marked_PPP_Comonotonic}
    \Xi\coloneqq \sum_{n \in \N}\delta_{(U_n,\Gamma_n,\bm{V}_n)}, \qquad \wt\Xi(\D s,\D x,\D \bm{v})=\Xi(\D s, \D x, \D \bm{v})-\D s \otimes \D x \otimes \sigma(\D \bm{v}).
\end{equation} 

Next, we note that (by Proposition~\ref{prop:small_jump_bound_comonotonic} below) for any $\ve\in(0,\infty)$ (and even $\ve=\infty$ when $\bm{X}$ and $\bm{Y}$ both have jumps of finite variation), the small-jump components of $\bm{X}$ and $\bm{Y}$ take the form
\begin{equation}\label{eq:comono_coupling_1}    \bm{M^X}_t\!\coloneqq\!\!\int_{[0,t]\times [\ve,\infty)\times \Sp^{d-1}}
\!\!\bm{v}\rho_{\bm{X}}^{\la}(x,\bm{v})\wt \Xi(\D s, \D x, \D \bm{v}), \enskip 
\bm{M^Y}_t\!\coloneqq\!\!\int_{[0,t]\times [\ve,\infty)\times \Sp^{d-1}}
\!\!\bm{v}\rho_{\bm{Y}}^{\la}(x,\bm{v})\wt \Xi(\D s, \D x, \D \bm{v}).
\end{equation}
The big-jump components of $\bm{X}$ and $\bm{Y}$ can similarly be expressed as
\begin{equation}\label{eq:comono_coupling_2}
    \bm{L^X}_t\coloneqq \int_{[0,t]\times (0,\ve)\times \Sp^{d-1}}\bm{v}\rho_{\bm{X}}^{\la}(x,\bm{v}) \Xi(\D s, \D x, \D \bm{v}), \enskip \bm{L^Y}_t\coloneqq \int_{[0,t]\times (0,\ve)\times \Sp^{d-1}}\bm{v}\rho_{\bm{Y}}^{\la}(x,\bm{v}) \Xi(\D s, \D x, \D \bm{v}).
\end{equation}

\begin{proposition}\label{prop:small_jump_bound_comonotonic}
Let L\'evy processes 
$\bm{X}$ and $\bm{Y}$
have characteristic triplets denoted by $(\bm{\gamma_X},\bm{\Sigma_X}\bm{\Sigma_X}^\tra,\nu_{\bm{X}})$ and $(\bm{\gamma_Y},\bm{\Sigma_Y}\bm{\Sigma_Y}^\tra,\nu_{\bm{Y}})$, respectively.
Assume that 
 the L\'evy measures of $\nu_{\bm{X}}$ and $\nu_{\bm{Y}}$ admit the radial decomposition in~\eqref{eq:radial_decomp}
and construct the processes $(\bm{M^X},\bm{M^Y},\bm{L^X},\bm{L^Y})$ 
by~\eqref{eq:comono_coupling_1} and~\eqref{eq:comono_coupling_2}, independent of standard Brownian motions $\bm{B^X}$ and $\bm{B^Y}$ on $\R^d$. Then there exists constants $\bm{\varpi_X},\bm{\varpi_Y}\in\R^d$, such that $\bm{X}_t\eqd \bm{\varpi_X}t+\bm{\Sigma_X}\bm{B}^{\bm{X}}_t+\bm{M}^{\bm{X}}_t+\bm{L}^{\bm{X}}_t$ and $\bm{Y}_t\eqd \bm{\varpi_Y}t+\bm{\Sigma_Y}\bm{B}^{\bm{Y}}_t+\bm{M}^{\bm{Y}}_t+\bm{L}^{\bm{Y}}_t$ for all $t\in[0,1]$. Moreover, this  coupling of $\bm{X}$ and $\bm{Y}$ satisfies   
\begin{align}\label{eq:small_jump_comono_coup}
    \E\bigg[\sup_{t \in [0,1]}\big|\bm{M^X}_t-\bm{M^Y}_t\big|^2\bigg]&\le 4\int_{[\ve,\infty)\times \Sp^{d-1}} (\rho_{\bm{X}}^{\la}(x,\bm{v})-\rho_{\bm{Y}}^{\la}(x,\bm{v}))^2 \D x \otimes \sigma(\D \bm{v}).
\end{align}
Furthermore, if $\int_{\R^d_{\bm{0}}}|\bm{w}|^2\1_{\R^d \setminus B_{\bm{0}}(1)}(\bm{w})\nu_{\bm{X}}(\D\bm{w})<\infty$ and $\int_{\R^d_{\bm{0}}}|\bm{w}|^2\1_{\R^d \setminus B_{\bm{0}}(1)}(\bm{w})\nu_{\bm{Y}}(\D\bm{w})<\infty$, then 
\begin{equation}
\label{eq:L2-J-comono}
\E\bigg[\sup_{t\in[0,1]}\big|\bm{M^X}_t+\bm{L^X}_t-(\bm{M^Y}_t+\bm{L^Y}_t)-\bm{m}t\big|^2\bigg]
\le 4\int_{(0,\infty)\times \Sp^{d-1}}(\rho_{\bm{X}}^{\la}(x,\bm{v})-\rho_{\bm{Y}}^{\la}(x,\bm{v}))^2 \D x \otimes \sigma(\D \bm{v}),
\end{equation}
were we define $\bm{m}\coloneqq\E[\bm{L^X}_1-\bm{L^Y}_1]
=\int_{(0,\ve)\times \Sp^{d-1}}\bm{v}
(\rho_{\bm{X}}^{\la}(x,\bm{v})-\rho_{\bm{Y}}^{\la}(x,\bm{v}))\D s \otimes\sigma(\D\bm{v})\in\R^d$. In particular, 
\begin{equation}
\label{eq:L2-X-comono}
\begin{split}
\mW_2(\bm{X},\bm{Y})
&\le |\bm{\varpi_X}-\bm{\varpi_Y}+\bm{m}|
    + 2 d^{1/2}|\bm{\Sigma_X}-\bm{\Sigma_Y}|\\
&\qquad   + 2\bigg(\int_{(0,\infty)\times \Sp^{d-1}}(\rho_{\bm{X}}^{\la}(x,\bm{v})-\rho_{\bm{Y}}^{\la}(x,\bm{v}))^2 \D x \otimes \sigma(\D \bm{v})\bigg)^{1/2}.
\end{split}
\end{equation}
\end{proposition}

Coupling the jumps of $\bm{X}$ and $\bm{Y}$ via~\eqref{eq:comono_coupling_1} and \eqref{eq:comono_coupling_2} is based on the idea behind the one-dimensional comonotonic coupling, applied to the magnitudes of the jumps of $\bm{X}$ and $\bm{Y}$. Indeed, in the coupling of Proposition~\ref{prop:small_jump_bound_comonotonic}, we align the angular components of the jumps and then couple the magnitudes via the right inverses $\rho_{\bm{X}}^\la(\cdot,\bm{v})$ and $\rho_{\bm{Y}}^\la(\cdot,\bm{v})$ (of possibly unbounded functions $x\mapsto\rho_{\bm{X}}([x,\infty),\bm{v})$ and $x\mapsto\rho_{\bm{Y}}([x,\infty),\bm{v})$) evaluated along the sequence $(\Gamma_n)_{n\in \N}$ of partial sums of iid standard exponentially distributed random variables. Note that this construction is analogous to the one-dimensional comonotonic coupling of real random variables described above, but allows for the functions $x\mapsto\rho_{\bm{X}}([x,\infty),\bm{v})$ and $x\mapsto\rho_{\bm{Y}}([x,\infty),\bm{v})$ to be unbounded.

\begin{proof}
We start by showing that there exist $\bm{\varpi_X},\bm{\varpi_Y}\in\R^d$, such that $\bm{X}_t\eqd\bm{\varpi_X}t+\bm{\Sigma_X}\bm{B}^{\bm{X}}_t+\bm{M}^{\bm{X}}_t+\bm{L}^{\bm{X}}_t$ and $\bm{Y}_t\eqd\bm{\varpi_Y}t+\bm{\Sigma_Y}\bm{B}^{\bm{Y}}_t+\bm{M}^{\bm{Y}}_t+\bm{L}^{\bm{Y}}_t$ for all $t\in[0,1]$. The proof of this fact is essentially given in~\cite[p.~4]{MR1833707}, we outline it here for completeness. By the symmetry of the construction, it is sufficient to prove the first equality in law only. Since $\bm{X}$ is a L\'evy process,  $\Xi_{\bm{X}}=\sum_{\{t:\Delta \bm{X}_t \ne 0\}}\delta_{(t,\Delta \bm{X}_t)}$ is a Poisson random measure on $[0,1]\times \R^d_{\bm{0}}$ of the jumps of $\bm{X}$ with mean measure $\Leb\otimes \nu_{\bm{X}}$~\cite[Thm~19.2]{MR3185174}. By~\eqref{eq:comono_coupling_1} and \eqref{eq:comono_coupling_2}, the equality in law $\bm{X}_t\eqd\bm{\varpi_X}t+\bm{\Sigma_X}\bm{B}^{\bm{X}}_t+\bm{M}^{\bm{X}}_t+\bm{L}^{\bm{X}}_t$ holds for some $\bm{\varpi_X}\in\R^d$ if 
\begin{equation}
\label{eq:equality_PPP}
    \Xi_{\bm{X}}\eqd \sum_{n=1}^\infty \delta_{(U_n,\rho_{\bm{X}}^{\la}(\Gamma_n,\bm{V}_n)\bm{V}_n)}.
\end{equation}
To prove this, consider the Poisson random measure $\Xi$ on $[0,1]\times(0,\infty)\times \Sp^{d-1}$, with mean measure $\Leb \otimes \Leb \otimes \sigma$, defined in~\eqref{eq:marked_PPP_Comonotonic}. Define $h:[0,1]\times(0,\infty)\times \Sp^{d-1} \to [0,1]\times \R^d$ by $h(t,x,\bm{v})\coloneqq(t,\rho_{\bm{X}}^\la(x,\bm{v})\bm{v})$. Crucially, by construction, we have $(\Leb \otimes \Leb \otimes \sigma)\circ h^{-1}=\Leb \otimes \nu_{\bm{X}}$ on $\mathcal{B}(\R^d_{\bm{0}})$. Thus, by the Mapping Theorem~\cite[Sec.~2.3]{MR1207584}, we get $\Xi \circ h^{-1}\eqd\Xi_{\bm{X}}$. Moreover, since $\sum_{n=1}^\infty \delta_{(U_n,\rho_{\bm{X}}^{\la}(\Gamma_n,\bm{V}_n)\bm{V}_n)} =\Xi \circ h^{-1} $ by construction, the equality in law in~\eqref{eq:equality_PPP} follows.

Next, we prove that 
\begin{align*}
    \bm{M^X}_t-\bm{M^Y}_t=\int_{[0,t]\times [\ve,\infty)\times \Sp^{d-1}}\bm{v}(\rho_{\bm{X}}^{\la}(x,\bm{v})-\rho_{\bm{X}}^{\la}(x,\bm{v})) \wt \Xi(\D s, \D x, \D \bm{v})
\end{align*} is a square-integrable martingale.  Let $\mathcal{F}_t$ to be the $\sigma$-field generated by $\Xi((0,s]\times A)$ for $0 \le s \le t$ and $A \in \mathcal{B}([\ve,\infty)\times \Sp^{d-1})$, then $\bm{M^X}-\bm{M^Y}$ is adapted w.r.t. $(\mathcal{F}_t)_{t\ge 0}$ and fulfills the martingale property by virtue of being an integral with respect to a compensated Poisson random measure. Furthermore, by the triangle inequality, $\bm{M^X}_t-\bm{M^Y}_t$ is square integrable since both $\bm{M^X}_t$ and $\bm{M^Y}_t$ are square integrable. Since the process $|\bm{M^X}_t-\bm{M^Y}_t|$ is a submartingale, Doob's maximal inequality~\cite[Prop.~7.16]{MR1876169} and Campbell's formula~\cite[p.~28]{MR1207584} imply
\[
\E\bigg[\sup_{t \in [0,1]}
    \big|\bm{M^X}_t-\bm{M^Y}_t\big|^2\bigg]
\le 4\E\big[\big|\bm{M^X}_1-\bm{M^X}_1\big|^2\big]
=4\int_{[\ve,\infty)\times \Sp^{d-1}} (\rho_{\bm{X}}^{\la}(x,\bm{v})-\rho_{\bm{Y}}^{\la}(x,\bm{v}))^2 \D x \otimes \sigma(\D \bm{v}).
\]

If $\1_{\R^d \setminus B_{\bm{0}}(1)}(\bm{w})\nu_{\bm{X}}(\D \bm{w})$ and $\1_{\R^d \setminus B_{\bm{0}}(1)}(\bm{w})\nu_{\bm{Y}}(\D \bm{w})$ have finite second moment, a similar bound can be established for the big-jump components using Doob's maximal inequality and Campbell's formula:
\begin{align*}
\E\bigg[\sup_{t\in[0,1]}\big|\bm{M^X}_t+\bm{L^X}_t-(\bm{M^Y}_t+\bm{L^Y}_t)-\bm{m}t\big|^2\bigg]
&\le 4\E\big[\big|\bm{M^X}_1+\bm{L^X}_1-(\bm{M^Y}_1+\bm{L^Y}_1)-\bm{m}\big|^2\big]\\
&=4\int_{(0,\infty)\times \Sp^{d-1}}(\rho_{\bm{X}}^{\la}(x,\bm{v})-\rho_{\bm{Y}}^{\la}(x,\bm{v}))^2 \D x \otimes \sigma(\D \bm{v}).
\end{align*}

Finally,~\eqref{eq:L2-X-comono} follows directly from~\eqref{eq:L2-J-comono} and the standard arguments given in Appendix~\ref{app:A}.
\end{proof}

\begin{proposition}
\label{prop:Lq-J-comonotonic}
Pick $q \in (0,1]$. Assume that 
 the L\'evy measures of $\nu_{\bm{X}}$ and $\nu_{\bm{Y}}$ admit the radial decomposition in~\eqref{eq:radial_decomp}
and construct the processes $(\bm{L^X},\bm{L^Y})$ 
by~\eqref{eq:comono_coupling_2}.
For any $\ve\in(0,\infty)$ (we may have $\ve=\infty$ when $\bm{X}$ and $\bm{Y}$ are of finite variation), the coupling $(\bm{L^X},\bm{L^Y})$ satisfies 
\begin{equation}
\label{eq:Lq-J_comonotonic}
\E\bigg[\sup_{t \in [0,1]}\big|\bm{L^X}_t-\bm{L^Y}_t\big|^q\bigg]\le  \int_{(0,\ve)\times \Sp^{d-1}} |\rho_{\bm{X}}^{\la}(x,\bm{v})-\rho_{\bm{Y}}^{\la}(x,\bm{v})|^q \D x \otimes \sigma(\D \bm{v}).
\end{equation}
In particular, the following inequality holds
\begin{equation}
\label{eq:Lq-X_comonotonic}
\begin{split}
\mW_q(\bm{X},\bm{Y})
&\le |\bm{\varpi_X}-\bm{\varpi_Y}|^q
+\bigg(4\int_{[\ve,\infty)\times \Sp^{d-1}}(\rho_{\bm{X}}^{\la}(x,\bm{v})-\rho_{\bm{Y}}^{\la}(x,\bm{v}))^2 \D x \otimes \sigma(\D \bm{v})\bigg)^{q/2}\\
&\qquad+2^qd^{q/2}|\bm{\Sigma_X}-\bm{\Sigma_Y}|^q
+\int_{(0,\ve)\times \Sp^{d-1}} |\rho_{\bm{X}}^{\la}(x,\bm{v})-\rho_{\bm{Y}}^{\la}(x,\bm{v})|^q \D x \otimes \sigma(\D \bm{v}).
\end{split}
\end{equation}
\end{proposition}

As was the case for the thinning coupling, we can again note that~\eqref{eq:Lq-J_comonotonic} \&~\eqref{eq:Lq-X_comonotonic} hold even without assuming that  $\1_{\R^d \setminus B_{\bm{0}}(1)}(\bm{w})\nu_{\bm{X}}(\D\bm{w})$ and $\1_{\R^d \setminus B_{\bm{0}}(1)}(\bm{w})\nu_{\bm{Y}}(\D\bm{w})$ have a finite $q$-moment. However, under such an assumption, the upper bounds are finite since the integral on the right of~\eqref{eq:Lq-J_comonotonic} is bounded by $\int_{\R^d_{\bm{0}}}\1_{\R^d \setminus B_{\bm{0}}(\kappa)}(\bm{w})
    |\bm{w}|^q\nu_{\bm{X}}(\D\bm{w})
    +\int_{\R^d_{\bm{0}}}\1_{\R^d \setminus B_{\bm{0}}(\kappa)}(\bm{w})
    |\bm{w}|^q\nu_{\bm{Y}}(\D\bm{w})<\infty$ for some $\kappa>0$.

\begin{proof}
For $\kappa \in (0,\ve)$, we denote by $\bm{L^X}_{t,\kappa}$ and $\bm{L^Y}_{t,\kappa}$ the truncated large jumps of $\bm{X}$ and $\bm{Y}$, given by 
\begin{equation*}
    \bm{L^X}_{t,\kappa}\coloneqq \int_{[0,t]\times (\kappa,\ve)\times \Sp^{d-1}}\bm{v}\rho_{\bm{X}}^{\la}(x,\bm{v}) \Xi(\D s, \D x, \D \bm{v}), \,\text{ and }\, \bm{L^Y}_{t,\kappa}\coloneqq \int_{[0,t]\times (\kappa,\ve)\times \Sp^{d-1}}\bm{v}\rho_{\bm{Y}}^{\la}(x,\bm{v}) \Xi(\D s, \D x, \D \bm{v}).
\end{equation*} Note that $\bm{L^X}_{t,\kappa}-\bm{L^Y}_{t,\kappa}=\int_{[0,t]\times (\kappa,\ve)\times \Sp^{d-1}}\bm{v}(\rho_{\bm{X}}^{\la}(x,\bm{v}) -\rho_{\bm{Y}}^{\la}(x,\bm{v})) \Xi(\D s, \D x, \D \bm{v})$, and thus, from the concavity of $x\mapsto x^q$ for $x>0$, it follows that
\[
\sup_{t\in[0,1]}\big|\bm{L^X}_{t,\kappa}-\bm{L^Y}_{t,\kappa}\big|^q
\le \int_{[0,1]\times (\kappa,\ve)\times \Sp^{d-1}}|\bm{v}|^q|\rho_{\bm{X}}^{\la}(x,\bm{v})-\rho_{\bm{Y}}^{\la}(x,\bm{v})|^q  \Xi(\D s, \D x, \D \bm{v}).
\] Since $\bm{v}\in \Sp^{d-1}$ we have that $|\bm{v}|^q=1$, and Campbell's theorem~\cite[p.~28]{MR1207584} then implies that
\begin{equation*}
    \E\left[\int_{[0,1]\times (\kappa,\ve)\times \Sp^{d-1}}|\rho_{\bm{X}}^{\la}(x,\bm{v})-\rho_{\bm{Y}}^{\la}(x,\bm{v})|^q  \Xi(\D s, \D x, \D \bm{v})\right] 
    =\int_{(\kappa,\ve)\times \Sp^{d-1}} |\rho_{\bm{X}}^{\la}(x,\bm{v})-\rho_{\bm{Y}}^{\la}(x,\bm{v})|^q \D x \otimes \sigma(\D \bm{v}).
\end{equation*}
Thus, altogether, this implies that
\begin{align*}
\E\bigg[\sup_{t \in [0,1]}
    \big|\bm{L^X}_{t,\kappa}-\bm{L^Y}_{t,\kappa}\big|^q\bigg]
\le \int_{(\kappa,\ve)\times \Sp^{d-1}} |\rho_{\bm{X}}^{\la}(x,\bm{v})-\rho_{\bm{Y}}^{\la}(x,\bm{v})|^q \D x \otimes \sigma(\D \bm{v}).
\end{align*}
Due to the monotone convergence theorem, it follows, as $\kappa \da 0$, that
\begin{equation*}
    \int_{(\kappa,\ve)\times \Sp^{d-1}} |\rho_{\bm{X}}^{\la}(x,\bm{v})-\rho_{\bm{Y}}^{\la}(x,\bm{v})|^q \D x \otimes \sigma(\D \bm{v}) \to \int_{(0,\ve)\times \Sp^{d-1}} |\rho_{\bm{X}}^{\la}(x,\bm{v})-\rho_{\bm{Y}}^{\la}(x,\bm{v})|^q \D x \otimes \sigma(\D \bm{v}).
\end{equation*}
Furthermore, Fatou's lemma together with the above observations imply that 
\begin{align*}
    \E\bigg[\liminf_{\kappa \da 0}\sup_{t \in [0,1]}
    \big|\bm{L^X}_{t,\kappa}-\bm{L^Y}_{t,\kappa}\big|^q\bigg] &\le \liminf_{\kappa \da 0}\E\bigg[\sup_{t \in [0,1]}
    \big|\bm{L^X}_{t,\kappa}-\bm{L^Y}_{t,\kappa}\big|^q\bigg]\\
    &\le \int_{(0,\ve)\times \Sp^{d-1}} |\rho_{\bm{X}}^{\la}(x,\bm{v})-\rho_{\bm{Y}}^{\la}(x,\bm{v})|^q \D x \otimes \sigma(\D \bm{v}).
\end{align*}
We can now conclude \eqref{eq:Lq-J_comonotonic}, as $\liminf_{\kappa \da 0}\sup_{t \in [0,1]}
    \big|\bm{L^X}_{t,\kappa}-\bm{L^Y}_{t,\kappa}\big|^q =\sup_{t \in [0,1]}
    \big|\bm{L^X}_{t}-\bm{L^Y}_{t}\big|^q$ a.s., since the largest jumps of $\bm{L^X}_{t,\kappa}$ and $\bm{L^Y}_{t,\kappa}$ are finite on the time interval $[0,1]$. 

Inequality~\eqref{eq:Lq-X_comonotonic} then follows from~\eqref{eq:small_jump_comono_coup},~\eqref{eq:Lq-J_comonotonic} and the elementary arguments in Appendix~\ref{app:A}.
\end{proof}

\section{Upper bounds on the Wasserstein distance in the domain of attraction}
\label{sec:stable_limits_upper}

The main aim of this section is to prove the upper bounds in Theorems~\ref{thm:upper_lower_simple},~\ref{thm:upper_lower_general} \&~\ref{thm:conv_BM_limit} above. In Section~\ref{subsec:domain_of_attraction} we give the characterisation, in terms of their generating triplets, of the L\'evy processes in $\R^d$ that are in the stable domain-of-attraction. The proof of the upper bounds in Theorem~\ref{thm:upper_lower_simple}, based on the thinning coupling, is given in Section~\ref{sec:heavyTail_Thinning}. The upper bounds in  Theorem~\ref{thm:upper_lower_general}  are established in Section~\ref{sec:conv_stab_comonot} using the comonotonic coupling. In Section~\ref{subsec:Brownian_limits}, we prove the upper bounds of Theorem~\ref{thm:conv_BM_limit} for the Brownian limit. In the proofs, we will rely on the following consequence of Jensen's inequality
\begin{equation}
    \label{eq:Wasserstein_relationship}
\mW_q(\bm{\mathcal{X}},\bm{\mathcal{Y}})\le \mW_{q'}(\bm{\mathcal{X}},\bm{\mathcal{Y}})^{\frac{q\wedge 1}{q'\wedge1}}\qquad\text{for any $0<q<q'$.}
\end{equation}

\subsection{Small-time domain of attraction for L\'evy processes}
\label{subsec:domain_of_attraction}

We start by defining the attractor. 

\begin{defin}[Stable process]
\label{def:stable}
For any $\alpha\in(0,2]$, the law of 
an $\alpha$-stable L\'evy process $\bm{Z}$
 is given by a generating triplet $(\bm{\gamma_Z},\bm{\Sigma_Z}\bm{\Sigma_Z}^\tra,\nu_{\bm{Z}})$
(for the cutoff function $\bm{w}\mapsto\1_{B_{\bm{0}}(1)}(\bm{w})$)
as follows: the L\'evy measure is given by~\eqref{eq:Levy-measure-stable}, i.e.,
\begin{equation*}
    \nu_{\bm{Z}}(A)
    \coloneqq c_\alpha\int_0^\infty\int_{\Sp^{d-1}}
        \1_{A}(r\bm{v})\sigma(\D\bm{v})r^{-\alpha-1}\D r,
    \quad A\in\mathcal{B}(\R^d_{\bm{0}}),
\end{equation*}
where $\sigma$ is a probability measure  on $\mathcal{B}(\Sp^{d-1})$  
and  $c_\alpha\in[0,\infty)$ an ``intensity'' parameter, satisfying
\begin{itemize}[leftmargin=1em, nosep]
    \item $\alpha=2$ [Brownian motion with zero drift]: $\bm{\Sigma_Z}\ne \bm{0}$, $\bm{\gamma_Z}=\bm{0}$ and $c_\alpha=0$ (i.e. $\nu_{\bm{Z}}\equiv 0$);
    \item $\alpha\in(1,2)$ [infinite variation,  zero-mean process]: $c_\alpha>0$, $\bm{\gamma_Z}=-\int_{\R^d\setminus B_{\bm{0}}(1)}\bm{x}\nu_{\bm{Z}}(\D\bm{x})$ and $\bm{\Sigma_Z}=\bm{0}$;
    \item $\alpha=1$ [Cauchy process]: either $c_\alpha>0$, with symmetric angular component $\int_{\Sp^{d-1}}\bm{v}\sigma(\D\bm{v})=\bm{0}$, or $c_\alpha=0$ and the process $\bm{Z}$ is  a deterministic nonzero linear drift, i.e. $\bm{Z}_t=\bm{\gamma_Z}t$ for all times $t$;
    \item $\alpha\in(0,1)$ [finite variation and zero natural drift]:  $c_\alpha>0$ and  
    $\bm{\gamma_Z}=\int_{B_{\bm{0}}(1)\setminus\{\bm0\}}\bm{x}\nu_{\bm{Z}}(\D\bm{x})$.
\end{itemize}
\end{defin}

It follows from the definition that an  $\alpha$-stable process  $\bm{Z}$  satisfies the scaling property $(\bm{Z}_{st})_{s\in[0,1]}\eqd(t^{1/\alpha}\bm{Z}_s)_{s\in[0,1]}$ for $t>0$. For $\alpha\in[1,2)$ (resp. $\alpha\in(0,1)$), a non-deterministic $\alpha$-stable process $\bm{Z}$ is of infinite (resp. finite) variation by~\cite[Thm~21.9]{MR3185174}, since~\eqref{eq:Levy-measure-stable} implies  $\int_{B_{\bm{0}}(1)\setminus\{\bm0\}}|\bm{x}|\nu_{\bm{Z}}(\D\bm{x})=\infty$ (resp. $\int_{B_{\bm{0}}(1)\setminus\{\bm0\}}|\bm{x}|\nu_{\bm{Z}}(\D\bm{x})<\infty$). Moreover, in the case of Cauchy process (stability index $\alpha=1$), $\gamma_{\bm{Z}}$ can be arbitrary if $c_\alpha>0$ and satisfies $\gamma_{\bm{Z}}\in\R^d_{\bm{0}}$ if $c_\alpha=0$.

For any $\bm{a}\in \Sp^{d-1}$, define $\scrL_{\bm{a}}(r)\coloneqq\{\bm{x}\in\R^d:\langle\bm{a},\bm{x}\rangle\ge r\}$ for any $r>0$. The following known result characterises the L\'evy processes in the domain of attraction of an $\alpha$-stable process defined above. Theorem~(\nameref{thm:small_time_domain_stable}) is a consequence of~\cite[Thm~15.14]{MR1876169} and~\cite[Thm~2]{MR3784492} (for completeness, we provide the proof in  Appendix~\ref{app:domain_of_attraction} below).

\begin{theorem}[S-DoA]
\label{thm:small_time_domain_stable}
Let $\bm{X}=(\bm{X}_t)_{t \in [0,1]}$ and $\bm{Z}=(\bm{Z}_t)_{t\in [0,1]}$ be L\'evy processes in $\R^d$. The following are equivalent:
\begin{itemize}[leftmargin=.5em, nosep]
\item $(\bm{X}_{st}/g(t))_{s \in [0,1]}\cid (\bm{Z}_s)_{s\in [0,1]}$ as $t \da 0$ in the Skorokhod space for some normalising function $g:(0,1]\to(0,\infty)$,
\item $\bm{X}_t/g(t)\cid \bm{Z}_1$  as $t \da 0$ for some normalising function $g:(0,1]\to(0,\infty)$,
\item $\bm{Z}$ is $\alpha$-stable for some $\alpha\in(0,2]$, the function $G:t\mapsto g(t)t^{-1/\alpha}$ is slowly varying at $0$, and the generating triplets $(\bm\gamma_{\bm{X}},\bm{\Sigma_X}\bm{\Sigma}_{\bm{X}}^\tra,\nu_{\bm{X}})$ and $(\bm\gamma_{\bm{Z}},\bm{\Sigma_Z}\bm{\Sigma}_{\bm{Z}}^\tra,\nu_{\bm{Z}})$ (for the cutoff function $\bm{w} \mapsto \1_{B_{\bm 0}(1)}(\bm{w})$) of $\bm{X}$ and $\bm{Z}$, respectively, are related as follows:
\begin{itemize}[leftmargin=1em, nosep]
\item if $\alpha=2$ (attraction to Brownian motion), then 
\begin{equation}
\label{eq:Brownian-limit}
G(t)^{-2} \bigg(\bm{\Sigma_X}\bm{\Sigma}_{\bm{X}}^\tra
    +\int_{B_{\bm{0}}(g(t))\setminus\{\bm{0}\}}\bm{x}\bm{x}^\tra\nu_{\bm{X}}(\D\bm{x})\bigg)
\to
\bm{\Sigma_Z}\bm{\Sigma}_{\bm{Z}}^\tra,\quad\text{as }t\da 0;
\end{equation}
\item if $\alpha\in(1,2)$, we have $\bm{\Sigma_X}=\bm{0}$ and
\begin{equation}
\label{eq:jump-stable-limit}
t\nu_{\bm{X}}(\scrL_{\bm{v}}(g(t)))\to \nu_{\bm{Z}}(\scrL_{\bm{v}}(1)),
\quad\text{as }t\da 0,\quad\text{for any }\bm{v}\in\Sp^{d-1};
\end{equation}
    \item if $\alpha=1$ (attraction to Cauchy process), then~\eqref{eq:jump-stable-limit} holds,
\begin{equation}
\label{eq:Cauchy-limit}
G(t)^{-1}\bigg(\bm{\gamma_X}-\int_{B_{\bm0}(1)\setminus B_{\bm0}(g(t))}\bm{x}\nu(\D\bm{x})\bigg)\to \bm{\gamma_Z},
\quad\text{as }t\da 0,
\end{equation}
and for any $\bm{v}\in\Sp^{d-1}$ for which $\langle\bm{v},\bm{X}\rangle$ has finite variation (i.e. $\int_{B_{\bm 0}(1)\setminus\{\bm 0\}}\!|\langle\bm{v},\bm{x}\rangle|\nu_{\bm{X}}(\D\bm{x})<\infty$) and $\nu_{\bm{Z}}(\scrL_{\bm{v}}(1))>0$, the process $\langle\bm{v},\bm{X}\rangle$ has zero natural drift: $\langle\bm{v},\bm{\gamma_X}\rangle=\int_{B_{\bm 0}(1)\setminus\{\bm 0\}}\langle\bm{v},\bm{x}\rangle\nu_{\bm{X}}(\D\bm{x})$.
\item if $\alpha\in(0,1)$, then~\eqref{eq:jump-stable-limit} holds, $\bm{X}$ has finite variation (i.e. $\int_{B_{\bm0}(1)\setminus\{\bm0\}}|\bm{x}|\nu_{\bm{X}}(\D\bm{x})<\infty$) and zero natural drift (i.e. $\bm{\gamma_X}=\int_{B_{\bm 0}(1)\setminus\{\bm 0\}}\bm{x}\nu_{\bm{X}}(\D\bm{x})$).
\end{itemize}
\end{itemize}
Moreover,  the function $g$ 
satisfying the weak limit above
is asymptotically unique at $0$: a positive  function $\wt g$ 
satisfies $(\bm{X}_{st}/\wt g(t))_{s \in [0,1]}\cid (\bm{Z}_s)_{s\in [0,1]}$ as $t \da 0$
if and only if $\wt g(t)/g(t)\to 1$ as $t\da 0$.
\end{theorem}

Note that in the case $\alpha=2$ in Theorem~(\nameref{thm:small_time_domain_stable}),
we may have $\bm{\Sigma_X}=\bm{0}$ (see Example~\ref{ex:arbitrarily_slow_BM_conv} below), but in this case the function $G$ cannot be asymptotically equal to a positive constant. Moreover, in the case $\alpha\in(1,2)$, the process $\bm{X}$
does not require centering since its mean is linear in time and thus disappears in the scaling limit. However, in the finite variation case (i.e. when $\alpha\in(0,1)$), the process $\bm{X}$ must have zero natural drift for the scaling limit to exist.

\subsection{Domain of normal attraction: the thinning coupling}\label{sec:heavyTail_Thinning}
Let
$(\bm{\gamma_X},\bm{\Sigma_X}\bm{\Sigma_X}^\tra,\nu_{\bm{X}})$
denote the generating triplet~\cite[Def.~8.2]{MR3185174} of $\bm{X}$
 with respect to the cutoff function $\bm{w}\mapsto\1_{B_{\bm{0}}(1)}(\bm{w})$
 on $\bm{w}\in\R^d$.
 Define the Blumenthal--Getoor (BG) index $\beta$ of $\bm{X}$ by
\begin{equation}
\label{eq:BG}
\beta:=\inf\{p>0:I_p<\infty\}\in[0,2],
\qquad
I_p:=\int_{B_{\bm{0}}(1)\setminus\{\bm{0}\}}|\bm{w}|^p\nu_{\bm{X}}(\D\bm{w}).
\end{equation} 
Fix $\beta_+\in[\beta,2]$ as follows:  $\beta_+\coloneqq\beta$ if $I_{\beta}<\infty$;
if $I_{\beta}=\infty$ and $\beta<1$, then pick $\beta_+\in(\beta,1)$; if $I_{\beta}=\infty$ and $\beta\ge 1$, then $\beta<2$ and hence 
choose $\beta_+\in(\beta,2)$.
In particular, note that $I_{\beta_+}<\infty$ and $\beta_+>0$. Furthermore, if $\int_{B_{\bm{0}}(1)\setminus\{\bm{0}\}} |\bm{w}|\nu_{\bm{X}}(\D \bm{w})<\infty$ (or, equivalently, if the pure-jump component of the L\'evy--It\^o decomposition~\eqref{eq:levy_ito_decomp} of $\bm{X}$ is finite variation), we say that $\bm{X}$ has zero natural drift if  $\bm{\gamma_X}=\int_{B_{\bm{0}}(1)\setminus\{\bm{0}\}} \bm{w}\nu_{\bm{X}}(\D \bm{w})$ and nonzero natural drift otherwise. Moreover, if $\nu_{\bm{X}}(B_{\bm{0}}(1)\setminus\{\bm 0\})<\infty$ (or, equivalently, the pure-jump component of $\bm{X}$ is of finite activity, i.e. a compound Poisson process), then $\beta=\beta_+=0$.
If 
$\beta_+=0$, throughout the paper we use the convention $1/\beta_+\coloneqq\infty$.

The following lemma gives an upper bound on the moments of the supremum of the norm of a general L\'evy process. Lemma~\ref{lem:moment_bound} plays an important role in the proofs of Section~\ref{sec:stable_limits_upper}.

\begin{lemma}
\label{lem:moment_bound}
Let $\bm{X}$ be a L\'evy process with 
generating triplet $(\bm{\gamma_X},\bm{\Sigma_X}\bm{\Sigma_X}^\tra,\nu_{\bm{X}})$.
 Recall the Blumenthal--Getoor index $\beta$ from~\eqref{eq:BG} and the associated quantity $\beta_+\in [\beta,2]$. Assume that, for some $p>0$, we have
$\int_{\R^d\setminus B_{\bm{0}}(1)}|\bm{w}|^p\nu_{\bm{X}}(\D\bm{w})<\infty$.
 Then there exist constants $C_i\in[0,\infty)$, $i=1,\ldots,4$, such that
\[
\E\bigg[\sup_{s\in[0,t]}|\bm{X}_s|^p\bigg]
\le \1_{\{\bm{\Sigma_X}\neq\bm{0}\}}C_1 t^{p/2} 
+C_2 t^p + C_3 t^{\min\{1,p/\beta_+\}},
\quad \text{ for } t\in[0,1].
\]
If $\int_{B_{\bm{0}}(1)\setminus\{\bm{0}\}}|\bm{w}|\nu_{\bm{X}}(\D\bm{w})< \infty$ and $\bm{X}$ has zero natural drift, i.e. $\bm{\gamma_X}=\int_{B_{\bm{0}}(1)\setminus\{\bm{0}\}}\bm{w}\nu_{\bm{X}}(\D\bm{w})$, then $C_2=0$ in the inequality above.
\end{lemma}

Note that, by the definition of $\beta_+$ above, the pure-jump component of $\bm{X}$ is a compound Poisson process if and only if $\beta_+=0$. In particular, if in addition in this case we have zero natural drift, then the pure-jump component of $\bm{X}$ is a compound Poisson process. The term $t^{p/2}$ in the bound of Lemma~\ref{lem:moment_bound} is present only if $\bm{X}$ has a non-trivial Gaussian component.

Lemma~\ref{lem:moment_bound}
is a multidimensional generalisation of~\cite[Lem.~2]{doi:10.1287/moor.2021.1163}. The proof of Lemma~\ref{lem:moment_bound}, given in Appendix~\ref{app:Proof_lemma_bound} below, is likewise a multidimensional generalisation of the arguments in the proof of~\cite[Lem.~2]{doi:10.1287/moor.2021.1163}. As in~\cite[Lem.~2]{doi:10.1287/moor.2021.1163}, the constants $C_i$, $i=1,\ldots,4$, can be given explicitly in terms of the characteristic triplet of $\bm{X}$. 

Consider a L\'evy process $\bm{X}$ in $\R^d$ in the domain of normal attraction of the $\alpha$-stable process $\bm{Z}$. Thus we may assume that $\bm{X}^t=(\bm{X}_{st}/t^{1/\alpha})_{s \in [0,1]}$ converges weakly to $\bm{Z}$
 as $t\da 0$.
We will now apply the thinning coupling, described in~\eqref{eq:comp_Poisson_measures} and~\eqref{eq:thining_coupling_defn} of Section~\ref{sec:thinning} above,
to quantify this convergence in terms of the Wasserstein distance under Assumption~(\nameref{asm:T}).

\begin{theorem}
\label{thm:d_thin_dom_attract}
Let $\alpha \in (0,2)\setminus\{1\}$ and Assumption~(\nameref{asm:T}) hold for some $p>0$. Then, for any $q\in(0,1]$ satisfying the condition $\int_{\R^d\setminus B_{\bm{0}}(1)}|\bm{w}|^q\nu_{\bm{X}}^{\mathrm{d}}(\D \bm{w})<\infty$, we have $\E\big[\sup_{s\in[0,1]} |\bm{R}_{st}/t^{1/\alpha}|^q\big] 
    = \Oh \big(t^{1-q/\alpha}\big)$ as $t \da 0$. Moreover, for any real number $q\in(0,\alpha)\cap(0,1]$, we let $\kappa(t)\coloneqq t^r$ for $t\in(0,1]$ and some $r\ge -1/\alpha$. Then, as $t \da 0$, we have
\begin{align}
\mW_q\big(\bm{J}^{\bm{S}^t,\kappa(t)},\bm{J}^{\bm{Z},\kappa(t)}\big)
    &=\begin{dcases}
    \Oh\big(t^{1-q/\alpha}\big(1+\log(1/t)\1_{\{p+q=\alpha,\,r\ne -1/\alpha\}}\big)\big), & \text{for }p+q\ge\alpha,\\
    \Oh\big(t^{p/\alpha + r(p+q-\alpha)}\big), & \text{for }p+q< \alpha. 
    \end{dcases} 
\label{eq:Up_bound_thin_J}\\
\mW_2\big(\bm{D}^{\bm{S}^t,\kappa(t)},\bm{D}^{\bm{Z},\kappa(t)}\big)^2
    &=\Oh\big(t^{p/\alpha+r(p-\alpha+2)}\big), \label{eq:Up_bound_thin_D}\\
\label{eq:Up_bound_thin_gamma}
    |\bm{\gamma}_{\bm{S}^t,\kappa(t)}-\bm{\gamma}_{\bm{Z},\kappa(t)}| 
&=\begin{dcases}
    \Oh\big(t^{1-1/\alpha}(1+\1_{\{p+1=\alpha,\,r\ne-1/\alpha\}}\log(1/t))\big),
    &p\ge \alpha-1>0,\\
    \Oh(t^{p/\alpha+ r(p-\alpha+1)}),&p<\alpha-1\text{ or }\alpha\in(0,1).
\end{dcases}
\end{align}
\end{theorem}

\begin{remark}\label{rem:construct_s_R_thin}
Under Assumption~(\nameref{asm:T}), we may decompose the process $\bm{X}$ as the sum $\bm{S}+\bm{R}$ of independent L\'evy processes $\bm{S}$ and $\bm{R}$ with generating triplets $(\bm{\gamma_S},\bm{0},\nu_{\bm{X}}^{\co})$ and $(\bm{\gamma_R},\bm{0},\nu_{\bm{X}}^{\mathrm{d}})$, respectively, such that, when $\alpha\in(0,1)$, both processes have zero natural drift (note that
for $\alpha\in(0,1)$, Assumption~(\nameref{asm:T}) and Theorem~(\nameref{thm:small_time_domain_stable})
imply that $\bm{X}$ has zero natural drift), and when $\alpha \in (1,2)$ then $\bm{R}$ has zero natural drift. For $t\in(0,1]$, let  $\bm{S}^t=(\bm{S}_{st}/t^{1/\alpha})_{s\in[0,1]}$ and $\bm{R}^t=(\bm{R}_{st}/t^{1/\alpha})_{s\in[0,1]}$ and note that 
$\bm{X}^t$ has the same law as $\bm{S}^t+\bm{R}^t$.
We couple $\bm{S}^t$ and $\bm{Z}$ via the coupling 
$(\bm{D}^{\bm{S}^t,\kappa},\bm{D}^{\bm{Z},\kappa},\bm{J}^{\bm{S}^t,\kappa},\bm{J}^{\bm{Z},\kappa})$ given in~\eqref{eq:comp_Poisson_measures} and~\eqref{eq:thining_coupling_defn} of
Section~\ref{sec:thinning}.
\end{remark}

\begin{remark}
\label{rem:best_r_in_general_thinning}
By~\eqref{eq:Wq-XY} and~\eqref{eq:Wasserstein_relationship}, we have 
\[
\mW_q(\bm{X}^t,\bm{Z})
\le \E\bigg[\sup_{s\in[0,1]} |\bm{R}_{s}^t|^q\bigg]
+\mW_q\big(\bm{J}^{\bm{S}^t,\kappa(t)},\bm{J}^{\bm{Z},\kappa(t)}\big)
+\mW_2\big(\bm{D}^{\bm{S}^t,\kappa(t)},\bm{D}^{\bm{Z},\kappa(t)}\big)^q
+|\bm{\gamma}_{\bm{S}^t,\kappa(t)}-\bm{\gamma}_{\bm{Z},\kappa(t)}|^q.
\]
A careful case-by-case analysis reveals that the upper bound implied by Theorem~\ref{thm:d_thin_dom_attract} on the distance above (which decreases as fast as the slowest of the terms on the right) decreases the fastest when $r$ is chosen as follows (recall $\alpha\in(0,2)\setminus\{1\}$):
\[
r=\begin{dcases}
    0,& \alpha>1,\\
    \frac{p}{\alpha(\alpha-p)},
        & \alpha<1,\, \alpha>p+q,\\
    \frac{\alpha-q(p+1)}{\alpha q(p+1-\alpha)},
        &\alpha<1,\, \alpha\le p+q.\\
\end{dcases}
\]
Moreover, in that case, we have
\[
\mW_q(\bm{X}^t,\bm{Z})
=\begin{dcases}
    \Oh\big(t^{\min\{
        p/2,\,
        \alpha-1\}q/\alpha}
        \big(1+|\log t|\1_{\{p+q=\alpha\}}+|\log t|^q\1_{\{p+1=\alpha\}}\big)\big),
        & \alpha>1,\\
    \Oh\big(t^{\min\{1-q/\alpha,
        pq/(\alpha(\alpha-p))\}}\big),
        & \alpha<1,\, \alpha>p+q,\\
    \Oh\big(t^{1-q/\alpha}\big(1+|\log t|\1_{\{p+q=\alpha\}}\big)\big),
        & \alpha<1,\, \alpha\le p+q.
\end{dcases}
\]
Since the above bounds are not easily interpretable because of the multiple cases depending on the parameters $(\alpha,p,q)$, we decided to only present the case $p\ge 1\vee 2(\alpha-1)$ in Theorem~\ref{thm:upper_lower_simple} above. In particular, this removes the possibility of a logarithmic term appearing in the upper bound.
\end{remark}

\begin{proof}[Proof of Theorem~\ref{thm:d_thin_dom_attract}]
The bound on $\bm{R}^t$ follows directly from Lemma~\ref{lem:moment_bound} with $\beta_+=0$ and the construction of $\bm{R}^t$. 

We now consider the process $\bm{S}^t$. Define the measure
\[
\mu(A)
\coloneqq c^{-1}\nu_{\bm{Z}}(A)=\frac{c_\alpha}{c}\int_{\Sp^{d-1}}\int_0^\infty\1_{A}(x\bm{v})\frac{\D x}{x^{\alpha+1}}\sigma(\D\bm{v}), \qquad \text{ for }A\in\mathcal{B}(\R^d_{\bm{0}}),
\]  where $c$ (resp. $c_\alpha$) is in Assumption~(\nameref{asm:T}) (resp. in~\eqref{eq:Levy-measure-stable} of the definition of $\bm{Z}$). The Radon--Nikodym derivative $f_{\bm{S}^t}(\bm{w}):=\nu_{\bm{S}^t}(\D\bm{w})/\mu(\D\bm{w})$ equals $f_{\bm{S}}(t^{1/\alpha}\bm{w})$ on the support of $\mu$, since L\'evy--Khintchine exponent satisfies $t\psi_{\bm{S}}(\bm{u}/t^{1/\alpha})=\psi_{\bm{S}^t}(\bm{u})$ and hence
\[
\nu_{\bm{S}^t}(\D\bm{w})
=t^{1+d/\alpha}\nu_{\bm{X}}^{\co}(\D(t^{1/\alpha}\bm{w}))
=t^{1+d/\alpha}f_{\bm{S}}(t^{1/\alpha}\bm{w})\mu(\D(t^{1/\alpha}\bm{w}))
=f_{\bm{S}}(t^{1/\alpha}\bm{w})\mu(\D\bm{w}).
\]
First we bound the large-jump component $\bm{J}^{\bm{S}^t,\kappa(t)}-\bm{J}^{\bm{Z},\kappa(t)}$: inequality~\eqref{eq:Lq-J} of Proposition~\ref{prop:Lq-J-thin} yields
\[
\E\bigg[\sup_{s\in[0,1]}|\bm{J}^{\bm{S}^t,\kappa(t)}_s-\bm{J}^{\bm{Z},\kappa(t)}_s|^q\bigg]
\le \frac{c_\alpha}{c}\int_{\Sp^{d-1}}\int_{\kappa(t)}^\infty
    |f_{\bm{S}}(t^{1/\alpha}x\bm{v})-c| x^{q-\alpha-1}\D x\sigma(\D\bm{v}).
\]
Recall that $f(x) \lesssim g(x)$ as $x\da 0$ means that there exists some $c_0,x_0>0$ such that $f(x) \le c_0 g(x)$ for all $x \le x_0$. 
Using Assumption~(\nameref{asm:T}), as $t \da 0$, we obtain 
\begin{align*}
\E\bigg[\sup_{s\in[0,1]}
    \big|\bm{J}^{\bm{S}^t,\kappa(t)}_s-\bm{J}^{\bm{Z},\kappa(t)}_s\big|^q\bigg]
&\lesssim \int_{t^r}^{t^{-1/\alpha}} 
        t^{p/\alpha} x^{p+q-\alpha-1}\D x
    +\int_{t^{-1/\alpha}}^\infty 
        x^{q-\alpha-1}\D x,
\end{align*}  where $\int_{t^{-1/\alpha}}^\infty 
        x^{q-\alpha-1}\D x =\Oh(t^{1-q/\alpha})$. Next, as $t \da 0$, we note that
\begin{align*}
    \int_{t^r}^{t^{-1/\alpha}}t^{p/\alpha} x^{p+q-\alpha-1}\D x 
    =\begin{dcases}
    \Oh\big(t^{1-q/\alpha}\big), & \text{for }p+q>\alpha,\\
    \Oh\big(t^{1-q/\alpha}\big(1+\log(1/t)\1_{\{r\ne -1/\alpha\}}\big)\big), & \text{for }p+q=\alpha,\\
    \Oh\big(t^{p/\alpha + r(p+q-\alpha)}\big), & \text{for }p+q< \alpha. 
    \end{dcases} 
\end{align*} 
Thus, since $r\ge -1/\alpha$, altogether we have, as $t \da 0$,  
\begin{equation*}
    \E\bigg[\sup_{s\in[0,1]}
    \big|\bm{J}^{\bm{S}^t,\kappa(t)}_s-\bm{J}^{\bm{Z},\kappa(t)}_s\big|^q\bigg] = \begin{dcases}
    \Oh\big(t^{1-q/\alpha}\big(1+\log(1/t)\1_{\{p+q=\alpha,\,r\ne -1/\alpha\}}\big)\big), & \text{for }p+q\ge\alpha,\\
    \Oh\big(t^{p/\alpha + r(p+q-\alpha)}\big), & \text{for }p+q< \alpha. 
    \end{dcases}  
\end{equation*}

Next, we find the rate for the small-jump component $\bm{D}^{\bm{S}^t,\kappa(t)}-\bm{D}^{\bm{Z},\kappa(t)}$. Assumption~(\nameref{asm:T}) and~\eqref{eq:L2-D-thin} of Proposition~\ref{prop:L2-thin} imply that, as $t \da 0$, 
\begin{align*}
\E\bigg[\sup_{s\in[0,1]}\big|\bm{D}^{\bm{S}^t,\kappa(t)}_s-\bm{D}^{\bm{Z},\kappa(t)}_s\big|^2\bigg]
&\le 4\frac{c_\alpha}{c}\int_{\Sp^{d-1}}\int_0^{\kappa(t)}
|f_{\bm{S}}(t^{1/\alpha}x\bm{v})-c|
x^{1-\alpha}\D x \sigma(\D\bm{v})\\
&\le 4K_T\frac{c_\alpha}{c}\int_0^{\kappa(t)}
t^{p/\alpha}
x^{p-\alpha+1}\D x
=\Oh\big(t^{p/\alpha+r(p-\alpha+2)}\big).
\end{align*}

Next, we control the difference $|\bm{\gamma}_{\bm{S}^t,\kappa(t)}-\bm{\gamma}_{\bm{Z},\kappa(t)}|$ of the drift terms. 
First, consider the infinite variation case $\alpha \in (1,2)$. Since $\bm{Z}$ has zero mean, representation~\eqref{eq:levy_ito_decomp} implies 
\begin{align*}
\bm{\gamma}_{\bm{S}^t,\kappa(t)}
&= t^{1-1/\alpha}\E[\bm{S}_1]
-\int_{\R^d}\bm{w}\1_{\R^d\setminus B_{\bm{0}}(\kappa(t))}(\bm{w})f_{\bm{S}}(t^{1/\alpha}\bm{w})\mu(\D\bm{w}), \quad \text{ and }\\
\bm{\gamma}_{\bm{Z},\kappa(t)}
&=
-\int_{\R^d}\bm{w}\1_{\R^d\setminus B_{\bm{0}}(\kappa(t))}(\bm{w})c\mu(\D\bm{w}).
\end{align*} 
Thus, we obtain
\begin{align}
\label{eq:difference_gammas}
\bm{\gamma}_{\bm{S}^t,\kappa(t)}-\bm{\gamma}_{\bm{Z},\kappa(t)}
&= t^{1-1/\alpha}\E[\bm{S}_1]
-\int_{\R^d}\bm{w}\1_{\R^d\setminus B_{\bm{0}}(\kappa(t))}(\bm{w})(f_{\bm{S}}(t^{1/\alpha}\bm{w})-c)\mu(\D\bm{w}).
\end{align} 
By Assumption~(\nameref{asm:T}), the integral in the display
satisfies
\begin{align*}
&\bigg|\int_{\R^d}\bm{w}\1_{\R^d\setminus B_{\bm{0}}(\kappa(t))}(\bm{w})(f_{\bm{S}}(t^{1/\alpha}\bm{w})-c)\mu(\D\bm{w})\bigg|\\
    &\le \int_{\R^d}|\bm{w}|\1_{\R^d\setminus B_{\bm{0}}(\kappa(t))}(\bm{w})|f_{\bm{S}}(t^{1/\alpha}\bm{w})-c|\mu(\D\bm{w})
    =\frac{c_\alpha}{c} \int_{\Sp^{d-1}} \int_{\kappa(t)}^\infty x^{-\alpha}|f_{\bm{S}}(t^{1/\alpha}x\bm{v})-c|\D x\,\sigma(\D\bm{v})\\
    &\le K_T\frac{c_\alpha}{c} \bigg(t^{p/\alpha}\int_{\kappa(t)}^{t^{-1/\alpha}}  x^{p-\alpha}\D x+\int_{t^{-1/\alpha}}^\infty  x^{-\alpha}\D x\bigg)
    =\begin{dcases}
        \Oh(t^{1-1/\alpha}),&p>\alpha-1,\\
        \Oh(t^{1-1/\alpha}(1+\log(1/t)\1_{\{r\ne -1/\alpha\}})),&p=\alpha-1,\\
        \Oh(t^{p/\alpha+r(p-\alpha+1)}),&p<\alpha-1,
    \end{dcases}
\end{align*}
where we used the fact that $r\ge -1/\alpha$. By~\eqref{eq:difference_gammas}, we obtain
\begin{equation*}
|\bm{\gamma}_{\bm{S}^t,\kappa(t)}-\bm{\gamma}_{\bm{Z},\kappa(t)}|
=\begin{dcases}
    \Oh\big(t^{1-1/\alpha}(1+\1_{\{p+1=\alpha,\,r\ne-1/\alpha\}}\log(1/t))\big),
    &p\ge \alpha-1,\\
    \Oh(t^{p/\alpha+ r(p-\alpha+1)}),&p<\alpha-1.
\end{dcases}
\end{equation*}

In the finite variation case $\alpha \in (0,1)$, recall that $\bm{S}$ and $\bm{Z}$ have zero natural drift, so that
\begin{align*}
\bm{\gamma}_{\bm{S}^t,\kappa(t)}&= \int_{\R^d}\bm{w}\1_{B_{\bm{0}}(\kappa(t))}(\bm{w})f_{\bm{S}}(t^{1/\alpha}\bm{w})\mu(\D\bm{w}), \quad \bm{\gamma}_{\bm{Z},\kappa(t)}=
\int_{\R^d}\bm{w}\1_{B_{\bm{0}}(\kappa(t))}(\bm{w})c\mu(\D\bm{w}).
\end{align*} 
Thus, we have, by Assumption~(\nameref{asm:T}),
\begin{align*}
|\bm{\gamma}_{\bm{S}^t,\kappa(t)}-\bm{\gamma}_{\bm{Z},\kappa(t)}|&\le \int_{\R^d}|\bm{w}|\1_{B_{\bm{0}}(\kappa(t))}(\bm{w})|f_{\bm{S}}(t^{1/\alpha}\bm{w})-c|\mu(\D\bm{w}) \\ 
    &\le K_T\frac{c_\alpha}{c} \int_0^{\kappa(t)} t^{p/\alpha}x^{p-\alpha}\D x 
    = \Oh(t^{p/\alpha+r(p-\alpha+1)}),
\end{align*}
completing the proof.
\end{proof}

\subsection{Domain of non-normal  attraction: the comonotonic coupling}
\label{sec:conv_stab_comonot}
Let $\bm{Z}$ be an $\alpha$-stable process on $\R^d$ for some $\alpha\in(0,2)$, defined as in Section~\ref{subsec:domain_of_attraction}, with ``intensity'' parameter $c_\alpha$, probability measure $\sigma$ on
 $\mathcal{B}(\Sp^{d-1})$ and L\'evy measure $\nu_{\bm{Z}}$ in~\eqref{eq:Levy-measure-stable}.

\begin{remark}\label{rem:comono_coup_S_R}
Under Assumption~(\nameref{asm:C}), we decompose the process $\bm{X}$ as the sum $\bm{S}+\bm{R}$ of the independent L\'evy processes $\bm{S}$ and $\bm{R}$ with generating triplets $(\bm{\gamma_S},\bm{0},\nu_{\bm{X}}^{\co})$ and $(\bm{\gamma_R},\bm{0},\nu_{\bm{X}}^{\mathrm{d}})$, respectively, such that, when $\alpha\in(0,1)$, both processes have zero natural drift, and when $\alpha \in (1,2)$ then $\bm{R}$ has zero natural drift.
Let the processes $(\bm{M}^{\bm{S}^t},\bm{M^Z},\bm{L}^{\bm{S}^t},\bm{L^Z})$ be coupled as in~\eqref{eq:comono_coupling_1} and~\eqref{eq:comono_coupling_2} from Section~\ref{sec:comonotonic_coupling}, where $\bm{S}^t=(\bm{S}_{st}/g(t))_{s\in[0,1]}$ and $\bm{R}^t=(\bm{R}_{st}/g(t))_{s\in[0,1]}$ for $t\in(0,1]$. Note that $\bm{X}^t$ has the same law as $\bm{S}^t+\bm{R}^t$ for $t\in(0,1]$ and that, under Assumption~(\nameref{asm:C}), $\bm{S}$ has a finite $q$-moment for every $q\in(0,\alpha)$ by~\eqref{eq:radial_tail_decomp}.
\end{remark}

\begin{theorem}
\label{thm:upper_bound_comono_tech}
Let $\alpha \in (0,2)\setminus\{1\}$, $\bm{X}$ and $\bm{Z}$ be as above and Assumptions~(\nameref{asm:S}) \& (\nameref{asm:C}) hold. Then, for every $q\in(0,1]$ with $\int_{\R^d\setminus B_{\bm{0}}(1)}|\bm{w}|^q\nu_{\bm{X}}^{\mathrm{d}}(\D\bm{w})<\infty$ 
we have $\E\big[\sup_{s\in[0,1]} |\bm{R}^t_s|^q\big] = \Oh \big(t^{1-q/\alpha}G(t)^{-q}\big)$ as $t\da 0$ and, if $p\ne\alpha-1$ and $q\in (0,\alpha)\cap(0,1]$ satisfy $q\notin\{\alpha/(p+1),\alpha/(\alpha\delta+1)\}$, we have, as $t \da 0$,
\begin{align}
    \mW_2\big(\bm{M}^{\bm{S}^t},\bm{M^Z}\big)
&=\Oh\big(G_2(t)+(1+G(t)^{p})t^{\min\{p/\alpha,\delta\}}\big),\label{eq:Up_bound_como_M}\\
\mW_q\big(\bm{L}^{\bm{S}^t}, \bm{L}^{\bm{Z}}\big) 
&= \Oh \big( G_2(t)^q
    + (1+G(t)^{pq})(1+G_1(x))^{q(1+p)}
        t^{\min\{pq/\alpha,q\delta,1-q/\alpha\}}\big), \qquad \text{ and }\label{eq:Up_bound_como_L}\\
|\bm{\varpi}_{\bm{S}^t}-\bm{\varpi}_{\bm{Z}}|
        &=\begin{dcases}
            \Oh\big(G_2(t)+t^{1-1/\alpha}G(t)^{-1} 
    +(1+G_1(t))t^{\min\{1-1/\alpha,\delta\}}\\
    \qquad +G(t)^{p}(1+G_1(t))^{1+p}t^{\min\{1-1/\alpha,p/\alpha\}}\big), & \alpha \in (1,2),\\
    \Oh\big(G_2(t)+t^{p/\alpha}G(t)^{p} +t^{\delta}\big), & \alpha \in (0,1).
        \end{dcases}\label{eq:Up_bound_como_varpi}
\end{align}
\end{theorem}

\begin{remark}\label{rem:thm_upper_lower_comno_tech} 
If $H\equiv 1$, then $G$ is constant and hence Theorem~\ref{thm:upper_bound_comono_tech} is also applicable to the domain of normal attraction: set $G_1\equiv 1$, $G_2\equiv0$ and $\delta$ arbitrarily large, then for $p\ne\alpha-1$ and $q\in (0,\alpha)\cap(0,1]$ with $\int_{\R^d\setminus B_{\bm{0}}(1)}|\bm{w}|^q\nu_{\bm{X}}^{\mathrm{d}}(\D\bm{w})<\infty$ and $q\ne\alpha/(p+1)$, we have, as $t \da 0$, 
\begin{align*}
\mW_2\big(\bm{M}^{\bm{S}^t},\bm{M^Z}\big) 
    &=\Oh\big(t^{p/\alpha}\big),\quad
    \mW_q\big(\bm{L}^{\bm{S}^t},\bm{L^Z}\big)
    =\Oh\big(t^{\min\{ pq/\alpha,1-q/\alpha\}}\big), \quad \text{and}\\
        |\bm{\varpi}_{\bm{S}^t}-\bm{\varpi}_{\bm{Z}}|&= \begin{dcases}
            \Oh\big(t^{\min\{1-1/\alpha, p/\alpha\}}\big), & \alpha \in (1,2),\\
            \Oh(t^{p/\alpha}),& \alpha\in(0,1).
        \end{dcases}
\end{align*}
Note that, when $q(p+1)>\alpha$, these rates match the ones in Theorem~\ref{thm:d_thin_dom_attract}, established under more general conditions using the thinning coupling.
\end{remark}

\begin{lemma}
\label{lem:bounds_h_2_multidim}
Under Assumption~(\nameref{asm:C}), there exists a function $\wt h:(0,\infty)\times \Sp^{d-1}\to(-1,\infty)$ and a constant $K_{\wt h} \ge 0$, such that, for all $x>0$ and $\bm{v} \in \Sp^{d-1}$, 
\begin{equation}
\label{eq:rho_X^la}
\rho_{\bm{X}}^{\co\la}(x,\bm{v})=\left( c_\alpha/c\right)^{1/\alpha}x^{-1/\alpha}G(1/x)(1+\wt h(x,\bm{v})) \quad \text{and}\quad 
|\wt h(x,\bm{v})|\le K_{\wt h}\big(1\wedge (x^{-p/\alpha}G(1/x)^p+x^{-\delta})\big).
\end{equation}
\end{lemma}

\begin{proof}
Note that $\rho_{\bm{X}}^{\co}([\rho_{\bm{X}}^{\co\la}(x,\bm{v}),\infty),\bm{v})
=x$ for all $\bm{v}\in \Sp^{d-1}$ and $x>0$. Hence, for all $\bm{v}\in \Sp^{d-1}$ and $x>0$,
\begin{align*}
x
&=\frac{c_\alpha}{\alpha}(1+h(\rho_{\bm{X}}^{\co\la}(x,\bm{v}),\bm{v}))H(\rho_{\bm{X}}^{\co\la}(x,\bm{v}))^\alpha\rho_{\bm{X}}^{\co\la}(x,\bm{v})^{-\alpha},
\quad\text{implying}\\
\rho_{\bm{X}}^{\co\la}(x,\bm{v})
&=\Big(\frac{c_\alpha}{\alpha}\Big)^{1/\alpha}x^{-1/\alpha}H(\rho_{\bm{X}}^{\co\la}(x,\bm{v}))(1+h(\rho_{\bm{X}}^{\co\la}(x,\bm{v}),\bm{v}))^{1/\alpha}\\
&=\Big(\frac{c_\alpha}{\alpha}\Big)^{1/\alpha}x^{-1/\alpha}G(1/x)\frac{H(\rho_{\bm{X}}^{\co\la}(x,\bm{v}))}{G(1/x)}\big(1+h(\rho_{\bm{X}}^{\co\la}(x,\bm{v}),\bm{v})\big)^{1/\alpha}.
\end{align*}
Thus, the first part of~\eqref{eq:rho_X^la} holds if 
$\wt h(x,\bm{v})\coloneqq (H(\rho_{\bm{X}}^{\co\la}(x,\bm{v}))/G(1/x)) (1+h(\rho_{\bm{X}}^{\co\la}(x,\bm{v}),\bm{v}))^{1/\alpha}-1 \in (-1,\infty)$.

Suppose now that~\eqref{eq:old_assump_(H)} in Assumption~(\nameref{asm:C}) holds for some $p,\delta>0$. Since $h$ is bounded by $K_h$ and by~\eqref{eq:old_assump_(H)}, we obtain $\wt h(x,\bm{v})\le (K_Q+1) (1+K_h)^{1/\alpha}-1$ for all $x>0$ and $\bm{v}\in \Sp^{d-1}$. Moreover, the elementary inequality $|(1+x)^r-1|\le |x|$ for any $r\in[0,1]$ and $x\ge -1$ and the triangle inequality yield $|(1+y)(1+x)^r-1|\le|x|+|y|(1+x)^r$, implies for all $x>0$ and $\bm{v} \in \Sp^{d-1}$, that
\begin{align*}
|\wt h(x,\bm{v})|
&\le |h(\rho_{\bm{X}}^{\co\la}(x,\bm{v}),\bm{v})|+\bigg|\frac{H(\rho_{\bm{X}}^{\co\la}(x,\bm{v}))}{G(1/x)}-1\bigg|(1+K_h)^{1/\alpha}\\
&\le K_h \rho_{\bm{X}}^{\co\la}(x,\bm{v})^p+K_Q(1+K_h)^{1/\alpha}x^{-\delta} \\
&\le K_h(c_\alpha/\alpha)^{p/\alpha}x^{-p/\alpha}G(1/x)^p \left( \frac{H(\rho_{\bm{X}}^{\co\la}(x,\bm{v}))}{G(1/x)}\right)^p (1+h(\rho_{\bm{X}}^{\co\la}(x,\bm{v}),\bm{v}))^{p/\alpha}+K_Q(1+K_h)^{1/\alpha}x^{-\delta}\\
& \le K_h(1+K_h)^{p/\alpha}(1+K_Q)^p(c_\alpha/\alpha)^{p/\alpha}x^{-p/\alpha}G(1/x)^p+K_Q(1+K_h)^{1/\alpha}x^{-\delta}.
\end{align*}
Choosing $K_{\wt h}\coloneqq \max\{(c_\alpha/\alpha)^{p/\alpha}K_h(1+K_h)^{p/\alpha}(1+K_Q)^p,(K_Q+1)(1+K_h)^{1/\alpha},1\}$, gives the last part of~\eqref{eq:rho_X^la}.
\end{proof}

\begin{lemma}\label{lem:large_jump_comono_coup}
Let $\alpha \in (0,2)\setminus\{1\}$ and $q \in (0,\alpha)\cap(0,1]$ satisfy $\int_{\R^d\setminus B_{\bm{0}}(1)}|\bm{w}|^q\nu_{\bm{X}}(\D\bm{w})<\infty$, where $q\ne\alpha/(p+1)$ and $q\ne\alpha/(\alpha\delta+1)$. Then, under Assumptions~(\nameref{asm:S}) \& (\nameref{asm:C}) we have, as $t \da 0$,
\begin{align*}
\mW_q\big(\bm{L}^{\bm{S}^t}, \bm{L}^{\bm{Z}}\big) 
&= \Oh\big(G_2(t)^q 
    +(1+G_1(t))^qt^{\min\{1-q/\alpha,q\delta\}}     
    +G(t)^{pq}(1+G_1(t))^{q(1+p)}t^{\min\{1-q/\alpha,pq/\alpha\}}\big).
\end{align*}
\end{lemma}

\begin{proof}
By Lemma~\ref{lem:bounds_h_2_multidim}, for all $x>0$, $t\in(0,1]$ and $\bm{v}\in \Sp^{d-1}$, it holds that 
\begin{equation}\label{eq:time_changed_right_inverse}
\rho_{\bm{S}^t}^{\la}(x,\bm{v})
=\rho_{\bm{X}}^{\co\la}(x/t,\bm{v})/g(t)
=\Big(\frac{c_\alpha}{\alpha}\Big)^{1/\alpha}x^{-1/\alpha} \frac{G(t/x)}{G(t)}\big(1+\wt h(x/t,\bm{v})\big).
\end{equation}
Hence, Proposition~\ref{prop:Lq-J-comonotonic} now implies that
\begin{align*}
\mW_q\big(\bm{L}^{\bm{S}^t},\bm{L}^{\bm{Z}}\big) 
&\le \int_{ \Sp^{d-1}}\int_0^\ve |\rho_{\bm{S}^t}^{\la}(x,\bm{v})-\rho_{\bm{Z}}^{\la}(x,\bm{v})|^q \D x \,\sigma(\D \bm{v})\\
&=\Big(\frac{c_\alpha}{\alpha}\Big)^{q/\alpha}\int_{\Sp^{d-1}}\int_0^\ve \bigg|\frac{G(t/x)}{G(t)}
    \big(1+\wt h(x/t,\bm{v})\big)-1\bigg|^q
    x^{-q/\alpha} 
        \D x \,\sigma(\D \bm{v})
\eqqcolon I(t). 
\end{align*}
To bound $I(t)$, we use the triangle inequality and the fact that $x \mapsto x^q$ is concave, to obtain
\begin{align*}
   \bigg(\frac{\alpha}{c_\alpha}\bigg)^{q/\alpha}I(t) 
   &\le \int_0^\ve x^{-q/\alpha} \left|\frac{G(t/x)}{G(t)}-1 \right|^q\D x
   + \int_{\Sp^{d-1}}\int_0^\ve x^{-q/\alpha}\left| \frac{G(t/x)}{G(t)}
   \wt h(x/t,\bm{v})\right|^q\D x\,\sigma(\D \bm{v}).
\end{align*} 
We consider each integral on its own. Assumption~(\nameref{asm:S}) implies that the first integral $I_1(t)$ in the display above is bounded by $G_2(t)^q\int_0^\ve x^{-q/\alpha}G_1(x)^q\D x <\infty$. Next, we bound the second integral $I_2(t)$ in the display above. Assumption~(\nameref{asm:S}) and~\eqref{eq:old_assump_(H)} yield, as $t\da 0$,
\begin{align*}
I_2(t)
&\lesssim \int_{\Sp^{d-1}}\int_0^\ve 
    (1+G_1(x)G_2(t))^q x^{-q/\alpha}  
    |\wt h(x/t,\bm{v})|^q\D x\,\sigma(\D \bm{v}) \\
&\lesssim \int_0^{t} (1+G_1(x))^q x^{-q/\alpha} \D x 
    + t^{q\delta}\int_t^\ve (1+G_1(x))^qx^{-q/\alpha-q\delta} \D x\\
&\qquad
    + t^{pq/\alpha}G(t)^{pq}\int_t^\ve (1+G_1(x))^q x^{-q/\alpha-pq/\alpha}\frac{G(t/x)^{pq}}{G(t)^{pq}} \D x\\
&\lesssim (1+G_1(t))^q t^{\min\{1-q/\alpha,q\delta\}}
    +  t^{pq/\alpha}G(t)^{pq}\int_t^\ve (1+G_1(x))^{q(1+p)}x^{-q/\alpha-pq/\alpha} \D x\\
&= \Oh\big((1+G_1(t))^qt^{\min\{1-q/\alpha,q\delta\}}     
    +G(t)^{pq}(1+G_1(t))^{q(1+p)}t^{\min\{1-q/\alpha,pq/\alpha\}}
    \big),
\end{align*}
giving the claim.
\end{proof}

\begin{lemma}\label{lem:small_jump_comono_coup}
Let $\alpha \in (0,2)\setminus\{1\}$. Under Assumptions~(\nameref{asm:S}) \& (\nameref{asm:C}) we have 
\begin{equation}
\label{eq:L2-M}
\mW_2\big(\bm{M}^{\bm{S}^t},\bm{M^Z}\big)
=\Oh\big(G_2(t)+t^{p/\alpha}G(t)^{p}+t^{\delta}\big), \qquad \text{as } t \da 0.
\end{equation}
\end{lemma}

\begin{proof}
Proposition~\ref{prop:small_jump_bound_comonotonic} together with~\eqref{eq:time_changed_right_inverse} shows that 
\begin{align*}
\mW_2\big(\bm{M}^{\bm{S}^t},\bm{M^Z}\big)^2 
&\le 4\int_{[\ve,\infty)\times \Sp^{d-1}}(\rho_{\bm{S}^t}^{\la}(x,\bm{v})-\rho_{\bm{Z}}^{\la}(x,\bm{v}))^2 \D x \otimes \sigma(\D \bm{v})\\
&=4\bigg(\frac{c_\alpha}{\alpha}\bigg)^{2/\alpha}\int_{\Sp^{d-1}}\int_\ve^\infty \left(\frac{G(t/x)}{G(t)}
    \big(1+\wt h(x/t,\bm{v})\big)-1\right)^2
    x^{-2/\alpha} \D x \sigma(\D \bm{v}) \eqqcolon I(t).
\end{align*}
To bound $I(t)$, we use the elementary inequality $(x+y)^2 \le 2(x^2+y^2)$, which implies,
\begin{align*}
   \frac{1}{8}\bigg(\frac{\alpha}{c_\alpha}\bigg)^{2/\alpha} I(t) &\le\int_\ve^\infty x^{-2/\alpha} \left(\frac{G(t/x)}{G(t)}-1 \right)^2\D x
   + \int_{\Sp^{d-1}}\int_\ve^\infty x^{-2/\alpha}\left( \frac{G(t/x)}{G(t)}
   \wt h(x/t,\bm{v})\right)^2\D x\,\sigma(\D \bm{v}).
\end{align*} 
By Assumption~(\nameref{asm:S}), the first integral $I_1(t)$ above is bounded by $G_2(t)^2\int_\ve^\infty x^{-2/\alpha}G_1(x)^2\D x<\infty$. Assumption~(\nameref{asm:S}) and~\eqref{eq:rho_X^la} imply, as $t \da 0$,
\begin{align*}
    I_2(t) &\lesssim t^{2p/\alpha}G(t)^{2p}\int_\ve^\infty (1+G_1(x))^2 \frac{G(t/x)^{2p}}{G(t)^{2p}}x^{-2/\alpha-2p/\alpha}\D x +t^{2\delta}\int_\ve^\infty (1+G_1(x))^2x^{-2/\alpha-2\delta}\D x\\
    &\lesssim t^{2p/\alpha}G(t)^{2p}\int_\ve^\infty (1+G_1(x))^{2(1+p)}x^{-2/\alpha-2p/\alpha}\D x +t^{2\delta}=\Oh(t^{2p/\alpha}G(t)^{2p} +t^{2\delta}).
\end{align*}
This completes the proof.
\end{proof}

In the following lemma, we find at what rate the drifts converge. 
\begin{lemma}\label{lem_drift_como_conv}
\noindent\nf{(a)} Let $\alpha \in (1,2)$, and $p,\delta>0$ where $p\ne\alpha-1$ and $\delta\ne(\alpha-1)/\alpha$. Then, under Assumptions~(\nameref{asm:S}) \& (\nameref{asm:C}), we have, as $t \da 0$,
    \begin{align*}
        |\bm{\varpi}_{\bm{S}^t}-\bm{\varpi}_{\bm{Z}}|&=\Oh\big(G_2(t)+t^{1-1/\alpha}G(t)^{-1} 
    +(1+G_1(t))t^{\min\{1-1/\alpha,\delta\}}
    \\
    &\qquad + G(t)^{p}(1+G_1(t))^{1+p}t^{\min\{1-1/\alpha,p/\alpha\}}\big).
    \end{align*}
\noindent\nf{(b)} Let $\alpha \in (0,1)$. Then, under Assumptions~(\nameref{asm:S}) \& (\nameref{asm:C}) we have, as $t \da 0$, 
\begin{equation*}
|\bm{\varpi}_{\bm{S}^t}-\bm{\varpi}_{\bm{Z}}|
=\Oh\big(G_2(t)+t^{p/\alpha}G(t)^{p} +t^{\delta}\big).
\end{equation*}
\end{lemma}
\begin{proof}
First, assume that $\alpha \in (1,2)$. The proof in this setting follows the steps of the proof of Lemma~\ref{lem:large_jump_comono_coup}. Note that
    \begin{align*}
        \bm{\varpi}_{\bm{S}^t}&= t^{1-1/\alpha}G(t)^{-1}(\bm{\varpi}_{\bm{S}}-\E[\bm{S}_1])-\int_{\Sp^{d-1}}\int_0^\ve \rho_{\bm{S}^t}^\la(x,\bm{v})\D x\sigma(\D\bm{v}),\\ \bm{\varpi}_{\bm{Z}}&=
-\int_{\Sp^{d-1}}\int_0^\ve \rho_{\bm{Z}}^\la(x,\bm{v})\D x\sigma(\D\bm{v}),
    \end{align*}
    and since $\bm{S}$ has a finite first moment, it follows that
\begin{equation*}
\bm{\varpi}_{\bm{S}^t}-\bm{\varpi}_{\bm{Z}}
= t^{1-1/\alpha}G(t)^{-1}(\bm{\varpi}_{\bm{S}}+\E[\bm{S}_1])
-\int_{\Sp^{d-1}}\int_0^\ve \rho_{\bm{S}^t}^\la(x,\bm{v})- \rho_{\bm{Z}}^\la(x,\bm{v})\D x\sigma(\D\bm{v}).
\end{equation*} Recall from \eqref{eq:time_changed_right_inverse}, that $\rho_{\bm{S}^t}^{\la}(x,\bm{v})
=(c_\alpha/\alpha)^{1/\alpha}x^{-1/\alpha}
\big(1+\wt h(x/t,\bm{v})\big)G(t/x)/G(t)$, which implies that
\begin{align*}
  & \bigg|\int_{\Sp^{d-1}}\int_0^\ve \rho_{\bm{S}^t}^\la(x,\bm{v})- \rho_{\bm{Z}}^\la(x,\bm{v})\D x\sigma(\D\bm{v})\bigg| 
   \le \int_{ \Sp^{d-1}}\int_0^\ve |\rho_{\bm{S}^t}^{\la}(x,\bm{v})-\rho_{\bm{Z}}^{\la}(x,\bm{v})| \D x \sigma(\D \bm{v})\\
   &\qquad = \bigg(\frac{c_\alpha}{\alpha}\bigg)^{1/\alpha}\int_{\Sp^{d-1}}\int_0^\ve \bigg|\frac{G(t/x)}{G(t)}
    \big(1+\wt h(x/t,\bm{v})\big)-1\bigg|
    x^{-1/\alpha} 
        \D x \sigma(\D \bm{v}) \eqqcolon I(t).
\end{align*}
The triangle inequality now implies, that
\begin{align*}
   \bigg(\frac{\alpha}{c_\alpha}\bigg)^{1/\alpha}I(t) 
   &\le \int_0^\ve x^{-1/\alpha} \left|\frac{G(t/x)}{G(t)}-1 \right|\D x
   + \int_{\Sp^{d-1}}\int_0^\ve x^{-1/\alpha}\left| \frac{G(t/x)}{G(t)}
   \wt h(x/t,\bm{v})\right|\D x\,\sigma(\D \bm{v}).
\end{align*} 
The two terms in the upper bound are denoted by $I_1(t)$ and $I_2(t)$. Following the calculations in the proof of Lemma~\ref{lem:large_jump_comono_coup}, we see by Assumption~(\nameref{asm:S}) and~\eqref{eq:rho_X^la}, that $I_1(t)$ in the display above is bounded by $G_2(t)\int_0^\ve x^{-1/\alpha}G_1(x)\D x <\infty$, and
\begin{equation*}
I_2(t)= \Oh\big((1+G_1(t))t^{\min\{1-1/\alpha,\delta\}}     
    +G(t)^{p}(1+G_1(t))^{1+p}t^{\min\{1-1/\alpha,p/\alpha\}}
    \big)\qquad\text{as $t\da 0$.}
\end{equation*}

Assume $\alpha \in (0,1)$. Recall $\bm{\varpi}_{\bm{S}^t}= \int_{\Sp^{d-1}}\int_\ve^\infty \rho_{\bm{S}^t}^\la(x,\bm{v})\D x\,\sigma(\D\bm{v})$ and $\bm{\varpi}_{\bm{Z}}=\int_{\Sp^{d-1}}\int_\ve^\infty \rho_{\bm{Z}}^\la(x,\bm{v})\D x\,\sigma(\D\bm{v})$ and hence, by Lemma~\ref{lem:bounds_h_2_multidim} and~\eqref{eq:time_changed_right_inverse},
\begin{align*}
|\bm{\varpi}_{\bm{S}^t}-\bm{\varpi}_{\bm{Z}}|
&\le \int_{\Sp^{d-1}}\int_\ve^\infty|\rho_{\bm{S}^t}^{\la}(x,\bm{v})-\rho_{\bm{Z}}^{\la}(x,\bm{v})| \D x  \,\sigma(\D \bm{v})\\
&=\bigg(\frac{c_\alpha}{\alpha}\bigg)^{1/\alpha}\int_{\Sp^{d-1}}\int_\ve^\infty \left|\frac{G(t/x)}{G(t)}
    \big(1+\wt h(x/t,\bm{v})\big)-1\right|
    x^{-1/\alpha} \D x \,\sigma(\D \bm{v}) \eqqcolon I(t).
\end{align*}
Bounding $I(t)$ using the triangle inequality, yields
\begin{align*}
   \bigg(\frac{\alpha}{c_\alpha}\bigg)^{1/\alpha} I(t) &\le\int_\ve^\infty x^{-1/\alpha} \left|\frac{G(t/x)}{G(t)}-1 \right|\D x
   + \int_{\Sp^{d-1}}\int_\ve^\infty x^{-1/\alpha}\left| \frac{G(t/x)}{G(t)}
   \wt h(x/t,\bm{v})\right|\D x\,\sigma(\D \bm{v}),
\end{align*} where the upper bound is denoted $I_1(t)+I_2(t)$.
Assumption~(\nameref{asm:S}) and~\eqref{eq:old_assump_(H)} with Lemma~\ref{lem:bounds_h_2_multidim}, imply the inequalities $I_1(t)\le G_2(t)\int_\ve^\infty x^{-1/\alpha}G_1(x)\D x<\infty$ and 
\begin{align*}
    I_2(t)& \lesssim t^{p/\alpha}G(t)^{p}\int_\ve^\infty (1+G_1(x)) \frac{G(t/x)^{p}}{G(t)^{p}}x^{-1/\alpha-p/\alpha}\D x +t^{\delta}\int_\ve^\infty (1+G_1(x))x^{-1/\alpha-\delta}\D x\\
    &=\Oh\big(t^{p/\alpha}G(t)^{p} +t^{\delta}\big), \qquad \text{ as } t \da 0.
\end{align*}
Assembling all the inequalities completes the proof.
\end{proof}

\begin{proof}[Proof of Theorem~\ref{thm:upper_bound_comono_tech}]
The bound on $\bm{R}^t$ follows directly from its definition and Lemma~\ref{lem:moment_bound} with $\beta_+=0$. The bounds on the big-jump components, the small-jump components and the drifts, follow directly from Lemmas~\ref{lem:large_jump_comono_coup},~\ref{lem:small_jump_comono_coup} \&~\ref{lem_drift_como_conv}, respectively.
\end{proof}

\subsection{Brownian limits: upper bounds}
\label{subsec:Brownian_limits}

In this subsection, we construct upper bounds on the distance between a L\'evy process with nonzero Gaussian component and its attracting Brownian motion. 
Recall that 
$(\bm{\gamma_X},\bm{\Sigma_X}\bm{\Sigma_X}^\tra,\nu_{\bm{X}})$
denotes the characteristic triplet~\cite[Def.~8.2]{MR3185174} of $\bm{X}$
 with respect to the cutoff function $\bm{w}\mapsto\1_{B_{\bm{0}}(1)}(\bm{w})$
 on $\bm{w}\in\R^d$ and $\beta_+$ is given in terms of the BG index defined in~\eqref{eq:BG}.

\begin{proposition}\label{prop:conv_BM_limit}
Let $\bm{X}$ be a L\'evy process on $\R^d$ with 
the characteristic triplet $(\bm{\gamma_X},\bm{\Sigma_X}\bm{\Sigma_X}^\tra,\nu_{\bm{X}})$.
Let $\bm{X}^t=(\bm{X}_{st}/\sqrt{t})_{s\in[0,1]}$ for $t\in (0,1]$
and assume 
 $\int_{\R^d\setminus B_{\bm{0}}(1)}|\bm{w}|^q\nu_{\bm{X}}(\D\bm{w})<\infty$
 for some $q\in (0,2]$. 
 Let $\bm{\Sigma_X}\bm{B}$ be the
Gaussian component of $\bm{X}$ in its L\'evy--It\^o decomposition~\eqref{eq:levy_ito_decomp} and define 
  $\bm{S}\coloneqq \bm{X}-\bm{\Sigma_X}\bm{B}$. 

\noindent{\normalfont(a)} If 
$\bm{S}$ is 
of infinite variation or has finite variation with infinite activity and zero natural drift, then
\begin{align*}
\mW_q\big(\bm{X}^t,\bm{\Sigma_X}\bm{B}\big)
=\Oh\big(t^{(q\wedge 1)
    (\min\{1/q,1/\beta_+\}-1/2)}\big), 
    \qquad \text{ as }t \da 0.
\end{align*}

\noindent{\normalfont(b)} If $\bm{S}$ has finite variation and nonzero natural drift, then 
\begin{align*}
\mW_q\big(\bm{X}^t,\bm{\Sigma_X}\bm{B}\big)
=\Oh\big(t^{(q\wedge 1)
    \min\{1/q-1/2,1/2\}}\big), 
    \qquad \text{ as }t \da 0.
\end{align*}
\end{proposition}

Note that, if $\bm{S}$ has infinite activity we have $\beta_+>0$ and if the BG index $\beta<2$,  then $1/\beta_+>1/2$. Hence, Proposition~\ref{prop:conv_BM_limit} provides bounds on the rate of convergence in the appropriate $L^q$-Wasserstein distance for the weak limit in Theorem~((\nameref{thm:small_time_domain_stable})) (case $\alpha=2$ and $G$ asymptotically constant). 
In the case $\beta=2$, it is well known that $\bm{X}^t_1$ converges weakly to the Gaussian law of
$\bm{\Sigma_X}\bm{B}_1$ (see e.g.~\cite[Prop.~I.2(i)]{MR1406564}), but the convergence 
of
$\mW_q\big(\bm{X}^t,\bm{\Sigma_X}\bm{B}\big)$
could be arbitrarily slow, see Example~\ref{ex:arbitrarily_slow_BM_conv} below. 
It is thus not surprising that Proposition~\ref{prop:conv_BM_limit} gives no information about the rate of convergence. 
Note also that
Proposition~\ref{prop:conv_BM_limit} covers the case $\bm{\Sigma_X}=\bm{0}$.
Moreover, the bound on the $L^q$-Wasserstein distance when the BG index is less than one is sharper if the natural drift is zero, than if it is not.

\begin{proof}[Proof of Proposition~\ref{prop:conv_BM_limit}]
Fix $t\in(0,1]$, let $\bm{B}$ be the Brownian motion in the L\'evy--It\^o decomposition~\eqref{eq:levy_ito_decomp} of the L\'evy process $\bm{X}$ and recall $\bm{X}^t=(\bm{X}_{st}/\sqrt{t})_{s\in[0,1]}$. Since the Brownian motion $\bm{B}$ satisfies the identity in law $(t^{-1/2}\bm{B}_{st})_{s\in[0,1]} \eqd (\bm{B}_{s})_{s\in[0,1]}$ by self-similarity (L\'evy's characterisation theorem), there exists a coupling $(\bm{X}^t,\bm{B}')$, such that $\bm{B}'=(t^{-1/2}\bm{B}_{st})_{s\in[0,1]}$ and $\bm{B}'\eqd \bm{B}$. Recalling $\bm{S}=\bm{X}-\bm{\Sigma}_{\bm{X}}\bm{B}$, we obtain
\begin{align}
\label{eq:trivial_coupling_bound}
\mW_q(\bm{X}^t,\bm{\Sigma}_{\bm{X}}\bm{B})^{q \vee 1} &\le \E\bigg[\sup_{s \in [0,1]}|\bm{\Sigma}_{\bm{X}}\bm{B}_{st}/\sqrt{t}+\bm{S}_{st}/\sqrt{t}-\bm{\Sigma}_{\bm{X}}\bm{B}'_s|^q\bigg]
=t^{-q/2}\E\bigg[\sup_{s \in [0,t]}|\bm{S}_s|^q\bigg].
\end{align} 
Note that the characteristic triplet 
$(\bm{\gamma_Y},\bm{0},\nu_{\bm{S}})$
of $\bm{S}$
is given by 
$\bm{\gamma_Y}=\bm{\gamma_X}$ and $\nu_{\bm{S}}=\nu_{\bm{X}}$. In particular,  the BG index of $\bm{S}$ equals that of $\bm{X}$. 

\underline{Part \nf{(a)}.} Assume $\bm{S}$ is of infinite variation. Since, $\bm{S}$ has no Gaussian component, by~\cite[Thm~21.9]{MR3185174} we have $\int_{\R^d_{\bm{0}}}|\bm{w}| \1_{B_{\bm{0}}(1)}(\bm{w})\nu_{\bm{X}}(\D \bm{w})=\infty$, implying that the BG index of $\bm{S}$ satisfies  $\beta\ge1$. Hence $\beta_+\in [\beta,2]$ satisfies: $\min\{1/q,1/\beta_+\}-1/2\le 1/\beta_+-1/2\le 1/2$. Thus, 
$t^{-q/2}\E[\sup_{s \in [0,t]}|\bm{S}_s|^q] \le C_2t^{q/2}+C_3t^{q(\min\{1/q,1/\beta_+\}-1/2)}$ by Lemma~\ref{lem:moment_bound},
implying that $$\text{$\mW_q(\bm{X}^t,\bm{\Sigma}_{\bm{X}}\bm{B})^{q \vee 1} \le 2\max\{C_2,C_3\}t^{q(\min\{1/\beta_+,1/q\}-1/2)}$ for $t\in(0,1]$.}$$

If $\bm{S}$ has finite variation and zero natural drift, then the bound in Lemma~\ref{lem:moment_bound} with $C_3=0$ yields  $\mW_q(\bm{X}^t,\bm{\Sigma}_{\bm{X}}\bm{B})^{q \vee 1} \le C_3 t^{q(\min\{1/q,1/\beta_+\}-1/2)}$. 
Noting that
$q/(q \vee 1)=q \wedge 1$ implies Part~(a).

\underline{Part (b).} 
Since 
$\bm{S}$ has finite variation, by definition~\eqref{eq:BG} and~\cite[Thm~21.9]{MR3185174}, we have $\beta\in[0,1]$ and $I_1<\infty$, thus implying $\beta_+\in[0,1]$.
By Lemma~\ref{lem:moment_bound}
applied to $\bm{S}$, we find
$t^{-q/2}\E[\sup_{s\in[0,t]}|\bm{S}_s|^q]
\le 
C_2t^{q/2}+C_3t^{q(\min\{1/q,1/\beta_+\}-1/2)}$ for $t\in(0,1]$.
Since $1\le 1/\beta_+$, we have $1/2\le 1/\beta_+-1/2$ 
and 
$ \min\{1/q-1/2,1/2\}\le \min\{1/q,1/\beta_+\}-1/2$.
Thus, for any $\beta_+\in[0,1]$,
by~\eqref{eq:trivial_coupling_bound} we get
$\mW_q(\bm{X}^t,\bm{\Sigma}_{\bm{X}}\bm{B})^{q \vee 1} \le 2\max\{C_2,C_3\} t^{q\min\{1/q-1/2,1/2\}}$. 
As in Part~(a), note 
that $q/(q \vee 1)=q \wedge 1$,
implying the claim in Part~(b).
\end{proof}

\section{Lower bounds on the Wasserstein distance in the domain of attraction}
\label{sec:lower_bounds}
In this section we prove the lower bounds from Theorems~\ref{thm:upper_lower_simple}, \ref{thm:upper_lower_general} \&~\ref{thm:conv_BM_limit}. We first cover the domain of non-normal attraction and then turn to the domain of normal attraction.

\subsection{Domain of non-normal attraction}
\label{subsec:prokhorov_lower_bound}

The lower bound on the rate 
of decay of $\mW_q(\bm{X}^t,\bm{Z})$ as $t\downarrow0$
is much greater than polynomial when the scaling function $g(t)=t^{1/\alpha}G(t)$ is such that $G$,  which is slowly varying at $0$, does not convergent to a positive constant (i.e. the process $\bm{X}$ is in the domain of non-normal attraction). To show this, we start with the following result, which can be viewed as an extension of~\cite[Thm~1]{MR3806899} from random walks to multidimensional L\'evy processes, stated for the $L^q$-Wasserstein distance. (We remark here that an extension for the Prokhorov distance, used in~\cite{MR3806899}, is also possible in this context.) Our proof below was inspired by that of~\cite[Thm~1]{MR3806899}. 
Our main tool in the proof of the lower bound in Theorem~\ref{thm:upper_lower_general} is the following.

\begin{proposition}
\label{prop:W_lower} 
Let $\bm{X}$ be in the domain of non-normal attraction of an $\alpha$-stable process $\bm{Z}$ with $\alpha\in(0,2]$ and define $a(t)\coloneqq G(2t)/G(t)$ for $t>0$. Then, for all $t>0$ and $q\in(0,1]\cap(0,\alpha)$, 
\begin{align*}
2^{1-q/\alpha}\mW_q(\bm{X}^t_1,\bm{Z}_1)
+a(t)^q\mW_q(\bm{X}^{2t}_1, \bm{Z}_1) 
\ge 
|1-a(t)^q|\E\big[|\bm Z_1|^q\big].
\end{align*} 
\end{proposition}

 In Lemma~\ref{lem:d_P_bound} we state some well-known facts used in the proof of Proposition~\ref{prop:W_lower}.

\begin{lemma}\label{lem:d_P_bound}
\noindent{\nf(a)} Let $\bm\xi$ be a random vector in $L^q$, i.e. $\E[|\bm\xi|^q]<\infty$, for some $q\in(0,1]$. Then, 
$$\mW_q(\bm{\xi},a\bm{\xi})\ge |1-a^q|\E\big[|\bm\xi|^q\big] \quad \text{ for any constant }a \in(0,\infty).$$

\noindent{\nf(b)} Assume that the random vectors
$\bm\xi_1, \bm\xi_2, \bm\zeta_1, \bm\zeta_2$ are in $L^q$, for some $q\in(0,1]$, and that 
$(\bm{\xi}_1,\bm{\zeta}_1)$ and $(\bm{\xi}_2, \bm{\zeta}_2)$ are independent. Then the following inequality holds:
\begin{equation*}
\mW_q(\bm{\xi}_1+\bm{\xi}_2,\bm{\zeta}_1+\bm{\zeta}_2)
\le \mW_q(\bm{\xi}_1,\bm{\zeta}_1) + \mW_q(\bm{\xi}_2,\bm{\zeta}_2).
\end{equation*}

\end{lemma}

\begin{proof}
\noindent (a) By the subadditivity of $t\mapsto t^q$ on $\R_+$, we have $|\bm x|^q\le (|\bm y|+|\bm x-\bm y|)^q\le |\bm y|^q+|\bm x-\bm y|^q$ for any $\bm x,\bm y\in\R^d$. A similar inequality holds by reversing the roles of $\bm x$ and $\bm y$, implying $|\bm x-\bm y|^q\ge ||\bm x|^q-|\bm y|^q|$. Hence, we have 
\[
\mW_q(\bm\xi,a\bm\xi)
=\inf_{(\bm\xi,\bm\zeta),
    \,\bm\zeta\eqd a\bm\xi}
    \E\big[|\bm\xi-\bm\zeta|^q\big]
\ge \inf_{(\bm\xi,\bm\zeta),
    \,\bm\zeta\eqd a\bm\xi}
    \big|\E\big[|\bm\xi|^q\big]-\E\big[|\bm\zeta|^q\big]\big|
=|1-a^q|\E\big[|\bm\xi|^q\big].
\]

\noindent (b) By~\cite[Thm~4.1]{villani2008optimal} and~\cite[Main Thm]{MR3393269} there exist minimal couplings $(\bm\xi_1, \bm\zeta_1)$ and $(\bm\xi_1, \bm\zeta_2)$, satisfying $\E[|\bm\xi_1-\bm\zeta_1|^q]=\mW_q(\bm\xi_1,\bm\zeta_1)$ and $\E[|\bm\xi_2-\bm\zeta_2|^q]=\mW_q(\bm\xi_2,\bm\zeta_2)$. The product of these two probability spaces yields a coupling of all four vectors $\bm\xi_1, \bm\xi_2, \bm\zeta_1, \bm\zeta_2$, such that $(\bm{\xi}_1,\bm{\zeta}_1)$ and $(\bm{\xi}_2, \bm{\zeta}_2)$ are independent. Thus,
\[
\mW_q(\bm{\xi}_1+\bm{\xi}_2,
    \bm{\zeta}_1+\bm{\zeta}_2)  
\le \E[|\bm\xi_1+\bm\xi_2-\bm\zeta_1-\bm\zeta_2|^q] 
 \le \E[|\bm\xi_1-\bm\zeta_1|^q] + \E[|\bm\xi_2-\bm\zeta_2|^q] =  \mW_q(\bm{\xi}_1,\bm{\zeta}_1) + \mW_q(\bm{\xi}_2,\bm{\zeta}_2),
\]
as claimed.
\end{proof}

\begin{proof}[Proof of Proposition~\ref{prop:W_lower}]
Recall $\bm{X}^t_1 = \bm{X}_t/g(t)$, and note that $\bm{X}_{2t}=\bm{X}_{2t}-\bm{X}_t+\bm{X}_t$, where $\bm{X}_{2t}-\bm{X}_t$ and $\bm{X}_t$ are independent and equal in distribution. Furthermore let $\bm{Z}^{(1)},\bm{Z}^{(2)}$ be independent copies of $\bm{Z}$. Recall that $g(t)=t^{1/\alpha}G(t)$ and note that $\bm{Z}_1\eqd 2^{-1/\alpha}\bm{Z}_1^{(1)}+2^{-1/\alpha}\bm{Z}^{(2)}_1$. This together with Lemma~\ref{lem:d_P_bound}(b) implies that
\begin{align*}
\mW_q\left( \frac{G(2t)}{G(t)} \bm{X}^{2t}_1,\bm{Z}_1\right)
&=\mW_q\left( \frac{\bm{X}_{2t}-\bm{X}_t}{(2t)^{1/\alpha}G(t)}+\frac{\bm{X}_{t}}{(2t)^{1/\alpha}G(t)},\frac{\bm{Z}^{(1)}_1}{2^{1/\alpha}}+\frac{\bm{Z}^{(2)}_1}{2^{1/\alpha}}\right) \\
&\le 2\mW_q\left(\frac{\bm{X}_{t}}{(2t)^{1/\alpha}G(t)},\frac{\bm{Z}^{(1)}_1}{2^{1/\alpha}}\right)
=2^{1-q/\alpha}\mW_q\left(\bm{X}^t_1,\bm{Z}_1\right).
\end{align*} 
The scaling property for the $\mW_q$-distance implies that $\mW_q(a(t) \bm{X}^{2t}_1 ,a(t) \bm{Z}_1 )=  a(t)^q\mW_q( \bm{X}^{2t}_1, \bm{Z}_1 )$. Putting everything together and applying the triangle inequality with Lemma~\ref{lem:d_P_bound}(a), yields
\begin{align*}
2^{1-q/\alpha}\mW_q(\bm{X}^t_1,\bm{Z}_1)
+ a(t)^q\mW_q(\bm{X}^{2t}_1, \bm{Z}_1) 
&\ge \mW_q\left( a(t) \bm{X}^{2t}_1,\bm{Z}_1\right)
+\mW_q\left( a(t)\bm{X}^{2t}_1,  a(t)\bm{Z}_1 \right)\\
&\ge \mW_q \left( \bm{Z}_1,a(t)\bm{Z}_1\right)
\ge |1-a(t)^q|\E\big[|\bm Z_1|^q\big],
\end{align*}
completing the proof.
\end{proof}

\subsection{Domain of normal attraction and the Toscani-Fourier lower bounds}
\label{subsec:lower_bounds_toscani_fourier}

We begin with the following technical result, used in the proofs of Theorems~\ref{thm:upper_lower_simple} \&~\ref{thm:conv_BM_limit}. Given two $d$-dimensional random vectors $\bm{\xi}$ and $\bm{\zeta}$ with characteristic functions $\varphi_{\bm{\xi}}(\bm{u})
\coloneqq\E[\exp(i\langle\bm{u},\bm{\xi}\rangle)]$ and $\varphi_{\bm{\zeta}}(\bm{u})\coloneqq\E[\exp(i\langle\bm{u},\bm{\zeta}\rangle)]$, respectively, and define the Toscani--Fourier distance (see~\cite[Eq.~(1)]{MR4170640}) as 
\[
T_s(\bm{\xi},\bm{\zeta})\coloneqq\sup_{\bm{u} \in \R^d_{\bm{0}}} \frac{|\varphi_{\bm{\xi}}(\bm{u})-\varphi_{\bm{\zeta}}(\bm{u})|}{|\bm{u}|^s},
\quad\text{ for } s>0.
\]
The following lemma is an extension of \cite[Prop.~2]{MR3833470} to the multivariate case and to $L^q$-Wasserstein distances for $q\in(0,1]$, and the proof is inspired by the proof in the one-dimensional case. For completeness, we give a simple proof below. 

\begin{lemma}\label{lem:ddim_lowbound_wass}
For any random vectors $\bm{\xi},\bm{\zeta}$ and $q\in(0,1]$, we have $\mW_q(\bm{\xi},\bm{\zeta}) \ge 2^{q-1}T_q(\bm{\xi},\bm{\zeta})$.
\end{lemma}

\begin{proof}
Fix $q\in(0,1]$. Since the map $\psi:x\mapsto e^{ix}$, $x\in\R$, satisfies $|1-\psi(x)|\le 2\min\{|x/2|,1\}\le 2|x/2|^q$ for $x\in\R$, we have $|\psi(x)-\psi(y)|\le 2^{1-q}|x-y|^q$ for any $x,y\in\R$. Hence, for any $\bm{u}\in\R^d\setminus\{\bm 0\}$, 
\[
\E\big[2^{1-q}|\bm\xi-\bm\zeta|^q\big]
\ge\frac{\E\big[2^{1-q}|\langle\bm{u},\bm\xi\rangle
    -\langle\bm{u},\bm\zeta\rangle|^q\big]}{|\bm{u}|^q}
\ge\frac{\E\big[|\psi(\langle\bm{u},\bm\zeta\rangle)-\psi(\langle\bm{u},\bm\zeta\rangle)|\big]}{|\bm{u}|^q}
\ge\frac{|\varphi_{\bm\xi}(\bm u)-\varphi_{\bm\zeta}(\bm u)|}{|\bm{u}|^q}.
\]
Since $\bm{u}\in\R^d\setminus\{\bm 0\}$ is arbitrary, the result follows.
\end{proof}

\subsubsection{Heavy-tailed domain of normal attraction}
Let $(\bm{X}_t)_{t \ge 0}$ be a L\'evy process on $\R^d$ in the small-time domain of attraction of an $\alpha$-stable process $\bm{Z}$, such that $\bm{X}^t_1=\bm{X}_t/t^{1/\alpha} \cid \bm{Z}_1$ as $t \da 0$.

\begin{lemma}
\label{lem:lower_bound_alpha_stab}
Let $\bm{X}$ be a L\'evy process that differs in law from the $\alpha$-stable process~$\bm{Z}$, $\alpha\in(0,2]$. Let $\psi_{\bm{X}}$ and $\psi_{\bm{Z}}$ denote their L\'evy--Khintchine exponents. Pick any $q\in(0,1]\cap(0,\alpha)$ and $\bm{u}_*\in \R^d\setminus\{\bm 0\}$ for which $C_*\coloneqq 2^{q-1}|\bm{u}_*|^{-q}|\psi_{\bm{X}_1}(\bm{u}_*)-\psi_{\bm{Z}_1}(\bm{u}_*)|> 0$. Then, we have
\begin{align*}
\mW_q(\bm{X}^t_1,\bm{Z}_1)
\ge C_*t^{1-q/\alpha}
+ \Oh(t^{2-q/\alpha}),
\qquad\text{as }t \da 0.
\end{align*}
\end{lemma}

\begin{proof}
First, Lemma~\ref{lem:ddim_lowbound_wass} implies that
\begin{align*}
\mW_q(\bm{X}^t_1,\bm{Z}_1) 
\ge 2^{q-1} T_q(\bm{X}^t_1,\bm{Z}_1)
\ge\frac{2^{q-1}}{|\bm{u}|^q}|\varphi_{\bm{X}^t_1}(\bm{u})-\varphi_{\bm{Z}_1}(\bm{u})|,
\qquad \text{for any }\bm{u}\in\R^d_{\bm{0}},
\end{align*} 
where $\varphi_{\bm\xi}$ denotes the characteristic function of the random vector $\bm{\xi}$. Second, set $\bm{u}=t^{1/\alpha}\bm{u}_*$ with $\bm{u}_*\in \R^d\setminus \{\bm{0}\}$ as in the statement of the lemma and note that 
\[
\varphi_{\bm{X}^t_1}(\bm{u})
=\E\big[\exp(i\langle \bm{X}^t_1,t^{1/\alpha}\bm{u}_*\rangle)\big]
=\E\big[\exp(i\langle \bm{X}_t/t^{1/\alpha},t^{1/\alpha}\bm{u}_*\rangle)\big]
=e^{t\psi_{\bm{X}}(\bm{u}_*)}.
\]
Similarly, since $\bm{Z}_1\eqd\bm{Z}_t/t^{1/\alpha}$, we have $\varphi_{\bm{Z}_1}(\bm{u})=\exp(t\psi_{\bm{Z}}(\bm{u}_*))$ and hence
\begin{align*}
\mW_q(\bm{X}^t_1,\bm{Z}_1)\ge \frac{2^{q-1}}{|\bm{u}_*|^q t^{q/\alpha}}\big| e^{t\psi_{\bm{X}}(\bm{u}_*)}-e^{t\psi_{\bm{Z}}(\bm{u}_*)}\big|.
\end{align*} 
Since for any $z \in \C$ we have $e^z=1+z+\Oh(z^2)$ as $|z|\to 0$, it follows that $|e^{at}-e^{bt}|=|at-bt+\Oh(t^2)|= |a-b|t + \Oh(t^2)$ for $a=\psi_{\bm{X}}(\bm{u}_*)$ and $b=\psi_{\bm{Z}}(\bm{u}_*)$. The result then follows.
\end{proof}

\subsubsection{Brownian domain of normal attraction}

The domain of normal attraction to a Brownian motion consists of the class of L\'evy processes with a nontrivial Brownian component (see e.g.~\cite{MR3784492} and~\cite[Prop.~I.2(i)]{MR1406564}). To construct a lower bound on the distance between the L\'evy process and its Brownian limit require the following lower estimates.
\begin{lemma}
\label{lem:tech_extent_lowerbound}
Let $\bm{Y}$ be a nonzero pure-jump L\'evy process on $\R^d$, let $\psi_{\bm{Y}}(\bm{u})$ denote its L\'evy-Khintchine exponent and $\nu_{\bm{Y}}$ its L\'evy measure.\\
{\normalfont(a)} If $\bm{Y}$ has finite variation and nonzero drift with direction $\bm{u}_*\in \Sp^{d-1}$. Then $|\psi_{\bm{Y}}(x\bm{u}_*)|\ge cx$ for some $c>0$ and all sufficiently large $x>0$.\\
{\normalfont(b)} Suppose there exist a locally finite measure $\rho$ on $(0,\infty)$ and a probability measure $\sigma$ on $\Sp^{d-1}$ with
\[
\nu_{\bm{Y}}(A)
\ge \int_{\Sp^{d-1}}\int_{(0,\infty)}\1_A(x\bm{v})\rho(\D x)\sigma(\D\bm{v}),
\qquad A\in\mathcal{B}(\R^d_{\bm{0}}).
\]
Define $\Upsilon(x):=\int_{(0,x)}r^2\rho(\D r)$ for $x>0$, and, given any $c\in (0,1)$, suppose that $C_{c,\bm{u}_*}\coloneqq\{\bm{v}\in\Sp^{d-1}:|\langle \bm{u}_*,\bm{v}\rangle|\ge c\}$ has positive $\sigma$-measure $m:=\sigma(C_{c,\bm{u}_*})>0$ for some $\bm{u}_*\in \Sp^{d-1}$. Then we have
\[
|\psi_{\bm{Y}}(x\bm{u}_*)|\ge \frac{c^2m}{3}r^2\Upsilon(x^{-1})
\qquad\text{for all }x>0.
\]
In particular, if $c_\delta:=\inf_{x\in(0,1)}x^{\delta-2}\Upsilon(x)>0$ for some $\delta\in(0,2)$, then $|\psi_{\bm{Y}}(x\bm{u}_*)|\ge (c_\delta c^2m/3)x^{\delta}$, $x>1$. 
\end{lemma}

\begin{proof}
Let $\Re z$ and $\Im z$ denote the real and imaginary parts of $z\in\C$, respectively.\\
(a) Since $\bm{Y}$ has finite variation, it is clear from the L\'evy-Khintchine formula without compensator that $|\psi_{\bm{Y}}(x\bm{u}_*)|\ge|\Im\psi_{\bm{Y}}(x\bm{u}_*)|\ge cx$ for some $c>0$ and all sufficiently large $x>0$. Indeed, this follows from~\cite[Prop.~2(ii)]{MR1406564} applied to the finite variation L\'evy process $\langle\bm{u_*},\bm{Y}\rangle$.\\
(b) Note at first, that $1-\Re e^{ix}=1-\cos(x)\ge \tfrac{1}{3}x^2\1_{\{|x|<1\}}$ for all $x \in\R$. Thus, the L\'evy-Khintchine formula applied to $\Re\psi$ yields 
\begin{align*}
3|\psi_{\bm{Y}}(x\bm{u}_*)|
&\ge 3|\Re \psi(x\bm{u}_*)|
\ge \int_{\R^d_{\bm{0}}}
    |\langle x\bm{u}_*,\bm{w}\rangle|^2
    \1_{\{|\langle x\bm{u}_*,\bm{w}\rangle|<1\}}\nu_{\bm{Y}}(\D\bm{w})\\
&\ge \int_{\R^d_{\bm{0}}}
    |\langle x\bm{u}_*,\bm{w}\rangle|^2
    \1_{\{x|\bm{w}|<1\}}\nu_{\bm{Y}}(\D\bm{w})
\ge \int_0^{1/x}\int_{C_{c,\bm{u}_*}}
    |\langle x\bm{u}_*,r\bm{v}\rangle|^2
    \sigma(\D\bm{v})\rho(\D r)\\
&\ge \int_0^{1/x}\int_{C_{c,\bm{u}_*}}
    c^2x^2r^2
    \sigma(\D\bm{v})\rho(\D r)
\ge c^2m\int_0^{1/x} x^2r^2\rho(\D r)
=c^2m x^2\Upsilon(x^{-1}).
\end{align*}
This proves the first claim in Part~(b). The second claim follows from the additional assumption.
\end{proof}

\begin{lemma}
\label{lem:lower_bound_gaussian_lim}
Let $\bm{X}$ be a L\'evy process on $\R^d$ with 
the characteristic triplet $(\bm{\gamma_X},\bm{\Sigma_X}\bm{\Sigma_X}^\tra,\nu_{\bm{X}})$.
For $t\in (0,1]$, denote
$\bm{X}^t=(\bm{X}_{st}/\sqrt{t})_{s\in[0,1]}$. Moreover, let $\bm{\Sigma_X}\bm{B}$ be the
Gaussian component of $\bm{X}$ in its L\'evy--It\^o decomposition~\eqref{eq:levy_ito_decomp} and define 
  $\bm{S}\coloneqq \bm{X}-\bm{\Sigma_X}\bm{B}$ with L\'evy--Khintchine exponent $\psi_{\bm{S}}$.\\
{\normalfont(a)} Pick any $\bm{u}_*\in\R^d_{\bm{0}}$ and define $C_*\coloneqq |\bm{u}_*|^{-1}|\psi_{\bm{S}_1}(\bm{u}_*)|>0$. Then, we have 
\begin{align*}
\mW_1(\bm{X}^t_1,\bm{\Sigma}_{\bm{X}}\bm{B}_1)
    \ge C_*\sqrt{t}+ \Oh(t^{3/2}),
\quad\text{as }t \da 0.
\end{align*}
{\normalfont(b)} Let $\lambda$ be the largest eigenvalue of $\bm{\Sigma}_{\bm{X}}\bm{\Sigma}^\tra_{\bm{X}}$. Suppose there exist $\delta\in[1,2)$ and vectors  $(\bm{u}_r)_{r\in(0,\infty)}$ on $\R^d_{\bm{0}}$ satisfying $|\bm{u}_r|=r$ and $c:=\inf_{r>1}r^{-\delta}|\psi_{\bm{S}}(\bm{u}_r)|>0$. Then for any $C_*\in(0,ce^{-\lambda/2})$ we have
\[
\mW_1(\bm{X}^t_1,\bm{\Sigma}_{\bm{X}}\bm{B}_1) \ge C_* t^{1-\delta/2}
\qquad\text{for all sufficiently small }t>0.
\]
\end{lemma}

\begin{proof}
Since $\bm{B}$ and $\bm{S}$ are independent, we have $\varphi_{\bm{X}^t_1}(\bm{u})=\varphi_{\bm{\Sigma}_{\bm{X}}\bm{B}_1}(\bm{u})\varphi_{\bm{S}_t}(\bm{u}/\sqrt{t})$. Hence, Lemma~\ref{lem:ddim_lowbound_wass} gives
\begin{equation}
\label{eq:lower_Gauss_lim_1}
\mW_1(\bm{X}^t_1,\bm{\Sigma}_{\bm{X}}\bm{B}_1) 
\ge T_1(\bm{X}^t_1,\bm{\Sigma}_{\bm{X}}\bm{B}_1)
\ge\frac{1}{|\bm{u}|}|\varphi_{\bm{\Sigma}_{\bm{X}}\bm{B}_1}(\bm{u})|\cdot
    |\varphi_{\bm{S}_t}(\bm{u}/\sqrt{t})-1|,
\quad\text{for any }
\bm{u}\in \R^d_{\bm{0}}.
\end{equation} 

(a) Let $\bm{u}_*\in \R^d$ be as in the statement. The result then follows from Lemma~\ref{lem:lower_bound_alpha_stab}.

(b) The proof is similar to that of Lemma~\ref{lem:lower_bound_alpha_stab}. The main idea is that, if $a_t\to 0$ as $t \da 0$ and $\sup_{t>0}|b_t|<\infty$, then 
\[
|e^{a_t+b_t}-e^{b_t}|
=  |e^{a_t}-1|\cdot |e^{b_t}|
\ge |a_t|\cdot\inf_{s>0}|e^{b_s}| + \Oh(|a_t|^2),
\quad\text{as }t\da 0.
\]
Set $b_t=-(t/2)\bm{u}_{t^{-1/2}}^\tra \bm{\Sigma}^2_{\bm{X}}\bm{u}_{t^{-1/2}}$ and $a_t=t\psi_{\bm{S}}(\bm{u}_{t^{-1/2}})$ for $t>0$. Note that $b_t=\Oh(1)$ as $t\da 0$ and $|e^{b_t}|\ge e^{-\lambda/2}$. Since $\bm{S}$ does not have a Brownian component, we have  $\lim_{|\bm{u}|\to\infty}|\bm{u}|^{-2}\cdot|\psi_{\bm{S}}(\bm{u})|=0$ (see, e.g.~\cite[Lem.~43.11]{MR3185174}) and thus $a_t\to0$ as $t\da0$. Moreover, $|a_t|\ge ct^{1-\delta/2}$ for $t<1$ by assumption. Thus, applying~\eqref{eq:lower_Gauss_lim_1} with $\bm{u}=\sqrt{t}\bm{u}_{t^{-1/2}}$ yields 
Part~(b).
\end{proof}

\begin{example}
\label{ex:arbitrarily_slow_BM_conv}
Consider an example inspired by~\cite[Ex.~4.2.1]{MR3784492}. Let $S$ be a real-valued martingale L\'evy process 
with L\'evy measure $\nu(\D y)=py^{-3}\log^{-1-p}(y)\1_{(0,1)}(y)\D y$ for some $p>0$ and all $y\in\R$. 
Let $B=(B_s)_{s \in [0,1]}$ be a standard Brownian motion independent of $S$ and let $\sigma^2>0$. Define $X^t\coloneqq (S_{st}/\sqrt{t}+\sigma B_{st}/\sqrt{t})_{s \in [0,1]}$ for $t>0$. Choose $g$ to satisfy $g(t)^2\log(1/g(t))^p \sim t$ as $t \da 0$. The function $g$ is regularly varying at $0$ with index $1/2$, and therefore $g(t) \sim \sqrt{t/\log(1/t)^p}$ as $t \da 0$ by~\cite[Thm~1.5.12]{MR1015093}. 
Note that $S_t/g(t)$ converges in law to a standard normal distribution by~\cite[Thm~2(i)]{MR3784492}. Hence, the L\'evy--Khintchine exponent $\psi_S$ of $S$ satisfies $t\psi_S(u/g(t))\to -u^2/2$ as $t\da 0$ for any $u\in\R$. In particular, by taking $t=g^{-1}(\sqrt{s})\sim s\log(1/s)^p$, which tends to $0$ as $s \da 0$, we obtain $s\log(1/s)^p\psi_S(1/\sqrt{s})\to -1/2$ as $s\da 0$. Then, a slight modification of the argument in the proof of Lemma~\ref{lem:lower_bound_gaussian_lim}(b) shows that 
$\liminf_{t\da 0}\log(1/t)^p\mW_1(X^t_1,\sigma B_1)>0$.
\end{example}

\section{Proofs of Section~\ref{sec:main_results}}
\label{sec:proofs}

In this section we give the proofs of the results stated in Section~\ref{sec:main_results}.

\begin{proof}[Proof of Theorem~\ref{thm:upper_lower_simple}]
\underline{Part (a).}
Recall from Assumption~(\nameref{asm:T}) and Remark~\ref{rem:construct_s_R_thin}, that we may decompose $\bm{X}$ as the sum $\bm{S}+\bm{R}$ of independent L\'evy processes $\bm{S}$ and $\bm{R}$ with generating triplets $(\bm{\gamma_S},\bm{0},\nu_{\bm{X}}^{\co})$ and $(\bm{\gamma_R},\bm{0},\nu_{\bm{X}}^{\mathrm{d}})$, respectively. For $t\in(0,1]$, denote $\bm{S}^t=(\bm{S}_{st}/t^{1/\alpha})_{s\in[0,1]}$ and $\bm{R}^t=(\bm{R}_{st}/t^{1/\alpha})_{s\in[0,1]}$. Let $\kappa(t)\coloneqq t^r$ for $t\in(0,1]$ and some $r\ge -1/\alpha$. Assume the processes $(\bm{D}^{\bm{S}^t,\kappa(t)},\bm{D}^{\bm{Z},\kappa(t)},\bm{J}^{\bm{S}^t,\kappa(t)},\bm{J}^{\bm{Z},\kappa(t)})$ are coupled as in Section~\ref{sec:thinning} (i.e. \eqref{eq:comp_Poisson_measures} and~\eqref{eq:thining_coupling_defn}). 

Note that $\mW_q(\bm{X}^t,\bm{Z}) \le \mW_q(\bm{S}^t,\bm{Z})+\E[\sup_{t \in [0,1]}|\bm{R}_{s}^t|^q]$ by the triangle inequality and the definition of $\mW_q$.
Theorem~\ref{thm:d_thin_dom_attract} (with $p=1$ and $q\in (0,1]\cap(0,\alpha)$) yields a 
bound on $\mW_q(\bm{X}^t,\bm{Z})$ via~\eqref{eq:Wq-XY} as follows:
\begin{gather*}
\E\bigg[\sup_{t \in [0,1]}\big|\bm{R}_{st}/t^{1/\alpha}\big|^q\bigg]
=\Oh\big(t^{1-q/\alpha}\big),\\
\mW_q\big(\bm{J}^{\bm{S}^t,\kappa(t)},\bm{J}^{\bm{Z},\kappa(t)}\big)
=\begin{dcases}
\Oh\big(t^{1/\alpha+r(q+1-\alpha)}\big),
& q <\alpha-1,\\
\Oh\big(t^{1-q/\alpha}(1+\log(1/t)\1_{\{q=\alpha-1,\,r\ne-1/\alpha\}})\big),
&q\ge\alpha-1,
\end{dcases}
\quad\text{by~\eqref{eq:Up_bound_thin_J},}\\
\mW_q\big(\bm{D}^{\bm{S}^t,\kappa(t)},\bm{D}^{\bm{Z},\kappa(t)}\big)\le \mW_2\big(\bm{D}^{\bm{S}^t,\kappa(t)},\bm{D}^{\bm{Z},\kappa(t)}\big)^q
=\Oh\big(t^{q/\alpha + rq(3-\alpha)}\big),
\quad\text{by~\eqref{eq:Wasserstein_relationship} and~\eqref{eq:Up_bound_thin_D},}\\
|\bm{\gamma}_{\bm{S}^t,\kappa(t)}-\bm{\gamma}_{\bm{Z},\kappa(t)}|^q 
=\begin{cases}
    \Oh\big(t^{q/\alpha + rq(2-\alpha)}\big), & \alpha\in(0,1),\\
    \Oh\big(t^{q(1-1/\alpha)}\big),& \alpha\in(1,2),
    \end{cases}
\quad\text{by~\eqref{eq:Up_bound_thin_gamma}.}
\end{gather*}
Part (a) can now be deduced by optimising $r$ as follows. If $\alpha<1$ then $q>0>\alpha-1$ and taking $r$ sufficiently large makes all terms become $\Oh(t^{1-q/\alpha})$. If $\alpha>1$ and $q<\alpha-1$, then the bounds are $\Oh(t^{1-q/\alpha})$, $\Oh(t^{1/\alpha+r(q+1-\alpha)})$, $\Oh(t^{q/\alpha+rq(3-\alpha)})$ and $\Oh(t^{q(1-1/\alpha)})$. Since $q\le 1$ and $1-1/\alpha\le 1/\alpha$, all these bounds can be made $\Oh(t^{q(1-1/\alpha)})$ by picking $r=0$. If $\alpha>1$ and $q\ge \alpha-1$, then the bounds are $\Oh(t^{1-q/\alpha})$, $\Oh(t^{1-q/\alpha}(1+\log(1/t)\1_{\{q=\alpha-1,\,r\ne-1/\alpha\}}))$, $\Oh(t^{q/\alpha+rq(3-\alpha)})$ and $\Oh(t^{q(1-1/\alpha)})$ and, as before, these can all be made $\Oh(t^{q(1-1/\alpha)})$ by picking $r=0$. Note here that the logarithmic term never arises in the dominant term, as it would require $1=q=\alpha-1$, but $\alpha<2$.

\underline{Part (b).} The claim follows from a direct application of Lemma~\ref{lem:lower_bound_alpha_stab}. 
\end{proof}

\begin{proof}[Proof of Theorem~\ref{thm:conv_BM_limit}]
Proposition~\ref{prop:conv_BM_limit} gives Part~(a). Parts~(b) and (c) follow from Lemma~\ref{lem:lower_bound_gaussian_lim}.
\end{proof}

In preparation for the proofs of Theorem~\ref{thm:upper_lower_general} and Corollary~\ref{cor:upper_lower_general}, we establish Lemmas~\ref{lem:1-a_SV},~\ref{lem:1-a_subpoly} and~\ref{lem:iterated-log} about slowly varying functions.

\begin{lemma}
\label{lem:1-a_SV}
Let $\ell\in\SV_\infty$ be $C^1$ with derivative $t\mapsto\wt\ell(t)/t$, where $|\wt\ell|\in\SV_\infty$. Then, for each $x>0$, we have $(\ell(t)-\ell(xt))/\wt\ell(t)\to-\log x$ as $t\to\infty$. Defining $L_x(t) \coloneqq 1-\ell(x t)/\ell(t)$ for $t,x>0$, the function $|L_x|$  is  asymptotically equivalent to the quotient $|\wt\ell(t)\log x|/\ell(t)\sim |L_x(t)|$ as $t\to\infty$ and, if $x\neq 1$, slowly varying at infinity. Moreover,  for $x>0$, the function $\Sigma(x)\coloneqq \sup_{t>0, y\in [x\wedge 1,x\vee1]}\wt\ell(yt)/\wt\ell(t)$  satisfies
\[
\big|L_x(t)\big|
=\bigg|\frac{\ell(xt)}{\ell(t)}-1\bigg|
\le \frac{|\wt\ell(t)|}{\ell(t)}
    \cdot\Sigma(x)|\log x|
\qquad \text{ for all }t,x>0.
\]
\end{lemma}

\begin{proof}
Since $|\wt\ell|$ is positive and $\wt\ell$ is continuous, $\wt\ell$ is either eventually positive or negative and either $\wt\ell$ or $-\wt\ell$ is slowly varying at infinity, respectively. By~\cite[Thm~1.2.1]{MR1015093}, $\sup_{y\in[a,b]}|\wt\ell(yt)/\wt\ell(t)-1|\to 0$ as $t\to\infty$ for any $0<a<b<\infty$. Thus, for all sufficiently large $t>0$, we have 
\[
\frac{\wt\ell(ty)}{\wt\ell(t)}\cdot\frac{1}{y}\le 
\bigg(1+\sup_{z\in [x\wedge 1,x\vee1]}\big|\wt\ell(zt)/\wt\ell(t)-1\big|\bigg) \frac{1}{y}\le 
\frac{2}{y}\qquad\text{for all $y\in [x\wedge 1,x\vee1]$.}
\]
The dominated convergence theorem now yields
\[
\frac{\ell(t)-\ell(xt)}{\wt\ell(t)}
=\frac{1}{\wt\ell(t)}\int_{ xt}^{t}\wt\ell(y)\frac{\D y}{y}
=\int_{x}^{1}\frac{\wt\ell(ty)}{\wt\ell(t)}\frac{\D y}{y}
\to \int_{x}^{1}\frac{\D y}{y}
=-\log x,\quad \text{as }t\to\infty,
\]
which establishes the first claim. 
Since $L_x(t)=\big((\ell(t)-\ell(xt))/\wt\ell(t)\big)\big(\wt\ell(t)/\ell(t)\big)$, the function $|L_x|$ is positive on a neighbourhood of infinity and asymptotically equivalent to  $|\wt\ell(t)\log x|/\ell(t)$  by the limit in the previous display.
Moreover, since $x>0$ in the limit was arbitrary,  for any $\lambda>0$ we have
\[
\frac{L_x(\lambda t)}{L_x(t)}
=\frac{\ell(t)}{\ell(\lambda t)}\cdot
    \frac{\wt\ell(\lambda t)^{-1}(\ell(\lambda t)-\ell(\lambda tx))}{\wt\ell(t)^{-1}(\ell(t)-\ell(x t))}\cdot \frac{\wt \ell(\lambda t)}{\wt \ell(t)}
\to 1,
\quad\text{as }t\to\infty,
\]
implying that $|L_x|$ is slowly varying at infinity. 

To establish the non-asymptotic inequality in the lemma, fix $x>0$ and note that
\[
|\ell(xt)-\ell(t)|
\le\bigg|\int_t^{xt}\wt\ell(y)\frac{\D y}{y}\bigg|
\le\int_{[x\wedge 1,x\vee1]}\big|\wt\ell(yt)\big|\frac{\D y}{y}
\le |\wt\ell(t)|\cdot \Sigma(x)|\log x|,
\]
giving the claim.
\end{proof}

Note that the assumption in Lemma~\ref{lem:1-a_SV} requires $\ell$ to be eventually strictly monotone. Moreover, if $\ell$ satisfies the conditions of Lemma~\ref{lem:1-a_SV}, then so does $\ell^q$ for any $q>0$. 

\begin{lemma}
\label{lem:1-a_subpoly} 
{\normalfont(a)} Let $\ell$ be slowly varying at infinity. Suppose that, for some $\lambda\in(0,\infty)\setminus\{1\}$ and non-increasing function $\phi_\lambda$, we have $\phi_\lambda(t)\ge|1-\ell(\lambda t)/\ell(t)|$ for all $t\ge 1$ and $\int_1^\infty \phi_\lambda(t)t^{-1}\D t<\infty$. Then $\ell$ has a positive finite limit at infinity.\\
{\normalfont(b)} Let $\phi$ be slowly varying at infinity with $\phi(t)\to 0$ as $t\to\infty$ and $\int_1^\infty\phi(t)t^{-1}\D t=\infty$. Then the functions $\ell_\pm(t)\coloneqq\exp(\pm\int_1^t\phi(s)s^{-1}\D s)$ are slowly varying at infinity, $\ell_+(t)\to\infty$, $\ell_-(t)\to 0$ and $|1-\ell_\pm(\lambda t)/\ell_\pm(t)|\sim|\log \lambda|\phi(t)$ as $t\to\infty$ for any $\lambda>0$.
\end{lemma}

The smallest non-increasing function $\varphi_\lambda$ with $\phi_\lambda(t)\ge|1-\ell(\lambda t)/\ell(t)|$ for all $t\ge 1$ is $t\mapsto \sup_{s\ge t}|1-\ell(\lambda s)/\ell(s)|$.

\begin{proof}[Proof of Lemma~\ref{lem:1-a_subpoly}]
Part (a). First assume $\lambda>1$. Define $U_\lambda(t)\coloneqq \sup_{x\in[1,\lambda]}|1-\ell(xt)/\ell(t)|$ and note that $U_\lambda(t)\to 0$ as $t\to\infty$ by the uniform convergence theorem~\cite[Thm~1.2.1]{MR1015093}.
As $\int_1^\infty \phi_\lambda(t)t^{-1}\D t<\infty$, we also have $\phi_\lambda(t)\to0$ as $t\to0$, making 
$\eta\coloneqq\inf\{t\ge 1:\max\{\phi_\lambda(s),U_\lambda(s)\}<1/2\text{ for all }s\ge t\}$ finite. 
Since 
$1+U_\lambda(t)\ge \ell(xt)/\ell(t)\ge 1-U_\lambda(t)$  and $U_\lambda(t)\le 1/2$ for all $t>\eta$ and $x\in[1,\lambda]$, we obtain
\[
U_\lambda(t)
\ge \log(1+U_\lambda(t))
\ge \log\bigg(\frac{\ell(x t)}{\ell(t)}\bigg)
\ge \log(1-U_\lambda(t))
\ge -2U_\lambda(t),\text{ implying $\bigg|\log\bigg(\frac{\ell(x t)}{\ell(t)}\bigg) \bigg|\le 2U_\lambda(t)$ }
\]
for all $t>\eta$ and $x\in[1,\lambda]$.
Similarly, for $t>\eta$ we have $|\log(\ell(\lambda t)/\ell(t))|\le 2 \phi_\lambda(t)$. For any $T>t\ge \eta$ set $n\coloneqq\lfloor \log(T/t)/\log\lambda\rfloor$, implying $T/(\lambda^n t)\in[1,\lambda)$.  By the monotonicity of $\phi_\lambda$ we obtain
\begin{align}
\nonumber
|\log\ell(T)-\log\ell(t)|
&\le\bigg|\log\bigg(\frac{\ell(T)}{\ell(\lambda^n t)}\bigg)\bigg|
+\sum_{k=1}^n\bigg|\log\bigg(\frac{\ell(\lambda^k t)}{\ell(\lambda^{k-1}t)}\bigg)\bigg|\\
\nonumber& 
\le 2U_\lambda(\lambda^n t)
+\sum_{k=1}^{n}2\phi_\lambda\big(\lambda^{k-1}t\big)
\le 2U_\lambda(\lambda^n t)
+\sum_{k=1}^{n}\frac{2}{\log\lambda}\int_{\lambda^{k-2}t}^{\lambda^{k-1}t}\phi_\lambda\big(s\big)\frac{\D s}{s}\\
\label{eq:uniform_lim_ell}
&\le 2U_\lambda(\lambda^n t)
+\frac{2}{\log\lambda}\int_{\lambda^{-1}t}^\infty \phi_\lambda\big(s\big)\frac{\D s}{s}\xrightarrow[]{t\to\infty} 0
\qquad \text{(uniformly in $T\in[t,\infty)$).}
\end{align}

If we had $\limsup_{t\to\infty}\log\ell(t)>\liminf_{t\to\infty}\log\ell(t)$, there would exist an increasing sequence $(t_k)_{k\in\N}$ and $\epsilon>0$ such that $t_k\to\infty$ and 
$|\log\ell(t_{k+1})-\log\ell(t_k)|>\epsilon$ for all $k\in\N$, contradicting~\eqref{eq:uniform_lim_ell}. Hence the limit $\lim_{t\to\infty}\log \ell(t)$ exists. By taking the limit as $T\to\infty$ on the left-hand side of~\eqref{eq:uniform_lim_ell} for any fixed $t$, it follows that $\lim_{t\to\infty}|\log \ell(t)|\neq \infty$. Thus $\ell$ has a finite and positive limit.
The case $\lambda<1$ can be established in a similar way.

Part (b). The statement follows from a direct application of Lemma~\ref{lem:1-a_SV}.
\end{proof}

\begin{lemma}
\label{lem:iterated-log}
Define iteratively the functions $\ell_1(t)=\log(e+t)$ and $\ell_{n+1}(t)=\log(e+\ell_n(t))$ for $t\ge 0$ and $n\in\N$. Then, the following statements hold.\\
{\normalfont(a)} For any $x>0$ and $c\ge 0$, we have $(c+\ell_n(xt))/(c+\ell_n(t)) \le 1+\1_{\{x>1\}}\log x$.\\
{\normalfont(b)} We have $(e+t)\ell_n'(t)=\prod_{i=1}^{n-1}(e+\ell_k(t))^{-1}$.\\
{\normalfont(c)} Suppose $\ell(t)=\ell_n(t)^{q_n}\cdots\ell_m(t)^{q_m}$ for some $1\le n\le m$ in $\N$ and either $q_n,\ldots,q_m\ge 0$ with $q_n,q_m>0$ or $q_n,\ldots,q_m\le 0$ with $q_n,q_m<0$. Set $\wt\ell(t)\coloneqq (e+t)\ell'(t)$, then we have
\[
\Sigma(x)
\coloneqq
\sup_{t>0,\,y\in [x\wedge 1,x\vee1]}\frac{\wt\ell(yt)}{\wt\ell(t)}
\le\begin{cases}
(1+\log x)^{\sum_{j=n}^m q_j^+},
& x\ge 1,\\
(1+|\log x|)^{m+\sum_{j=n}^mq_j^-},
& x<1.
\end{cases}
\]
\end{lemma}

\begin{proof}
For $x<1$ we have $\ell_n(xt)\le \ell_n(t)$. For $x>1$, we have
\[
\ell_{n}(xt)-\ell_{n}(t)
=\int_t^{xt} \frac{1}{\prod_{k=1}^{n-1}(e+\ell_k(s))}\frac{\D s}{e+s}
\le \frac{1}{\prod_{k=1}^{n-1}(e+\ell_k(t))}\int_t^{xt} \frac{\D s}{s}
=\frac{\log x}{\prod_{k=1}^{n-1}(e+\ell_k(t))}.
\]
In particular, we may add $c\ge 0$ to $\ell_{n}(xt)$ and $\ell_{n}(t)$ and divide by $c+\ell_{n}(t)$ to obtain
\[
\frac{c+\ell_{n}(xt)}{c+\ell_{n}(t)}
\le 1+\frac{\log x}{(c+\ell_n(t))\prod_{k=1}^{n-1}(e+\ell_k(t))}
\le 1+\log x,
\quad n\in\N,\, x>1,\,t>0,
\]
implying Part (a). Part (b) is obvious, so we need only establish Part (c).

It is simple to show that for $a_1,a_2,b_1,b_2>0$, the fraction $(a_1+a_2)/(b_1+b_2)$ lies between $a_1/b_1$ and $a_2/b_2$. An inductive argument implies that for any $a_1,\ldots,a_k,b_1,\ldots,b_k>0$, we have 
\[
\min_{j\in\{1,\ldots,k\}}\frac{a_j}{b_j}
\le\frac{\sum_{j=1}^k a_j}{\sum_{j=1}^k b_j}
\le\max_{j\in\{1,\ldots,k\}}\frac{a_j}{b_j}.
\]
Thus, by virtue of Parts (a) and (b) and denoting $\wt\ell_j(t)\coloneqq (e+t)\ell_j'(t)$, we have
\begin{align*}
\Sigma(x)
&=\sup_{t>0,\,y\in [x\wedge 1,x\vee1]}\frac{\ell(yt)\sum_{j=n}^m |q_j|\wt\ell_{j}(yt)/\ell_{j}(yt)}{\ell(t)\sum_{j=n}^m |q_j|\wt\ell_{j}(t)/\ell_{j}(t)}
\le\sup_{t>0,\,y\in [x\wedge 1,x\vee1]}
\max_{j\in\{n,\ldots,m\},\,q_j\ne 0}
\frac{\ell(yt)\wt\ell_{j}(yt)/\ell_{j}(yt)}{\ell(t)\wt\ell_{j}(t)/\ell_{j}(t)}\\
&=\sup_{t>0,\,y\in [x\wedge 1,x\vee1]}\max_{j\in\{n,\ldots,m\},\,q_j\ne 0}
\prod_{i=1}^{j-1}\frac{e+\ell_i(t)}{e+\ell_i(yt)}
\prod_{i=n}^m
\bigg(\frac{\ell_i(yt)}{\ell_i(t)}\bigg)^{q_i-\1_{\{i=j\}}}\\
& \le\begin{cases}
(1+\log x)^{\sum_{j=n}^m q_j^+},
&\!\! x\ge 1,\\
(1+|\log x|)^{m+\sum_{j=n}^mq_j^-},
&\!\! x<1.
\end{cases}
\end{align*}
This completes the proof.
\end{proof}

\begin{proof}[Proof of Theorem~\ref{thm:upper_lower_general}]
\underline{Part (a).} Recall from Remark~\ref{rem:comono_coup_S_R}, that we can decompose $\bm{X}$ as the sum $\bm{S}+\bm{R}$ of the independent processes $\bm{S}$ and $\bm{R}$ with generating triplets $(\bm{\gamma_S},\bm{0},\nu_{\bm{X}}^{\co})$ and $(\bm{\gamma_R},\bm{0},\nu_{\bm{X}}^{\mathrm{d}})$, respectively. For $t \in [0,1]$, let $\bm{S}^t=(\bm{S}_{st}/g(t))_{s\in[0,1]}$ and $\bm{R}^t=(\bm{R}_{st}/g(t))_{s\in[0,1]}$. Assume that $(\bm{M}^{\bm{S}^t},\bm{M^Z},\bm{L}^{\bm{S}^t},\bm{L^Z})$ is coupled as in~\eqref{eq:comono_coupling_1} and~\eqref{eq:comono_coupling_2}. 

Note that $\mW_q(\bm{X}^t,\bm{Z}) \le \mW_q(\bm{S}^t,\bm{Z})+\E[\sup_{t \in [0,1]}|\bm{R}_{s}^t|^q]$ by the triangle inequality. Next, we apply~\eqref{eq:Wq-XY} with $p=1$ to $\mW_q(\bm{S}^t,\bm{Z})$ and use Theorem~\ref{thm:upper_bound_comono_tech} and Potter's bounds~\cite[Thm~1.5.6]{MR1015093} (applied to the slowly varying function $G_2$) to show that each resulting term is $\Oh(G_2(t)^q)$:
\begin{itemize}
    \item $\mW_q\big(\bm{M}^{\bm{S}^t},\bm{M^Z}\big) \le \mW_2\big(\bm{M}^{\bm{S}^t},\bm{M^Z}\big)^q $ by~\eqref{eq:Wasserstein_relationship}, and~\eqref{eq:Up_bound_como_M} then yields the bound; 
    \item $\mW_q\big(\bm{L}^{\bm{S}^t}, \bm{L}^{\bm{Z}}\big) $ is bounded by~\eqref{eq:Up_bound_como_L};
    \item $|\bm{\varpi}_{\bm{S}^t}-\bm{\varpi}_{\bm{Z}}|^q$ is bounded by~\eqref{eq:Up_bound_como_varpi}. 
\end{itemize}
Similarly, by Potter's bounds and Theorem~\ref{thm:upper_bound_comono_tech}, $\E[\sup_{t \in [0,1]}|\bm{R}_{s}^t|^q]=\Oh(t^{1-q/\alpha}G(t)^{-q})=\Oh(G_2(t)^{q})$.

\noindent \underline{Part (b).} Recall $a(t)= G(2t)/G(t)\to 1$ as $t\to 0$. Directly from Proposition~\ref{prop:W_lower}, we see that 
$$
3\max\{\mW_q(\bm{X}^t_1,\bm{Z}_1)
,\mW_q(\bm{X}^{2t}_1, \bm{Z}_1)\} 
\ge 
|1-a(t)^q|\E[|\bm{Z}_1|^q] 
\quad\text{for all sufficiently small $t >0$,}
$$ 
yielding the first claim of part (b). 
The second claim follows from Lemma~\ref{lem:1-a_subpoly} since $G$ is assumed not to have a positive finite limit.
\end{proof}

\begin{proof}[Proof of Corollary~\ref{cor:upper_lower_general}]
Part (b) follows from Lemma~\ref{lem:iterated-log} above. Given Theorem~\ref{thm:upper_lower_general}, it suffices to show that the assumptions in Corollary~\ref{cor:upper_lower_general}(a) imply those of Theorem~\ref{thm:upper_lower_general} and that the upper and lower bounds have the desired form. These facts follow from Lemmas~\ref{lem:1-a_SV} \&~\ref{lem:iterated-log}. Indeed, for instance, the function $G_1$ in Assumption~(\nameref{asm:S}) is given by the upper bound on $x\mapsto\Sigma(x)|\log x|$ given in Lemma~\ref{lem:iterated-log}, where $\Sigma$ is as in Lemma~\ref{lem:1-a_SV}. 
\end{proof}

\section{Concluding remarks}
\label{sec:conclusion}

Over small time horizons, a L\'evy process may be attracted to an $\alpha$-stable process with heavy tails (i.e. $\alpha\in(0,2)$) or Brownian motion (i.e. $\alpha=2$). In this paper, we established upper and lower bounds on the rate of convergence in $L^q$-Wasserstein distance in both regimes, as listed below.
\begin{itemize}[leftmargin=2em, nosep]
\item For $\alpha\in(0,2]\setminus\{1\}$ and processes in the domain of non-normal attraction, the Wasserstein distance is bounded above and below by slowly varying functions (see Theorem~\ref{thm:upper_lower_general}), both of which are slower than any power of logarithm greater than $1$. 
\item For $\alpha\in(0,2)\setminus\{1\}$ and processes in the domain of normal attraction, we establish upper and lower bounds that are polynomial in $t$ (see Theorem~\ref{thm:upper_lower_simple}). The established bounds are often rate-optimal in $L^q$-Wasserstein distance for $q<\alpha<1$ or $q=1<\alpha$ and proportional to $t^{1-q/\alpha}$. 
\item For $\alpha=2$ (i.e. Brownian limit) and processes in the domain of normal attraction, the established upper and lower bounds are also polynomial in $t$ (see Theorem~\ref{thm:conv_BM_limit}). In this case, the bounds are rate-optimal when the Blumenthal--Getoor index $\beta\le 1$ and otherwise there is a polynomial gap. This suggests at least one of the bounds is not sharp. Establishing sharper bounds in this special case is non-trivial (as classical tools such as the Berry--Esseen theorem fail to provide converging bounds) and is therefore left for future work.
\end{itemize}

The process $\bm{R}$ in Assumption~(\nameref{asm:T}) (resp.~(\nameref{asm:C})) in the thinning (resp. comonotonic) coupling is assumed to have finitely many jumps on compact time intervals. Our results can be extended to the case where this process has infinitely many jumps on compact time intervals and a Blumenthal--Getoor index $\beta<\alpha$. For such an extension, the moments of $\sup_{s\in[0,1]}|\bm{R}_s^t|$, as a function of $t$, can be controlled via Lemma~\ref{lem:moment_bound} and would result in worse convergence rates as $\beta\uparrow\alpha$. We chose not to include this simple extension as the convergence rates would be much harder to express in terms of all the model parameters, resulting in a less concise presentation of our results.

The tools developed in Section~\ref{sec:couplings_general} could be used for the omitted case $\alpha=1$ to establish upper and lower bounds on the Wasserstein distance in the domains of normal and non-normal attraction. However, as multiple cases would arise, requiring careful treatment of the emerging slowly varying functions, we leave such extension for future work. 

The upper bounds in the heavy-tailed case $\alpha\in(0,2)\setminus\{1\}$ are based on two distinct couplings introduced in Section~\ref{sec:couplings_general}: the comonotonic and thinning couplings. We mention briefly that it is possible to combine both couplings. Consider two L\'evy processes $\bm{X}$ and $\bm{Y}$ with L\'evy measures $f_{\bm{X}}\D\mu$ and $f_{\bm{Y}}\D\mu$ where $0\le f_{\bm{X}}\le 1$ and $0\le f_{\bm{Y}}\le 1$ are measurable and $\mu$ is a L\'evy measure. It is then possible to first apply the thinning coupling to synchronise the jumps arising from the L\'evy measure $g\D\mu$, where $g\coloneqq \min\{f_{\bm{X}},f_{\bm{Y}}\}$, and then apply the comonotonic coupling to the remaining jumps of $\bm{X}$ and $\bm{Y}$ with corresponding L\'evy measures $(f_{\bm{X}}-g)\D\mu$ and $(f_{\bm{Y}}-g)\D\mu$. It appears, however, that this combined coupling will not yield improved rates of convergence to heavy-tailed stable limits in all cases and would instead result in a more flexible set of assumptions but with many more cases to consider in its analysis. 

When $\alpha\in(0,2)$, the $\alpha$-stable limit has heavy tails and its $q$-moment is finite if and only if $q<\alpha$, making it impossible to obtain general converging bounds in the $L^2$-Wasserstein distance, which play a key role in various applications including Multilevel Monte Carlo. However,
substituting the standard Euclidean metric on $\R^d$ with an equivalent \textit{bounded} metric would remove this obstruction.  
We expect our couplings to perform well and have a fast converging $L^2$-Wasserstein distance under the bounded metric on $\R^d$. Such an extension of our results is also left for future work. 

The present work focused on the small-time stable domain of attraction where the small jumps of the process dominate the activity.
Finally we remark that it is natural to expect that the comonotonic coupling developed in this paper could also typically achieve asymptotically optimal  convergence rate in the Wasserstein distance in the scaling limits of the long-time stable domain of attraction.
This is because, in the long-time horizon regime, the activity in the limit is dominated by the large jumps of the L\'evy process, which are  efficiently coupled under the comonotonic coupling of Section~\ref{sec:comonotonic_coupling} above.

\section*{Acknowledgements}
JGC and AM were supported by EPSRC grant EP/V009478/1 and The Alan Turing Institute under the EPSRC grant EP/N510129/1; JGC was also supported by DGAPA-PAPIIT grant 36-IA104425; AM was also supported by EPSRC grant EP/W006227/1; DKB was funded by the CDT in Mathematics and Statistics at The University of Warwick and is supported by AUFF NOVA grant AUFF-E-2022-9-39.
The authors would like to thank the Isaac Newton Institute for Mathematical Sciences, Cambridge, for support  during the INI satellite programme \textit{Heavy tails in machine learning}, hosted by The Alan Turing Institute, London, and the INI programme \textit{Stochastic systems for anomalous diffusion} hosted at INI in Cambridge, where work on this paper was undertaken. This work was supported by EPSRC grant EP/R014604/1.

\bibliographystyle{abbrv}
\bibliography{Referencer}

\begin{thebibliography}{10}

\bibitem{MR3803961}
S.~Asmussen and J.~Ivanovs.
\newblock Discretization error for a two-sided reflected {L}\'{e}vy process.
\newblock {\em Queueing Syst.}, 89(1-2):199--212, 2018.

\bibitem{MR1834755}
S.~Asmussen and J.~Rosi\'{n}ski.
\newblock Approximations of small jumps of {L}\'{e}vy processes with a view
  towards simulation.
\newblock {\em J. Appl. Probab.}, 38(2):482--493, 2001.

\bibitem{MR4170640}
G.~Auricchio, A.~Codegoni, S.~Gualandi, G.~Toscani, and M.~Veneroni.
\newblock The equivalence of {F}ourier-based and {W}asserstein metrics on
  imaging problems.
\newblock {\em Atti Accad. Naz. Lincei Rend. Lincei Mat. Appl.},
  31(3):627--649, 2020.

\bibitem{BANG2021109187}
D.~Bang, J.~Gonz\'{a}lez~C\'{a}zares, and A.~Mijatovi\'{c}.
\newblock A {G}aussian approximation theorem for {L}\'{e}vy processes.
\newblock {\em Statist. Probab. Lett.}, 178:Paper No. 109187, 4, 2021.

\bibitem{MR1406564}
J.~Bertoin.
\newblock {\em {L}\'{e}vy processes}, volume 121 of {\em Cambridge Tracts in
  Mathematics}.
\newblock Cambridge University Press, Cambridge, 1996.

\bibitem{MR1015093}
N.~H. Bingham, C.~M. Goldie, and J.~L. Teugels.
\newblock {\em Regular variation}, volume~27 of {\em Encyclopedia of
  Mathematics and its Applications}.
\newblock Cambridge University Press, Cambridge, 1989.

\bibitem{MR4161123}
K.~Bisewski and J.~Ivanovs.
\newblock Zooming-in on a {L}\'{e}vy process: failure to observe threshold
  exceedance over a dense grid.
\newblock {\em Electron. J. Probab.}, 25:Paper No. 113, 33, 2020.

\bibitem{MR3959085}
J.~Blanchet and K.~Murthy.
\newblock Quantifying distributional model risk via optimal transport.
\newblock {\em Math. Oper. Res.}, 44(2):565--600, 2019.

\bibitem{MR3806899}
C.~B\"{o}rgers and C.~Greengard.
\newblock Slow convergence in generalized central limit theorems.
\newblock {\em C. R. Math. Acad. Sci. Paris}, 356(6):679--685, 2018.

\bibitem{MR1482707}
M.~Broadie, P.~Glasserman, and S.~Kou.
\newblock A continuity correction for discrete barrier options.
\newblock {\em Math. Finance}, 7(4):325--349, 1997.

\bibitem{MR1805321}
M.~Broadie, P.~Glasserman, and S.~G. Kou.
\newblock Connecting discrete and continuous path-dependent options.
\newblock {\em Finance Stoch.}, 3(1):55--82, 1999.

\bibitem{MR2792485}
M.~E. Caballero, J.~C. Pardo, and J.~L. P\'{e}rez.
\newblock On {L}amperti stable processes.
\newblock {\em Probab. Math. Statist.}, 30(1):1--28, 2010.

\bibitem{MR2307403}
S.~Cohen and J.~Rosi\'{n}ski.
\newblock Gaussian approximation of multivariate {L}\'{e}vy processes with
  applications to simulation of tempered stable processes.
\newblock {\em Bernoulli}, 13(1):195--210, 2007.

\bibitem{MR2867949}
E.~H.~A. Dia and D.~Lamberton.
\newblock Connecting discrete and continuous lookback or hindsight options in
  exponential {L}\'{e}vy models.
\newblock {\em Adv. in Appl. Probab.}, 43(4):1136--1165, 2011.

\bibitem{MR2851060}
E.~H.~A. Dia and D.~Lamberton.
\newblock Continuity correction for barrier options in jump-diffusion models.
\newblock {\em SIAM J. Financial Math.}, 2(1):866--900, 2011.

\bibitem{Ester2024}
C.~Duval, T.~Jalal, and E.~Mariucci.
\newblock Nonparametric density estimation for the small jumps of {L}\'evy
  processes, 2024.
\newblock \href{https://arxiv.org/abs/2404.09725}{arXiv:2404.09725}.

\bibitem{MR1458613}
P.~Embrechts, C.~Kl\"{u}ppelberg, and T.~Mikosch.
\newblock {\em Modelling extremal events}, volume~33 of {\em Applications of
  Mathematics (New York)}.
\newblock Springer-Verlag, Berlin, 1997.
\newblock For insurance and finance.

\bibitem{MR4321329}
V.~Fomichov, J.~Gonz\'{a}lez~C\'{a}zares, and J.~Ivanovs.
\newblock Implementable coupling of {L}\'{e}vy process and {B}rownian motion.
\newblock {\em Stochastic Process. Appl.}, 142:407--431, 2021.

\bibitem{Gibbs02onchoosing}
A.~L. Gibbs and F.~E. Su.
\newblock On choosing and bounding probability metrics.
\newblock {\em INTERNAT. STATIST. REV.}, pages 419--435, 2002.

\bibitem{MR0233400}
B.~V. Gnedenko and A.~N. Kolmogorov.
\newblock {\em Limit distributions for sums of independent random variables}.
\newblock Addison-Wesley Publishing Co., Reading, Mass.-London-Don Mills.,
  Ont., revised edition, 1968.
\newblock Translated from the Russian, annotated, and revised by K. L. Chung,
  With appendices by J. L. Doob and P. L. Hsu.

\bibitem{YouTube_talk}
J.~Gonz\'alez~C\'azares, D.~Kramer-Bang, and A.~Mijatovi\'c.
\newblock Presentation on ``asymptotically optimal {W}asserstein couplings for
  the small-time stable domain of attraction''.
\newblock \href{https://youtu.be/76eJD6a8Kko?si=eqh4c_Dyy9ERP44a}{YouTube
  presentation} on the YouTube channel
  \href{https://www.youtube.com/@prob-am7844}{Prob-AM}, 10 Nov 2024.

\bibitem{2021simulation}
J.~Gonz\'{a}lez~C\'{a}zares and A.~Mijatovi\'{c}.
\newblock Simulation of the drawdown and its duration in {L}\'{e}vy models via
  stick-breaking {G}aussian approximation.
\newblock {\em Finance Stoch.}, 26(4):671--732, 2022.

\bibitem{doi:10.1287/moor.2021.1163}
J.~I. Gonz\'{a}lez~C\'{a}zares, A.~Mijatovi\'{c}, and G.~Uribe~Bravo.
\newblock Geometrically convergent simulation of the extrema of lévy
  processes.
\newblock {\em Mathematics of Operations Research}, 47(2):1141--1168, 2022.

\bibitem{MR1106283}
F.~G\"{o}tze.
\newblock On the rate of convergence in the multivariate {CLT}.
\newblock {\em Ann. Probab.}, 19(2):724--739, 1991.

\bibitem{MR629531}
P.~Hall.
\newblock Two-sided bounds on the rate of convergence to a stable law.
\newblock {\em Z. Wahrsch. Verw. Gebiete}, 57(3):349--364, 1981.

\bibitem{MR3784492}
J.~Ivanovs.
\newblock Zooming in on a {L}\'{e}vy process at its supremum.
\newblock {\em Ann. Appl. Probab.}, 28(2):912--940, 2018.

\bibitem{MR4244191}
J.~Ivanovs and J.~D. Th{\o}stesen.
\newblock Discretization of the {L}amperti representation of a positive
  self-similar {M}arkov process.
\newblock {\em Stochastic Process. Appl.}, 137:200--221, 2021.

\bibitem{MR2172843}
O.~Johnson and R.~Samworth.
\newblock Central limit theorem and convergence to stable laws in {M}allows
  distance.
\newblock {\em Bernoulli}, 11(5):829--845, 2005.

\bibitem{MR1876169}
O.~Kallenberg.
\newblock {\em Foundations of modern probability}.
\newblock Probability and its Applications (New York). Springer-Verlag, New
  York, second edition, 2002.

\bibitem{Optimal2022}
W.~S. Kendall, M.~B. Majka, and A.~Mijatovi\'{c}.
\newblock Optimal {M}arkovian coupling for finite activity {L}\'{e}vy
  processes.
\newblock {\em Bernoulli}, 30(4):2821--2845, 2024.

\bibitem{MR1207584}
J.~F.~C. Kingman.
\newblock {\em Poisson processes}, volume~3 of {\em Oxford Studies in
  Probability}.
\newblock The Clarendon Press, Oxford University Press, New York, 1993.
\newblock Oxford Science Publications.

\bibitem{MR2724421}
A.~Kuznetsov.
\newblock Wiener-{H}opf factorization and distribution of extrema for a family
  of {L}\'{e}vy processes.
\newblock {\em Ann. Appl. Probab.}, 20(5):1801--1830, 2010.

\bibitem{MR2977987}
A.~Kuznetsov, A.~E. Kyprianou, and J.~C. Pardo.
\newblock Meromorphic {L}\'{e}vy processes and their fluctuation identities.
\newblock {\em Ann. Appl. Probab.}, 22(3):1101--1135, 2012.

\bibitem{MR647969}
R.~LePage.
\newblock Multidimensional infinitely divisible variables and processes. {II}.
\newblock In {\em Probability in {B}anach spaces, {III} ({M}edford, {M}ass.,
  1980)}, volume 860 of {\em Lecture Notes in Math.}, pages 279--284. Springer,
  Berlin-New York, 1981.

\bibitem{manou-abi}
S.~Manou-Abi.
\newblock Rate of convergence to alpha stable law using {Z}olotarev distance.
\newblock {\em Journal of Statistics: Advances in Theory and Applications},
  18:166--177, 2017.

\bibitem{MR3833470}
E.~Mariucci and M.~Rei{\ss}.
\newblock Wasserstein and total variation distance between marginals of
  {L}\'{e}vy processes.
\newblock {\em Electron. J. Stat.}, 12(2):2482--2514, 2018.

\bibitem{MR3393269}
P.~Pegon, F.~Santambrogio, and D.~Piazzoli.
\newblock Full characterization of optimal transport plans for concave costs.
\newblock {\em Discrete Contin. Dyn. Syst.}, 35(12):6113--6132, 2015.

\bibitem{MR1619170}
S.~T. Rachev and L.~R\"{u}schendorf.
\newblock {\em Mass transportation problems. {V}ol. {I}}.
\newblock Probability and its Applications (New York). Springer-Verlag, New
  York, 1998.
\newblock Theory.

\bibitem{MR2548505}
E.~Rio.
\newblock Upper bounds for minimal distances in the central limit theorem.
\newblock {\em Ann. Inst. Henri Poincar\'{e} Probab. Stat.}, 45(3):802--817,
  2009.

\bibitem{MR1833707}
J.~Rosi\'{n}ski.
\newblock Series representations of {L}\'{e}vy processes from the perspective
  of point processes.
\newblock In {\em L\'{e}vy processes}, pages 401--415. Birkh\"{a}user Boston,
  Boston, MA, 2001.

\bibitem{MR2327834}
J.~Rosi\'{n}ski.
\newblock Tempering stable processes.
\newblock {\em Stochastic Process. Appl.}, 117(6):677--707, 2007.

\bibitem{MR3185174}
K.-i. Sato.
\newblock {\em {L}\'{e}vy processes and infinitely divisible distributions},
  volume~68 of {\em Cambridge Studies in Advanced Mathematics}.
\newblock Cambridge University Press, Cambridge, 2013.
\newblock Translated from the 1990 Japanese original, Revised edition of the
  1999 English translation.

\bibitem{villani2008optimal}
C.~Villani.
\newblock {\em Optimal Transport: Old and New}.
\newblock Grundlehren der mathematischen Wissenschaften. Springer Berlin
  Heidelberg, 2008.

\end{thebibliography}

\begin{appendix}

\section{Proof of Lemma~\ref{lem:moment_bound}}
\label{app:Proof_lemma_bound} 
Given any $\kappa\in(0,1]$ consider the L\'evy--It\^o decomposition $\bm{X}_t = \bm{\gamma}_{\bm{X}}^\ka t + \bm{\Sigma_X}\bm{B}_t + \bm{D}^\ka_t + \bm{J}^\ka_t$ given in~\eqref{eq:levy_ito_decomp}, where $\bm{B}$ is a standard Brownian motion, $\bm{D}^{(\kappa)}$ is the pure-jump martingale containing all the jumps of $\bm{X}$ of magnitude less than $\kappa$ and $\bm{J}^{(\kappa)}$ is the driftless compound Poisson process containing all the jumps of $\bm{X}$ of magnitude at least $\kappa$. Since 
$|\bm{\Sigma_X}\bm{B}_s|\le |\bm{\Sigma_X}|\cdot|\bm{B}_s|$, we have 
\begin{equation}\label{eq:supremum_bound_norm}
    \sup_{s\in[0,t]}|\bm{X}_s| 
\le |\bm{\gamma}_{\bm{X}}^\ka|t + |\bm{\Sigma_X}|\sup_{s\in[0,t]}|\bm{B}_s| + \sup_{s\in[0,t]}|\bm{D}^\ka_s| + \sup_{s\in[0,t]}|\bm{J}^\ka_s|, \quad \text{ for } t \in [0,1].
\end{equation}
By the elementary bound $(\sum_{i=1}^n x_i)^p\le n^{(p-1)^+}\sum_{i=1}^n x_i^p$, $p>0$, $(p-1)^+=\max\{p-1,0\}$ and $x_i\ge 0$, $i\in\{1,\ldots,n\}$, we only need to bound the $p$-th moment of each summand on the right-hand side of the display above. 
Recall that 
$\beta_+$ is the quantity associated to the BG index of $\bm{X}$ defined in~\eqref{eq:BG}.

\underline{Case $\beta_+>0$.} Define $\kappa\coloneqq t^{1/\beta_+}$. 
To bound the drift term $|\bm{\gamma}_{\bm{X}}^\ka|t$, first assume $\beta_+\ge 1$ and note that 
\begin{align*}
|\bm{\gamma}_{\bm{X}}^\ka|
=\bigg|\bm{\gamma}_{\bm{X}} - \int_{B_{\bm{0}}(1)\setminus B_{\bm{0}}(\kappa)}\bm{w}\nu_{\bm{X}}(\D\bm{w})\bigg|
&\le|\bm{\gamma}_{\bm{X}}| + \int_{B_{\bm{0}}(1)\setminus B_{\bm{0}}(\kappa)}|\bm{w}|\nu_{\bm{X}}(\D\bm{w})\\
&\le|\bm{\gamma}_{\bm{X}}| + \int_{B_{\bm{0}}(1)\setminus B_{\bm{0}}(\kappa)}\kappa^{1-\beta_+}|\bm{w}|^{\beta_+}\nu_{\bm{X}}(\D\bm{w})
\le|\bm{\gamma}_{\bm{X}}| + \kappa^{1-\beta_+}I_{\beta_+}.
\end{align*}
Thus, $(|\bm{\gamma_{\bm{X}}}^\ka|t)^p$ is bounded by a constant multiple of $t^p+t^{p/\beta_+}$. If $\beta_+\in(0,1)$ and the natural drift of $\bm{X}$ is zero (i.e. $\bm{\gamma_X}=\int_{B_{\bm{0}}(1)\setminus\{\bm{0}\}}\bm{w}\nu_{\bm{X}}(\D\bm{w})$),
then $|\bm{\gamma}_{\bm{X}}^\ka|$ is bounded (and convergent) as $t\da 0$,
\begin{equation*}
|\bm{\gamma}_{\bm{X}}^\ka|
=\bigg|\int_{B_{\bm{0}}(\kappa)\setminus\{\bm{0}\}}\bm{w}\nu_{\bm{X}}(\D\bm{w})\bigg|
\le\int_{B_{\bm{0}}(\kappa)\setminus\{\bm{0}\}}|\bm{w}|\nu_{\bm{X}}(\D\bm{w})
\le\int_{B_{\bm{0}}(\kappa)\setminus\{\bm{0}\}}\kappa^{1-\beta_+}|\bm{w}|^{\beta_+}\nu_{\bm{X}}(\D\bm{w})
\le\kappa^{1-\beta_+}I_{\beta_+},
\end{equation*}
making $(|\bm{\gamma}_{\bm{X}}^\ka|t)^p$ bounded by a multiple of $t^{p/\beta_+}$. Hence, in this case, we may take $C_2=0$ in Lemma~\ref{lem:moment_bound} (it will become clear from the remainder of the proof that none of the other summands on the right-hand side of the inequality  in~\eqref{eq:supremum_bound_norm} will produce a term of order $t^p$).

The $p$-th moment of the Brownian term is easily bounded by a constant multiple of $t^{p/2}$ since we have $\sup_{s\in[0,t]}|\bm{B}_s|\eqd t^{1/2}\sup_{s\in[0,1]}|\bm{B}_s|$. If $\bm{\Sigma_X}$ is a zero matrix, then $|\bm{\Sigma_X}|=0$ and hence $C_1=0$. 

Next, we bound the big-jump term $\bm{J}^\ka$. Let $A_\kappa:=\R^d\setminus B_{\bm{0}}(\kappa)$ and recall that $\bm{J}_t=\sum_{k=1}^{N_t}\bm{R}_k$ for some Poisson random variable $N_t$ with mean $t\nu_{\bm{X}}(A_\kappa)$ and iid random vectors $(\bm{R}_n)_{n\in\N}$ independent of $N_t$ with law $\nu_{\bm{X}}(\cdot\cap A_\kappa)/\nu_{\bm{X}}(A_\kappa)$. Recall, from the formula for the moments of a Poisson random variable, that $\E[N_t^k]=\sum_{j=1}^k\stirII{k}{j}(t\nu_{\bm{X}}(A_\kappa))^j$, where $\stirII{k}{j}$ denotes the Stirling number of the second kind.  
Note that the triangle inequality of the Euclidean norm $|\cdot|$ implies that 
$$
|\bm{J}^\ka_s|^p\le \left(\sum_{k=1}^{N_s}|\bm{R}_k|\right)^p
\le
\left(\sum_{k=1}^{N_t}|\bm{R}_k|\right)^p
\le
N_t^{(p-1)^+}\sum_{k=1}^{N_t}|\bm{R}_k|^p, \quad \text{ for every }s\in[0,t].
$$
Denote $\ceil{p}:=\inf\{n\in\N\,:\,n\ge p\}$, and note that since $\bm{R}_k$, $k\in\N$, are iid and independent of $N_t$ and $1\le (p-1)^++1\le \ceil{p}$, we find
\begin{equation}\label{eq:bound_large_jumps_moment_bound}
\begin{aligned}
    \E\bigg[\sup_{s\in[0,t]}|\bm{J}^\ka_s|^p\bigg]
&\le 
\E\big[|\bm{R}_1|^p\big]\E\Big[N_t^{\ceil{p}}\Big]
=\int_{A_\kappa}|\bm{w}|^p\frac{\nu_{\bm{X}}(\D\bm{w})}{\nu_{\bm{X}}(A_\kappa)}    \cdot\sum_{k=1}^{\ceil{p}}\stirII{\ceil{p}}{k}(t\nu_{\bm{X}}(A_\kappa))^k\\
& =
t\int_{A_\kappa}|\bm{w}|^p\nu_{\bm{X}}(\D\bm{w})    \cdot\sum_{k=1}^{\ceil{p}}\stirII{\ceil{p}}{k}(t\nu_{\bm{X}}(A_\kappa))^{k-1}.
\end{aligned}
\end{equation}
Note that $\nu_{\bm{X}}(A_\kappa)\le\nu_{\bm{X}}(A_1)+\int_{B_{\bm{0}}(1)\setminus B_{\bm{0}}(\kappa)}\kappa^{-\beta_+}|\bm{w}|^{\beta_+}\nu_{\bm{X}}(\D\bm{w})\le \nu_{\bm{X}}(A_1)+\kappa^{-\beta_+}I_{\beta_+}$ and hence $t\nu_{\bm{X}}(A_\kappa)$ is bounded in $t\in[0,1]$ (recall that $\kappa=t^{1/\beta_+}$), making the sum in the display also bounded. Denote $I_p'=\int_{A_1}|\bm{w}|^p\nu_{\bm{X}}(\D\bm{w})$, which we assumed finite, and hence 
\begin{align*}
\int_{A_\kappa}|\bm{w}|^p\nu_{\bm{X}}(\D\bm{w})
&\le I_p'+\int_{B_{\bm{0}}(1)\setminus B_{\bm{0}}(\kappa)}|\bm{w}|^p\nu_{\bm{X}}(\D\bm{w})\\
&\le I_p'+\int_{B_{\bm{0}}(1)\setminus B_{\bm{0}}(\kappa)}\kappa^{-(\beta_+-p)^+}|\bm{w}|^{\max\{\beta_+,p\}}\nu_{\bm{X}}(\D\bm{w})
\le I_p'+\kappa^{-(\beta_+-p)^+}I_{\max\{\beta_+,p\}}.
\end{align*}
Thus, there is a finite constant $C>0$ such that
\[
\E\bigg[\sup_{s\in[0,t]}|\bm{J}^\ka_s|^p\bigg]
\le Ct(I_p'+\kappa^{-(\beta_+-p)^+}I_{\max\{\beta_+,p\}})
= C(I_p't+I_{\max\{\beta_+,p\}}t^{\min\{1,p/\beta_+\}}).
\]

In the case where $\beta_+>0$, it remains to bound the small-jump term $\bm{D}^\ka$. In this case, we show that the $p$-th moment is bounded by a multiple of $t^{p/\beta_+}$. We may assume without loss of generality that $p>1$, since the other cases would follow by Jensen's inequality since $\E[|\bm{\xi}|^q]\le\E[|\bm{\xi}|^p]^{q/p}$ for any $q\le p$. Since $|\bm{D}^\ka_t|$ is a submartingale, Doob's maximal inequality and the elementary inequality $|x|^p\le (p/e)^p e^{|x|}$ imply 
\[
\E\bigg[\sup_{s\in[0,t]}|\bm{D}^\ka_s|^p\bigg]
\le \Big(\frac{p}{p-1}\Big)^p\E\big[|\bm{D}^\ka_t|^p\big]
= \Big(\frac{\kappa p}{p-1}\Big)^p\E\big[|\kappa ^{-1}\bm{D}^\ka_t|^p\big]
\le \Big(\frac{\kappa p^2/e}{p-1}\Big)^p\E\Big[e^{\kappa^{-1}|\bm{D}^\ka_t|}\Big],
\] for $t \in [0,1]$. Thus, to complete the proof it suffices to show that the expectation on the right is bounded as $t\da 0$. Let $\{\bm{e}_i\}_{i=1}^{2^d}$ be the vertices of the hypercube centered at the origin with sides parallel to the axes and side length $2$ (e.g., the vectors $(1,1,\ldots,1)$ and $(-1,-1,\ldots,-1)$ are opposite vertices of this hypercube). Note that $e^{|\bm{w}|}\le e^{|\bm{w}|_1}\le\sum_{i=1}^{2^d}e^{\langle\bm{e}_i,\bm{w}\rangle}$ where $|(s_1,\ldots,s_d)|_1\coloneqq\sum_{i=1}^d |s_i|$ denotes the $\ell^1$-norm in $\R^d$. Hence, it suffices to show that $\E[\exp(\langle\kappa^{-1}\bm{e}_i,\bm{D}^\ka_t\rangle)]$ is bounded as $t\da 0$ for each $i\in\{1,\ldots,2^d\}$. The L\'evy--Khintchine formula, the elementary inequality 
$e^x-1-x\le e^{c}x^2$ for $x\in[-c,c]$, $c>0$, 
and the Cauchy-Schwarz inequality $|\langle\kappa^{-1}\bm{e}_i,\bm{w}\rangle|\le |\bm{e}_i|=\sqrt{d}$ for all $\bm{w}\in B_{\bm{0}}(\kappa)$ and $i\in\{1,\ldots,2^d\}$
yield
\begin{align*}
\log\E\big[e^{\langle\kappa^{-1}\bm{e}_i,\bm{D}^\ka_t\rangle}\big]
&=t\int_{B_{\bm{0}}(\kappa)\setminus\{\bm{0}\}}\big(e^{\langle\kappa^{-1}\bm{e}_i,\bm{w}\rangle}-1-\langle\kappa^{-1}\bm{e}_i,\bm{w}\rangle\big)\nu_{\bm{X}}(\D\bm{w})\\
&\le t\int_{B_{\bm{0}}(\kappa)\setminus\{\bm{0}\}}e^{\sqrt{d}}\kappa^{-2}\langle\bm{e}_i,\bm{w}\rangle^2\nu_{\bm{X}}(\D\bm{w})
\le t\int_{B_{\bm{0}}(\kappa)\setminus\{\bm{0}\}}de^{\sqrt{d}}\kappa^{-2}|\bm{w}|^2\nu_{\bm{X}}(\D\bm{w})\\
&\le t\int_{B_{\bm{0}}(\kappa)\setminus\{\bm{0}\}}de^{\sqrt{d}}\kappa^{-\beta_+}|\bm{w}|^{\beta_+}\nu_{\bm{X}}(\D\bm{w})
\le de^{\sqrt{d}}I_{\beta_+},
\end{align*}
completing the proof in the case $\beta_+>0$.

\underline{Case  $\beta_+=0$.} 
Note in this case, that the pure-jump component of $\bm{X}$ is compound Poisson. Thus, as in~\eqref{eq:supremum_bound_norm}, we have that $\sup_{s \in [0,t]}|\bm{X}_s| \le |\bm{\gamma}_{\bm{X}}|t + |\bm{\Sigma_X}|\sup_{s\in[0,t]}|\bm{B}_s|+ \sup_{s\in[0,t]}|\wt{\bm{J}}_s|$ for all $t \in [0,1]$, where $\wt{\bm{J}}_t=\bm{X}_t-\bm{\gamma}_{\bm{X}}t - \bm{\Sigma_X}\bm{B}_t$ for all $t \in [0,1]$. The bound on the $p$-moment of the Brownian term follows exactly as in the case of $\beta_+>0$ and is a constant multiple of $t^{p/2}$. From the term $(|\bm{\gamma}_{\bm{X}}|t)^p$, we get a multiple of $t^p$. Note that $\wt{\bm{J}}_s$ is a compound Poisson process with finitely many jumps on $\R^d \setminus\{\bm 0\}$, with $\beta=0$, and hence $\wt{\bm{J}}_t=\sum_{n=1}^{N_t}\wt{\bm{R}}_k$ for some Poisson random variable $N_t$ with mean $t\nu_{\bm{X}}(\R^d \setminus\{\bm 0\})$ and iid random vectors $(\wt{\bm{R}}_n)_{n\in\N}$ independent of $N_t$ with law $\nu_{\bm{X}}(\cdot\cap (\R^d \setminus\{\bm 0\}))/\nu_{\bm{X}}(\R^d \setminus\{\bm 0\})$. For the term $\E[\sup_{s\in[0,t]}|\wt{\bm{J}}_s|^p]$, we now use the same proof as in the case of $\beta_+>0$, until~\eqref{eq:bound_large_jumps_moment_bound}. Hence, we see that
\begin{equation*}
    \E\bigg[\sup_{s\in[0,t]}|\wt{\bm{J}}_s|^p\bigg] \le d^{p/2}t\int_{\R^d \setminus\{\bm 0\}}|\bm{w}|^p\nu_{\bm{X}}(\D\bm{w})    \cdot\sum_{k=1}^{\ceil{p}}\stirII{\ceil{p}}{k}(t\nu_{\bm{X}}(\R^d \setminus\{\bm 0\}))^{k-1}.
\end{equation*} Since $\wt{\bm{J}}$ has finite activity, it follows that $\int_{\R^d \setminus\{\bm 0\}}|\bm{w}|^p\nu_{\bm{X}}(\D\bm{w})<\infty$. Moreover, since the sum in the display above is bounded in $t \in [0,1]$, we get that $\E[\sup_{s\in[0,t]}|\wt{\bm{J}}_s|^p]$ is bounded by a multiple of $t$, concluding the proof of Lemma~\ref{lem:moment_bound}.

\section{Small-time domains of attraction - proof of Theorem~(\nameref{thm:small_time_domain_stable})}
\label{app:domain_of_attraction}

The proof is essentially a consequence of~\cite[Thm~15.14]{MR1876169} and~\cite[Thm~2]{MR3784492}. Recall that $B_{\bm{a}}(r)=\{\bm{x}\in\R^d:|\bm{x}-\bm{a}|<r\}$ denotes the open ball in $\R^d$ with centre $\bm{a}\in\R^d$ and radius $r>0$, by $\Sp^{d-1}$ the unit sphere in $\R^d$ and define $\scrL_{\bm{a}}(r)\coloneqq\{\bm{x}\in\R^d:\langle\bm{a},\bm{x}\rangle\ge r\}$.

Since $\bm{X}$ and $\bm{Z}$ are L\'evy processes, the stated weak convergences are equivalent to $\langle\bm\lambda,\bm{X}_t\rangle/g(t)\cid \langle\bm\lambda,\bm{Z}_1\rangle$ as $t\da 0$ for any $\bm\lambda\in\R^d$ by~\cite[Cor.~15.7]{MR1876169}. By~\cite[Thm~2]{MR3784492}, it follows that, for some $\alpha\in(0,2]$, $g(t)=t^{1/\alpha}G(t)$ for $t>0$ where $G\in\SV_0$ and, moreover, $\langle\bm\lambda,\bm{Z}_1\rangle$ is $\alpha$-stable for all $\bm\lambda\in\R^d$. Thus, $\bm{Z}$ is itself $\alpha$-stable. We then have the following cases. 

If $\alpha=2$, then, by~\cite[Thm~2(i)]{MR3784492}, the weak convergence in the direction $\bm{v}\in\Sp^{d-1}$ is equivalent to
\[
G(t)^{-2} \bigg(|\bm{v}^\tra\bm{\Sigma}_{\bm{X}}|^2
    +\int_{B_{\bm{0}}(g(t))\setminus\{\bm{0}\}}|\langle\bm{v},\bm{x}\rangle|^2\nu_{\bm{X}}(\D\bm{x})\bigg)
\to|\bm{v}^\tra\bm{\Sigma}_{\bm{Z}}|^2,\quad\text{as }t\da 0,
\]
so the weak convergence in $\R^d$ is equivalent to~\eqref{eq:Brownian-limit}, completing the proof in this case.

If $\alpha\in(1,2)$, the weak convergence in the direction $\bm{v}\in\Sp^{d-1}$ is equivalent to~\eqref{eq:jump-stable-limit} by~\cite[Thm~2(iii)]{MR3784492}, completing the proof in this case.

If $\alpha\in(0,1)$, the weak convergence in the direction $\bm{v}\in\Sp^{d-1}$ is equivalent to~\eqref{eq:jump-stable-limit} and $\langle\bm{v},\bm{X}\rangle$ having zero natural drift by~\cite[Thm~2(iii)]{MR3784492}. Since the latter condition is required for all $\bm{v}\in\R^d$, it is equivalent to $\bm{X}$ having zero natural drift $\bm{\gamma_X}=\int_{B_{\bm 0}(1)\setminus\{\bm 0\}}\bm{x}\nu_{\bm{X}}(\D\bm{x})$, completing the proof in this case.

If $\alpha=1$, the weak convergence in the direction $\bm{v}\in\Sp^{d-1}$ may be different depending on the behaviour of the limiting process in this direction. If $\nu_{\bm{Z}}(\scrL_{\bm{v}}(1))=0$ then $\langle\bm{v},\bm{Z}\rangle$ is a linear drift and the weak convergence, by~\cite[Thm~2(ii)]{MR3784492}, is equivalent to the following two limits as $t\da 0$:
\begin{gather*}
\frac{t}{g(t)}\bigg(\langle\bm{v},\bm{\gamma_X}\rangle-\int_{B_{\bm0}(1)\setminus B_{\bm0}(g(t))}\langle\bm{v},\bm{x}\rangle\nu(\D\bm{x})\bigg)
\to\langle\bm{v},\bm{\gamma_Z}\rangle,
\quad\text{and}\\
g(t)\nu_{\bm{X}}(\scrL_{\bm{v}}(g(t)))\bigg(\langle\bm{v},\bm{\gamma_X}\rangle-\int_{B_{\bm0}(1)\setminus B_{\bm0}(g(t))}\langle\bm{v},\bm{x}\rangle\nu(\D\bm{x})\bigg)^{-1}\to 0,\quad\text{whenever }\langle\bm{v},\bm{\gamma_Z}\rangle\ne 0,
\end{gather*}
(where we recall that $t/g(t)=1/G(t)$). By the first limit, the second limit is equivalent to the following limit: $t\nu_{\bm{X}}(\scrL_{\bm{v}}(g(t)))\to 0=\nu_{\bm{Z}}(\scrL_{\bm{v}}(1))$. If instead $\nu_{\bm{Z}}(\scrL_{\bm{v}}(1))>0$, then, by~\cite[Thm~2(ii)]{MR3784492}, the weak limit is equivalent to the following. The process $\langle\bm{v},\bm{X}\rangle$ has zero natural drift whenever it and $\langle\bm{v},\bm{Z}\rangle$ have finite variation and the following limits hold as $t\da 0$:
\begin{gather*}
\frac{\nu_{\bm{Z}}(\scrL_{\bm{v}}(1))}{g(t)\nu_{\bm{X}}(\scrL_{\bm{v}}(g(t)))}\bigg(\langle\bm{v},\bm{\gamma_X}\rangle-\int_{B_{\bm0}(1)\setminus B_{\bm0}(g(t))}\langle\bm{v},\bm{x}\rangle\nu(\D\bm{x})\bigg)\to \langle\bm{v},\bm{\gamma_Z}\rangle,
\quad\text{and}\\
t\nu_{\bm{X}}(\scrL_{\bm{v}}(g(t)))\to \nu_{\bm{Z}}(\scrL_{\bm{v}}(1)).
\end{gather*}
By the second limit, the first limit can be rewritten as the first limit in the display above. Thus, in either case, the conditions are equivalent to those stated in Theorem~(\nameref{thm:small_time_domain_stable}) in the direction $\bm{v}$. Since the directional limits are equivalent to the corresponding limits in $\R^d$, the result follows.\qed

\section{Proof of the inequality in Equation~(\ref{eq:Wq-XY})}\label{app:A}

Recall that the two L\'evy processes $\bm{X}$ and $\bm{Y}$ in $\R^d$ have the L\'evy--It\^o decompositions $\bm{X}_t=\bm{\gamma}_{\bm{X},\kappa} t+\bm{\Sigma_X}\bm{B^X}_t+\bm{D}^{\bm{X},\kappa}_t+\bm{J}^{\bm{X},\kappa}_t$ and $\bm{Y}_t=\bm{\gamma}_{\bm{Y},\kappa} t+\bm{\Sigma_Y}\bm{B^Y}_t+\bm{D}^{\bm{Y},\kappa}_t+\bm{J}^{\bm{Y},\kappa}_t$, see Section~\ref{sec:couplings_general}. Recall that we chose coupling $\bm{B^X}=\bm{B^Y}$,
implying $|(\bm{\Sigma_X}
    -\bm{\Sigma_Y})\bm{B^X}_t|\le |\bm{\Sigma_X}\bm{B^X}_t
    -\bm{\Sigma_Y}\bm{B^Y}_t|\le
    |\bm{\Sigma_X}
    -\bm{\Sigma_Y}|\cdot |\bm{B^X}_t|$
    (where $|\bm{\Sigma_X}
    -\bm{\Sigma_Y}|$ is the Frobenius norm of the matrix $\bm{\Sigma_X}
    -\bm{\Sigma_Y}$). Applying 
Doob's maximal inequality, we obtain  
\begin{equation}\label{eq:Wasserstein-2-brownian_compon}
\begin{split}
\mW_q\big(\bm{\Sigma_X}\bm{B^X},
        \bm{\Sigma_Y}\bm{B^Y}\big)
&\le
\mW_2\big(\bm{\Sigma_X}\bm{B^X},
        \bm{\Sigma_Y}\bm{B^Y}\big)^{q\wedge 1}\\
&\le\E\bigg[\sup_{t\in[0,1]}|\bm{\Sigma_X}\bm{B^X}_t
    -\bm{\Sigma_Y}\bm{B^Y}_t|^2\bigg]^{(q\wedge 1)/2}
\le (2 \sqrt{d} |\bm{\Sigma_X}-\bm{\Sigma_Y}|)^{q\wedge 1}.
\end{split}
\end{equation}
Applying the triangle inequality, we obtain 
\begin{multline*}
\sup_{t \in [0,1]}|\bm{\gamma}_{\bm{X},\kappa} t+\bm{\Sigma_X}\bm{B}^{\bm{X}}_t+\bm{D}^{\bm{X},\kappa}_t+\bm{J}^{\bm{X},\kappa}_t-\bm{\gamma}_{\bm{Y},\kappa} t-\bm{\Sigma_Y}\bm{B}^{\bm{Y}}_t-\bm{D}^{\bm{Y},\kappa}_t-\bm{J}^{\bm{Y},\kappa}_t|\\
\le |\bm{\gamma}_{\bm{X},\kappa} -\bm{\gamma}_{\bm{Y},\kappa} |
+\sup_{t \in [0,1]} |\bm{\Sigma_X}\bm{B}^{\bm{X}}_t-\bm{\Sigma_Y}\bm{B}^{\bm{Y}}_t|
+\sup_{t \in [0,1]} |\bm{D}^{\bm{X},\kappa}_t-\bm{D}^{\bm{Y},\kappa}_t|
+\sup_{t \in [0,1]} |\bm{J}^{\bm{X},\kappa}_t-\bm{J}^{\bm{Y},\kappa}_t|.
\end{multline*} 
For $q \in (0,1]$ (resp. $q\in(1,2]$)
inequality~\eqref{eq:Wq-XY} 
follows by subadditivity $(a+b)^q \le a^q +b^q$ for $a,b\ge0$ (resp. Minkowski's inequality).
\end{appendix}

\end{document}